\documentclass[11pt,a4paper,oneside]{article}

\usepackage[english]{babel}
\usepackage{amsmath,amssymb,amsfonts,amsthm,mathrsfs}
\usepackage{yfonts}
\usepackage{graphicx,color}
\usepackage{epsfig}
\usepackage{multirow}
\usepackage{float}
\usepackage{hyperref}
\usepackage{enumitem}
\usepackage{caption}
\usepackage{subcaption}

\usepackage{pb-diagram}
\usepackage[all]{xy}
\usepackage[latin1]{inputenc}
\usepackage{amsmath,amsthm}
\usepackage[english]{babel}
\usepackage{latexsym}
\usepackage{amsfonts}
\usepackage{amssymb}
\usepackage{euscript}
\usepackage{color}
\usepackage{longtable}
\usepackage{tabularx}
\usepackage{pdfsync}
\usepackage{pstricks,pst-plot,pst-node,epsfig}
\usepackage{hyperref}
\usepackage{pgfplots}
\usepackage{multicol}

\usepackage{pb-diagram}
\usepackage[all]{xy}
\def\a{\alpha}
\def\b{\beta}
\def\t{\theta}
\def\d{\delta}

\def\s{\sigma}
\def\t{\theta}
\def\l{\lambda}

\def\k{\kappa}
\def\o{\omega}

\catcode`\@=11
\def\newremark#1{\@ifnextchar[{\@orem{#1}}{\@nrem{#1}}}
\def\@nrem#1#2{%
    \@ifnextchar[{\@xnrem{#1}{#2}}{\@ynrem{#1}{#2}}}
\def\@xnrem#1#2[#3]{\expandafter\@ifdefinable\csname #1\endcsname
    {\@definecounter{#1}\@addtoreset{#1}{#3}%
        \expandafter\xdef\csname the#1\endcsname{\expandafter\noexpand
            \csname the#3\endcsname \@remcountersep \@remcounter{#1}}%
        \global\@namedef{#1}{\@rem{#1}{#2}}\global\@namedef{end#1}{\@endremark}}}
\def\@ynrem#1#2{\expandafter\@ifdefinable\csname #1\endcsname
    {\@definecounter{#1}%
        \expandafter\xdef\csname the#1\endcsname{\@remcounter{#1}}%
        \global\@namedef{#1}{\@rem{#1}{#2}}\global\@namedef{end#1}{\@endremark}}}
\def\@orem#1[#2]#3{\expandafter\@ifdefinable\csname #1\endcsname
    {\global\@namedef{the#1}{\@nameuse{the#2}}%
        \global\@namedef{#1}{\@rem{#2}{#3}}%
        \global\@namedef{end#1}{\@endremark}}}
\def\@rem#1#2{\refstepcounter
    {#1}\@ifnextchar[{\@yrem{#1}{#2}}{\@xrem{#1}{#2}}}
\def\@xrem#1#2{\@beginremark{#2}{\csname the#1\endcsname}\ignorespaces}
\def\@yrem#1#2[#3]{\@opargbeginremark{#2}{\csname
        the#1\endcsname}{#3}\ignorespaces}
\def\@remcounter#1{\noexpand\arabic{#1}}
\def\@remcountersep{.}
\def\@beginremark#1#2{\rm \trivlist \item[\hskip \labelsep{\bf #1\ #2}]}
\def\@opargbeginremark#1#2#3{\rm \trivlist
    \item[\hskip \labelsep{\bf #1\ #2\ (#3)}]}
\def\@endremark{\endtrivlist}
\catcode`\@=12
\newcommand{\biindice}[3]%
{
    
    \begin{array}[t]{c}
        #1\\
        {\scriptstyle #2}\\
        {\scriptstyle #3}
    \end{array}
    
}
\def\a{\alpha}
\def\b{\beta}
\def\t{\theta}
\def\d{\delta}

\def\s{\sigma}
\def\t{\theta}
\def\l{\lambda}

\def\k{\kappa}
\def\o{\omega}
\def\mc{\mathcal}
\def\mf{\mathfrak}
\def\ua{\uparrow}
\def\da{\downarrow}

\newcommand{\R}{\mathbb R}
\newcommand{\N}{\mathbb N}

\numberwithin{equation}{section}

\theoremstyle{definition}

\theoremstyle{plain}
\newtheorem{theorem}{Theorem}[section]
\newtheorem{proposition}{Proposition}[section]
\newtheorem{lemma}{Lemma}[section]
\newtheorem{corollary}{Corollary}[section]
\newtheorem{remark}{Remark}[section]
\newtheorem{example}{Example}[section]
\def\stretchx{\Bumpeq{\!\!\!\!\!\!\!\!{\longrightarrow}}}



\title{\vskip-2.5cm
{Rich dynamics in planar systems with heterogeneous nonnegative weights}
\thanks{This paper has been written under the auspices of the Ministry of
Science, Technology and Universities of Spain, under Research
Grants PGC2018-097104-B-I00 and PID2021-123343NB-I00, and of the IMI of Complutense
University. } }

\author{
\sc Juli\'an L\'opez-G\' omez
\\
\small Universidad Complutense de Madrid
\\
\small Instituto de Matem\'{a}tica Interdisciplinar (IMI)
\\
\small Departamento de An\'alisis Matem\'atico y Matem\'atica
Aplicada
\\
\small  Plaza de las Ciencias 3, 28040   Madrid, Spain
\\
\small E-mail: {\tt   julian@mat.ucm.es}
\medskip
\\
\sc Eduardo Mu\~{n}oz-Hern\'andez
\\
\small Universidad Complutense de Madrid
\\
\small Instituto de Matem\'{a}tica Interdisciplinar (IMI)
\\
\small Departamento de An\'alisis Matem\'atico y Matem\'atica
Aplicada
\\
\small  Plaza de las Ciencias 3, 28040   Madrid, Spain
\\
\small E-mail: {\tt eduardmu@ucm.es }
\medskip
\\
\sc Fabio Zanolin
\\
\small Universit\`{a} degli Studi di Udine
\\
\small Dipartimento di Scienze Matematiche, Informatiche e Fisiche
\\
\small  Via delle Scienze 2016, 33100 Udine, Italy
\\
\small E-mail: {\tt fabio.zanolin@uniud.it } }

\date{\today}

\numberwithin{equation}{section}

\begin{document}

\maketitle

{

\begin{abstract}
This paper studies the global structure of the set of nodal solutions of a generalized
Sturm--Liouville boundary value problem associated to the quasilinear  equation
$$
-(\phi(u'))'= \lambda u + a(t)g(u), \quad \lambda\in {\mathbb R},
$$
where $a(t)$ is non-negative with some positive humps separated away by
intervals of degeneracy where $a\equiv 0$. When $\phi(s)=s$ this equation includes
a generalized prototype of a classical model going back to Moore and Nehari \cite{MN-1959}, 1959.
This is the first paper where the general case when $\l\in\R$ has been addressed when $a\gneq 0$.
The semilinear case with $a\lneq 0$ has been recently treated by L\'{o}pez-G\'{o}mez and Rabinowitz
\cite{LGR-2015,LGR-2017,LGR-2020}.
\\
\\
{\it 2010 Mathematics Subject Classification:} 34B08, 34B24, 35C15.
\\
\\
{\it Keywords and Phrases}. Nonlinear differential equations, planar systems,
degenerate weights, Moore--Nehari equation, nodal solutions, global bifurcation theory,
a priori bounds, Poincar\'{e} maps.
\end{abstract}

}

%\markright{\today}

\section{Introduction}\label{section-1}

In this paper we investigate the nodal behavior of the solutions of the quasilinear differential equation
\begin{equation}
\label{i.1}
-(\phi(u'))'= \lambda u + a(t)g(u), \quad \lambda\in {\mathbb R},
\end{equation}
where:
\begin{itemize}
\item $\phi:J_1\to J_2$ is an increasing homeomorphism defined
on the open interval $J_1$, with $-\infty\leq \inf J_1 < 0 <\sup
J_1\leq +\infty$, onto the open interval $J_2$, with $-\infty\leq
\inf J_2 < 0 <\sup J_2\leq +\infty$, and such that $\phi(0)=0$, as discussed by Bereanu and Mawhin \cite{BeMa-2007} and Man\'{a}sevich and Mawhin  \cite{MaMa-1998};
\item $g:{\mathbb R}\to {\mathbb R}$ is a continuous function
with $g(0)=0$ and $g(s)s > 0$ for  $s\not=0$;
\item $a: I=[0,L]\to {\mathbb R}^+:=[0,+\infty)$ satisfies $a\in
L^1(I)$.
\end{itemize}
Our main assumption concerning the weight function $a(t)$ is that
the interval $I$ splits into finitely many adjacent
subintervals $I^+_i$, $I^0_i$, such that $a\gneqq 0$ on $I^+_i$ and
$a\equiv 0$ on $I^0_i$. More precisely, we consider
$I^+_i= [t_i,s_i]$ and $I^0_i= (s_i,t_{i+1})$, where
\begin{equation}
\label{i.2}
   0 = t_0 < s_0 < t_1 < \dots < t_i < s_i < \dots < s_{m-1} < t_m < s_m = L.
\end{equation}
Thus, we assume to have $m+1$ intervals where the weight function $a(t)$ is
nontrivial, separated by $m$ intervals where $a\equiv 0$. Our main aim is to prove the existence of multiple solutions to equation \eqref{i.1} with a prescribed, arbitrarily large, number of oscillations
in the intervals $I^+_i$ and a \lq\lq linear\rq\rq\ behavior in the intervals
$I^0_i$. The equation \eqref{i.1} is complemented with either Dirichlet, Neumann, or mixed boundary conditions in $[0,L]$.
In the past decades a great deal of research has been devoted
to the so-called nonlinear problems with sign
indefinite weight, in which the domain of the function $a(t)$ splits into some intervals
where $a>0$ and $a<0,$ respectively (see, for instance, \cite{LGTZ-2014} and the references therein
for an investigation on the multiplicity result in the case of the boundary condition $u(0)=u(L)=M$).
The equations studied in the present paper exhibit some peculiar dynamical features which
look analogous to the sign indefinite case when $\lambda <0$ and are completely different for
$\lambda \geq 0$; in particular, our aim is to
enlighten the role that the intervals of degeneracy of the weight function have in producing
new multiplicity results.
\par
This analysis is motivated by some recent researches in semilinear  boundary
value problems with nonnegative weights that are \lq\lq degenerate\rq\rq \ in some
parts of their domains. The interest in these problems goes back to the pioneering work of Moore and Nehari
\cite{MN-1959}, 1959, concerning the Dirichlet boundary value problem
\begin{equation}
\label{i.3}
\left\{ \begin{array}{l} -u'' = a(t)u^{2n+1} \quad \hbox{in} \;\;[0,L], \\ u(0)=u(L)=0, \end{array}\right.
\end{equation}
where $n\geq 1$ is an integer, $L=\a+1$ for some $\a>1$, $a(t)\equiv 1$ for $t\in [0,1]\cup
[\alpha,\alpha+1]$ and $a(t)\equiv 0$ in $(1,\a)$. Under these assumptions, Moore and Nehari \cite{MN-1959} proved the existence of three positive solutions for some values of $\alpha>1$. It turns out that non-negative weight functions possessing zones of degeneracy in
their domains are interesting also in wide classes of reaction-diffusion equations arising in Population
Dynamics where the coefficients governing the interaction between
the species may vanish in some regions of their domains. The readers are sent to L\'{o}pez-G\'{o}mez and Sabina \cite{LGS-1995}, the authors's papers
\cite{LGMH-2020,LGMHZ-2020,LGMHZ-2021} and the monograph \cite{LG15}, as well as to the lists of references therein, for some paradigmatic examples where the weight functions vanish either in the spatial or the temporal variables.
\par
More recently, the existence of nodal solutions for the symmetric variant of the Moore--Nehari problem
\begin{equation}
\label{i.4}
\left\{ \begin{array}{l} -u'' = a(t)|u|^{p-1}u \quad \hbox{in} \;\;[0,L], \\ u(0)=u(L)=0, \end{array}\right.
\end{equation}
has been accomplished in Gritsans and Sadyrbaev \cite{GS-2015}
for  $p>1$, referred to as the superlinear case here, as well as
in Kajikiya \cite{Ka-2021} for $0<p<1$, usually refereed to as the sublinear case.
In these works, the weight function $a(t)$ is defined on the symmetric interval
$[-1,1]$ and it is a positive constant in a neighborhood of $-1$ and $1$, while $a=0$ in the
central interval. Solutions with an arbitrarily  large number of zeroes
in the intervals where $a>0$ are shown to exist according to the length of the interval $a^{-1}(0)$.
\par
A different line of research had been already introduced  by  Rabinowitz \cite{Rab-71jde,Rab-74mm} in the early seventies, where the nonlinear eigenvalue problem
\begin{equation}
\label{i.5}
\left\{ \begin{array}{l} -u'' = \l u - a(t)|u|^{p-1}u   \quad \hbox{in} \;\;[0,L], \\ u(0)=u(L)=0, \end{array}\right.
\end{equation}
was analyzed adopting the point of view of global bifurcation theory, settled in the foundational
papers of Rabinowitz \cite{Rab-71,Rab-73}.
According to \cite{Rab-71jde,Rab-74mm}, in \eqref{i.5}, $a(t)>0$ for all $t\in [0,L]$, and $p>1$.
The degenerate problem where the interior of $a^{-1}(0)$ is nonempty in \eqref{i.5} has been analyzed
(for the case $p>1$), much more recently, by L\'{o}pez-G\'{o}mez and Rabinowitz \cite{LGR-2015,LGR-2017,LGR-2020}, where it has been established that the structure of the set of nodal solutions can be far more intricate than in the classical
case when $a(t)$ is everywhere positive (see the numerical experiments of Molina-Meyer in \cite{LGMMR-2017}).
\par
The first goal of this paper is adapting the perspective of global bifurcation theory to analyze the structure of the set of non-trivial solutions, positive, negative and nodal, of the semilinear boundary value problem
\begin{equation}
\label{i.6}
\left\{ \begin{array}{l} -u'' = \l u + a(t)|u|^{p-1}u   \quad \hbox{in} \;\;[0,L], \\ u(0)=u(L)=0, \end{array}\right.
\end{equation}
in the general case when $a\in\mc{C}[t_i,s_i]$ satisfies $a(t)>0$ for all $t\in [t_i,s_i]$ and $i\in\{0,...,m\}$, while $a=0$ in $I^0_j$ for all $j\in\{0,...,m\}$. Thus, there are $m+1$ intervals where the weight function $a(t)$ is positive, separated away by $m$ intervals where $a\equiv 0$. In particular, $a(t)$ is far from being constant on the intervals $I_i^+=[t_i,s_i]$, $i\in\{0,...,m\}$. The existence of a priori bounds for the problem \eqref{i.6} when $p>0$, $p\neq 1$, can be derived with the blowing-up techniques of Amann and L\'{o}pez-G\'{o}mez \cite{ALG-98}, going back to Gidas and Spruck \cite{GS-81}. Our main findings can be packaged into the next result, where we are denoting $\s_n:=\left(\frac{n\pi}{L}\right)^2$ for all integer $n\geq 1$. Note that, by the oddity of the nonlinearity of \eqref{i.6} its solutions arise by pairs: $(\l,u)$ and $(\l,-u)$. Moreover, any non-zero solution, $(\l,u)$ with $u\neq 0$, has finitely many interior zeroes, or nodes, in $(0,L)$, and any zero, $\tau\in [0,L]$, must be simple, i.e., $u'(\tau)\neq 0$. In particular, either $u'(0)>0$, or $u'(0)<0$, and, similarly, $u'(L)\neq 0$.

\begin{theorem}
\label{th1.1}
Suppose $p>0$, $p\neq 1$. Then, for every integer $n\geq 1$, the problem \eqref{i.6} has a solution with
$n-1$ interior nodes if, and only if, $\l < \s_n$. In such case, \eqref{i.6} has, at least, two solutions of this type: One with $u'(0)>0$ plus the opposite. Consequently:
\begin{itemize}
\item For every $\l <\s_1$ and any integer $\kappa\geq 1$, \eqref{i.6} possesses, at least, two solutions with $\kappa-1$ zeroes in $(0,L)$.
\item For every pair of integers $\kappa\geq n\geq 1$ and $\l \in [\s_n,\s_{n+1})$, \eqref{i.6} has, at least, two solutions with $\kappa$ zeroes in $(0,L)$.
\end{itemize}
\end{theorem}

Figure \ref{glob-bif} shows the minimal global bifurcation diagram of \eqref{i.6} in both cases: $p>1$ and $p\in (0,1)$. When $p>1$, the nodal solutions with $n-1$ interior nodes bifurcate  from $u=0$ at $\l=\s_n$. When $p\in (0,1)$, they do bifurcate from infinity at $\l=\s_n$. In both cases, these bifurcations are subcritical. It turns out that, in  the special case
when $a\equiv \mu>0$, by using some standard phase-portrait techniques,  it is easily seen that
\eqref{i.6} admits, at most, two nodal solutions with $n-1$ interior nodes for each integer $n\geq 1$.
Naturally, as illustrated by the example of Moore and Nehari \cite{MN-1959}
and the more recent works \cite{GS-2015, Ka-2021}, the number of nodal
solutions with a fixed number of nodes for a particular value of $\l$
can vary with the value of $m$ and the length of the intervals $I^+_i$ and $I^0_i$ of $a(t)$.
Thus, besides the components plotted in Figure \ref{glob-bif}, \eqref{i.6}
might have one, or several, additional components that have not been plotted in
these (minimalist) global bifurcation diagrams.
\par
Although some available global bifurcation theorems for $\mc{C}^1$-Fredholm operators can be invoked to get some global results concerning the existence of non-trivial solutions with a prescribed nodal behavior for the quasilinear boundary value problem
\begin{equation}
\label{i.7}
\left\{ \begin{array}{l} -(\phi(u'))' = \l u + a(t)|u|^{p-1}u   \quad \hbox{in} \;\;[0,L], \\ u(0)=u(L)=0, \end{array}\right.
\end{equation}
it turns out that these results inherit, essentially,  a local character, as in many circumstances the classical solutions of \eqref{i.7} can develop singularities as the solutions separate away from $u=0$ becoming somewhere larger.
Indeed, by making the special, but important,  choice
\begin{equation}
\label{i.8}
  \phi(s):=\frac{s}{\sqrt{1+s^2}},\qquad s\in (-\infty,+\infty),
\end{equation}
with $p>1$, one can apply the global bifurcation theorem of L\'{o}pez-G\'{o}mez and Mora-Corral  \cite{LGMC-2005} and L\'{o}pez-G\'{o}mez \cite{LG-2016} to show the existence of a component of solutions with $n-1$ interior zeroes bifurcating from $u=0$ at $\l=\s_n$ for every integer $n\geq 1$. Alternatively, one can adapt the global bifurcation theorem of L\'{o}pez-G\'{o}mez and Omari \cite{LGO-2019}, which provides us with a global component of bounded variation positive solutions. However, it is folklore that the classical solutions of \eqref{i.7}, with the choice \eqref{i.8}, can develop singularities when they separate away from zero (see, eg., Cano-Casanova et al. \cite{CLGT-2012} and
L\'{o}pez-G\'{o}mez and Omari \cite{LGO-2020}). Thus, this methodology does not seem the most appropriate to get multiplicity results of classical nodal solutions of \eqref{i.7} without some additional, very serious,
effort to show the regularity of its non-zero solutions. Therefore, to accomplish our second goal in this paper, which consists in studying the generalized quasilinear problem \eqref{i.7},
we have preferred here to adopt a different approach
based on planar shooting type methods.
This allows us to study in a rather compact way problems involving the $p$-Laplacian, the mean-curvature operator, or the Minkowski operator (see, e.g.,  L\'{o}pez-G\'{o}mez and Omari \cite{LGO-2022}, Bereanu and Mawhin \cite{BeMa-2007,BeMa-2009}, and the references therein).
\par
Precisely, to study the quasilinear equation \eqref{i.1} we will analyze the dynamical system associated with the underlying system
\begin{equation}
\label{i.9}
\begin{cases}
x'= h(y),\\
y' = -\lambda x - a(t)g(x),
\end{cases}
\end{equation}
where
$$
   h(y):= \phi^{-1}(y).
$$
By the definition of $\phi$, $h$ is a continuous and monotone
increasing function, defined on the open interval
$J_2:=(\varrho_{-},\varrho_{+})$, with $-\infty \leq \varrho_{-} <
0 < \varrho_{+} \leq +\infty$, and such that $h(0)=0$.  Solutions
of \eqref{i.9} are intended in the Carath\'{e}odory sense and are interpreted
as solutions of \eqref{i.1} as well. Namely, $u(\cdot)$ is a solution
\eqref{i.1} if it is a $C^1$-function with $\phi(u')$ absolutely continuous
and \eqref{i.1} is satisfied for almost every  $t\in [0,L]$. In our applications
the weight function $a(t)$ will be chosen as a piecewise constant function. Thus,
the solutions $(x,y)$ of system \eqref{i.9} have  $x$ of class $C^1$, while $y$ is
continuously differentiable except at the points where $a(t)$ is discontinuous.
\par
The analysis of \eqref{i.9} is of huge interest on its own, regardless whether or not it is related to an equation like \eqref{i.1}. Indeed, examples of \eqref{i.9} (with $\lambda=0$) arise, rather naturally,
in dealing with some important classes of Lotka--Volterra predator--prey systems where seasonal interaction effects play an important role (see the recent paper of the authors \cite{LGMHZ-2021b}). Actually, analyzing the dynamics of \eqref{i.9} allows us to deal
with a broad class of boundary value problems of generalized Sturm--Liouville type related to
the quasilinear equation \eqref{i.1}. These problems extend considerably the Dirichlet problem
\eqref{i.7}. Our results can be applied, similarly, to deal with the periodic problem,
by assuming $a: {\mathbb R}\to {\mathbb R}$ a $T$-periodic function
with a certain number of intervals where $a(t)$ is positive separated by intervals of degeneracy in $[0,T].$
\par
The plan of this paper is the following.  In Section 2 we analyze the semilinear problem \eqref{i.6}
in the general case when $a(t)$ is piecewise continuous. In such case, our model can be viewed as a sort of generalized Moore--Nehari prototype.
The main findings of this section have been already packaged in Theorem \ref{th1.1}.
According to it, for every $\l\in\R$, the problem \eqref{i.6} admits infinitely many oscillatory solutions.
We will come back to consider this problem again in Section 7, where it will be used as the simplest prototype model for testing the abstract results developed in Sections 3--6. Naturally, the abstract results will provide us with finitely many nodal solutions under the appropriate circumstances, because, at least when dealing with the mean curvature operator, sufficiently large nodal solutions will not provide us with classical solutions anymore. In some sense, these highly oscillatory solutions are lost when $-D^2$ is inter-exchanged by a quasilinear operator. Adopting this perspective the abstract results of this paper for the dynamical system
\eqref{i.9} are rather satisfactory.
\par
In Section 3 we fix the class of weight functions considered in Sections 3--7. For simplicity and following also previous works in this area, like those of Moore and Nehari \cite{MN-1959},  Kajikiya \cite{Ka-2021} and Cubillos et al. \cite{CLGT-2022}, we confine ourselves to deal with stepwise function coefficients, by taking $a(t)\equiv \mu_i>0$ if $t_i\leq t\leq s_i$ for all
$i\in\{0,...,m\}$. From the method of the proofs, it will become apparent that our results
are stable with respect to sufficiently small perturbations of the coefficients in the $L^1$-norm on $[0,L]$. For such choices  of $a(t)$ it is natural to describe the dynamics associated with the system \eqref{i.9}
as the superposition of those associated to the fully nonlinear system
\begin{equation}
\label{i.10}
x'= h(y), \quad
y' = -\lambda x - \mu_ig(x)
\end{equation}
coupled with the linear one in the second component
\begin{equation}
\label{i.11}
x'= h(y), \quad
y' = -\lambda x.
\end{equation}
For both of these systems we analyze their most significant  arc-deformation properties related to the corresponding Poincar\'{e} maps in Section 4.
Recall that the Poincar\'{e} map associated with system
\eqref{i.9}, on a time-interval $[\tau_0,\tau_1]$, is the application which maps
a point $p_0\in\R^2$ to the point $p_1=(\hat{x}(\tau_1),\hat{y}(\tau_1))$, where
$(\hat{x}(\cdot),\hat{y}(\cdot))$ is the unique solution of \eqref{i.9} with
initial value $p_0$ at the time $\tau_0$. The uniqueness of the solution for the associated Cauchy
problems on the given intervals is implicitly assumed. By the continuous dependence of the solutions
with respect to the initial data, the Poincar\'{e} map is a homeomorphism on its domain.
For our choice of the weight function, leading to a sequence of autonomous systems, each Poincar\'{e}
map on an interval $[t_i,s_i]$ (resp. $[s_i,t_{i+1}]$) is equivalent to the map on $[0,\tau_i]$ with
$\tau_i:=s_i-t_i$ (resp. $[0,\varsigma_i]$ with $\varsigma_i:=t_{i+1}-s_i$). Thus, we can denote by
$\Phi_i$ and $\Psi_i$ the Poincar\'{e} maps associated with
\eqref{i.10} and \eqref{i.11} on the intervals of length $\tau_i$ and $\varsigma_i$, respectively.
Since we study \eqref{i.9} for every $\lambda\in\R$, the system \eqref{i.11} exhibits different
dynamics as $\lambda$ varies. Precisely, the dynamics consists of a shift parallel to the $x$-axis when $\lambda=0$,
like in Moore and Nehari \cite{MN-1959}, a saddle point at the origin
when $\lambda<0$, and a center at the origin when $\lambda>0$.  For this reason,
we have to study  throughout this paper these three cases separately.
We conclude Section \ref{section-2} by analyzing how acts the Poincar\'{e} map $\Psi_i$ according to the sign of $\l$.
\par
As already commented above, Section \ref{sec-non} analyzes the effect of the transformation of certain arcs through the Poincar\'{e} maps $\Phi_i$,  $i\in\{0,\ldots, m\}$.   Similarly, Section 5 focuses attention into the arc-transformation properties of $\Psi_i$ and, hence,  $\Psi_i\circ \Phi_i$. These results are pivotal to get our main abstract results, delivered in Section 6, concerning  the multiplicity of nodal solutions for the generalized Sturm--Liouville problem associated to \eqref{i.1}. Essentially, in Section 6 we have established that, for every $(n_0, n_1, ..., n_m)\in \N^{m+1}$,  the system \eqref{i.9}
possesses a solution $(x,y)$ such that $x(t)$ has $n_i$ nodes in the interval $[t_i,s_i]$ for every $i\in\{0,...,m\}$. Naturally, the number of zeroes of $x(t)$ in the vanishing intervals of $a(t)$, $[s_i, t_i+1]$, $i\in\{1,...,m\}$, where the system is linear in $x$, depends on the sign of $\l$. The most delicate part the abstract analysis in Section 6 is  getting the appropriate twist conditions of the autonomous flow on each of the intervals $[t_i,s_i]$,  as well as their balances with respect to the lengths of the vanishing intervals of $a(t)$, measured by  $\varsigma_i:=t_{i+1}-s_i$, which should be sufficiently large to get a maximal number of nodal solutions. Finally, in Section \ref{section-4}, we illustrate the applicability of
these results to the generalized Moore--Nehari model that we have studied previously in Section 2. This is
a good model for testing the flexibility of the abstract results of Section 6, as it admits infinitely many oscillatory solutions for every $\l\in\R$.
\par
It should be emphasized that incorporating the parameter $\l$ to the previous studies on the Moore--Nehari model adds a number of different phenomenologies to the underlying theory that have not been previously documented. In particular, we have found that, for  $\lambda <0$ and $t_{i+1}-s_i$ sufficiently large, we
have more solutions than a priori expected, presumably by the higher complexity of the Markov-type transition diagram sketched in Figure  \ref{fig-04}, while, for $\lambda >0$, in order to get nodal solutions, one must impose an additional  compatibility condition requiring an upper bound for the lengths of the vanishing intervals of $a(t)$, the $\varsigma_i$'s. Some numerical examples establishing the compatibility of these restrictions will be delivered in Section 7. As a matter of fact, by simply having a glance at the bifurcation diagrams of our semilinear Moore--Nehari model in Section 2, it is easily realized that the larger is $\l>0$ the lower is the complexity of the solution set of \eqref{i.6}. As Kajikiya \cite{Ka-2021,Ka-2022} has used phase-portrait techniques in the special case when $\l=0$,
our results seem to be new even for the simplest semilinear prototype model \eqref{i.6}.

\section{A pivotal paradigmatic semilinear problem}

This section analyzes the structure of the set of solutions of the semilinear problem
\begin{equation}
\label{ii.1}
\begin{cases}
-u''= \lambda u + a(t)|u|^{p-1}u\quad\hbox{in}\;\;(0,L),
\\[3pt]
u(0)=u(L)=0,
\end{cases}
\end{equation}
where $L>0$,  $p>0$ with $p\neq1$, $a\gneq 0$, and $\l\in\mathbb{R}$ is regarded as a bifurcation parameter. When $p=1$, \eqref{ii.1} becomes a linear problem. In such case, if we denote by
$$
  \s(-D^2-a(t))=\Big\{\s_n[-D^2-a(t)]\;:\;n\geq 1\Big\}
$$
the spectrum of $-D^2-a(t)$ in $(0,L)$ under Dirichlet boundary conditions, then, the solution set of
\eqref{ii.1} consists of the trivial solutions, $(\l,0)$ with $\l\in\R$, plus a numerable union of straight lines,
$$
  \bigcup_{n\geq 1}\Big\{\left(\s_n[-D^2-a(t)],u\right)\;:\;u\in \hbox{Ker\,}\Big(-D^2-a(t)-\s_n[-D^2-a(t)]\Big)\Big\},
$$
which has been sketched in Figure \ref{fig-deg}.

\begin{figure}[h!]
    \centering
    \includegraphics[scale=0.65]{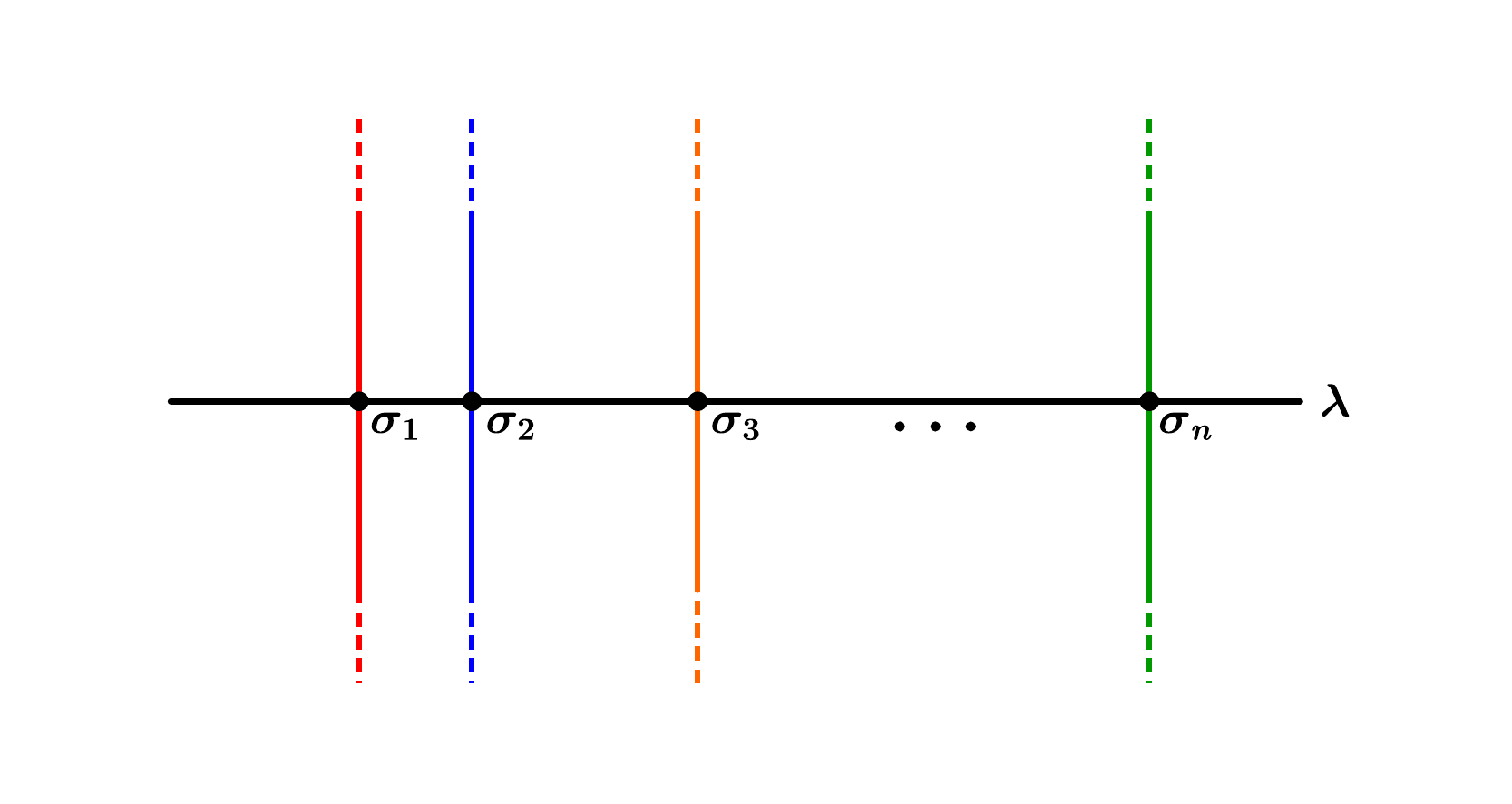}
    \caption{\tiny{\small{Bifurcation diagram for $p=1$,
   where $\s_n\equiv \s_n[-D^2-a(t)]$.}}}
    \label{fig-deg}
\end{figure}

In Figure \ref{fig-deg}, we are plotting the parameter $\l$ in abscissas versus the values
\begin{equation}
\label{ii.2}
M_+:=\max_{t\in[0,L]}|u(t)|, \quad\hbox{or}\quad M_-:=-\max_{t\in[0,L]}|u(t)|,
\end{equation}
in ordinates, according to whether $u'(0)>0$, or $u'(0)<0$, respectively, as in the remaining bifurcation
diagrams of this paper.   Since
$$
  \lim_{n\to \infty}\s_n[-D^2-a(t)]=\infty,
$$
these eigenvalues are positive, except for, at most, finitely many among them.
Note that, in the special case when $a(t)\equiv \mu>0$ is a positive constant, then
$$
  \s_n[-D^2-a(t)]=\s_n[-D^2]-\mu=\left( \frac{n\pi}{L}\right)^2-\mu \quad
  \hbox{for all integer}\;\; n\geq 1.
$$
Throughout this paper, we will simply denote
$$
  \s_n:=\s_n[-D^2]=\left( \frac{n\pi}{L}\right)^2,\qquad n\geq 1.
$$
In Figure \ref{fig-deg}, all the bifurcations from $u=0$ are vertical.
\par
When $p>1$, the nonlinearity has a superlinear growth at zero and at infinity, as in
Moore and Nehari and L\'{o}pez-G\'{o}mez and Rabinowitz \cite{MN-1959,LGR-2020}, while if $p\in (0,1)$,  then it inherits a sublinear growth at zero and at infinity, as in Kajikiya \cite{Ka-2021}.
\par
The first part of this section analyzes, very briefly, the special case when  $a(t)$ is a positive constant, $\mu>0$, for all $t\in[0,L]$. Then, we will study  \eqref{ii.1} for a general class of piece-wise continuous weight functions $a(t)$ that remain bounded away from zero in the intervals where it is positive.
\par
An important feature of the problem \eqref{ii.1} is the fact that, since the nonlinearity is an odd function of $u$, the solutions arise by pairs: $(\l,u)$ and $(\l,-u)$, which facilitates the construction of the set
of solutions of the problem. Another pivotal feature of \eqref{ii.1} is that, for every solution $(\l,u)$ with $u\neq 0$ and any $z\in [0,L]$ such that $u(z)=0$, necessarily $u'(z)\neq 0$. Thus, the zeroes of any non trivial solution are simple and isolated. Therefore, any non-trivial solution of \eqref{ii.1} has, at most, finitely many zeroes in $(0,L)$.

\subsection{The special case when $a\equiv \mu>0$ in $(0,L)$}

In such case,  the planar system associated to the differential equation of \eqref{ii.1} becomes
\begin{equation}
\label{ii.3}
(\mathcal{N}_{\l})\quad\left\{
\begin{array}{ll}
x'&=y,\\
y'&=-\l x-\mu|x|^{p-1}x.
\end{array}
\right.
\end{equation}
When $p\in (0,1)$,  $|x|^{p-1}x$ admits a continuous extension up to $x=0$ and the uniqueness of a solution for the associated Cauchy problem follows from Rebelo \cite[Th. 1]{Re-2000}; the existence being a direct consequence of the Peano's theory. When $p>1$ the nonlinearity is Lipschitz continuous and these results are a direct consequence from the Cauchy--Lipschitz theorem.
\par
Our construction of the set of solutions of \eqref{ii.3} is based on the several phase portraits of the model and hence, it relies on the period function associated to the periodic orbits of \eqref{ii.3} surrounding the origin. By the oddity of $y$ and $|x|^{p-1}x$, the necessary  time to cross each quadrant in the phase-plane is the same. Thus, it is the quarter of the period. Let denote by $\mathcal{H}(x,y)$ the associated Hamiltonian
\[
\mathcal{H}(x,y)=\frac{1}{2}y^2+\frac{\l}{2}x^2+\frac{\mu}{p+1}|x|^{p+1}.
\]
Then, for the appropriate range of $x_+>0$, the periodic orbit of \eqref{ii.3}, $(x(t),y(t))$, $t\in\R$, through the point $(x_+,0)$  satisfies
\begin{equation}
\label{ii.4}
\mc{H}(x(t),y(t))=\frac{1}{2}y^2(t)+\frac{\l}{2}x^2(t)+\frac{\mu}{p+1}|x(t)|^{p+1}
=\frac{\l}{2}x_+^2+\frac{\mu}{p+1}x_+^{p+1}
\end{equation}
for all $t\in\R$. Therefore, after some straightforward manipulations, it is apparent that the period of the orbit trough $(x_+,0)$ can be expressed in the form
\begin{equation}
\label{ii.5}
\mathcal{T}(x_+) \equiv \mc{T}_\l(x_+) := 4\int_0^1\frac{ds}{\sqrt{\l(1-s^2)+\frac{2\mu}{p+1}x_+^{p-1}(1-s^{p+1})}}.
\end{equation}
The precise range of $x_+$'s for which $\mathcal{T}(x_+)$ is well defined,  as well as the precise behavior of the map $\mathcal{T}$ depends on $\l$ and the energy level
$\mc{H}=c$ of the Hamiltonian. Indeed,
when $\l\geq 0$, the origin, which is the unique equilibrium of  \eqref{ii.3}, is a global nonlinear center
and \eqref{ii.5} makes sense for all $x_+>0$. However, when $\l<0$, then  the situation is a bit more subtle, as \eqref{ii.3} possesses two additional equilibria at the points
\begin{equation}
\label{ii.6}
  \o_\pm := \pm \left( \frac{-\lambda}{\mu}\right)^{\frac{1}{p-1}},
\end{equation}
and the global structure of the phase portrait depends on the size of $p>0$. Suppose $p>1$. Then,
the origin is a saddle point whose stable and unstable manifolds consist of two homoclinic connections
surrounding $\o_\pm$, with zero energy level, passing through $(\pm x^*,0)$, where we are setting
\begin{equation}
\label{ii.7}
  x^*:= \left( \frac{-\lambda(p+1)}{2\mu}\right)^{\frac{1}{p-1}}.
\end{equation}
In this case, the integral curves ${\mathcal H}(x,y) = c$ with $c>0$ are periodic orbits surrounding
the origin with period given by \eqref{ii.5} for all $x_+>x^*$, while,
for every $\mc{H}(\o_\pm,0)<c<0$, $\mc{H}^{-1}(c)$
consists of two periodic orbits surrounding each of the equilibria $(\o_\pm,0)$ in the phase portrait.
Instead, when $p\in (0,1)$, the origin is a local center whose periodic orbits fill
the region enclosed by two heteroclinic connections linking the equilibria  $(\o_\pm,0)$ at the energy level
\begin{equation}
\label{ii.8}
 c^* := {\mathcal H}(\o_\pm,0) = \left( \frac{-\lambda}{\mu}\right)^{\frac{2}{p-1}}
\frac{\lambda (p-1)}{2(p+1)}.
\end{equation}
Thus, for every $x_+\in (0,\o_+)$, $\mc{H}(x_+,0)<c^*$ and the solution through $(x_+,0)$ is a periodic orbit surrounding $(0,0)$ whose period is given by \eqref{ii.5}.
\par
Based on \eqref{ii.5} and on some well known properties of planar phase-portraits, the period map $\mc{T}$ behaves like sketched in
Table \ref{tab1} according to the sign of $\l$ and the size of $p$. Table \ref{tab1} provides us with the
monotonicity properties of $\mc{T}$ and its limiting behavior at the ends of its interval of definition.
\vspace{0.2cm}

\begin{table}[h!]
    \begin{center}
        \begin{tabular}{| c | l | l | l |}
%\hline \multicolumn{4}{|c|}{ \hbox{Behavior of the period map } ${\mc T}(x_+)$ } \\ \hline
 \hline  $\l$ &  behavior of $x_+$ & case $p>1$ & case  $0<p<1$ \\ \hline
            \multirow{3}{*}{$\l\geq 0$}
            & $x_+$\;\hbox{increases} & ${\mc T}(x_+)$\;\hbox{decreases} & ${\mc T}(x_+)$\;\hbox{increases}
            \\ \cline{2-4}
            & $x_+\downarrow 0$ & ${\mc T}(x_+)\uparrow  2\pi/\sqrt{\l}$ & ${\mc T}(x_+)\downarrow 0$
            \\ \cline{2-4}
            & $x_+\uparrow +\infty$ & ${\mc T}(x_+)\downarrow 0$ & ${\mc T}(x_+)\uparrow 2\pi/\sqrt{\l}$
            \\ \cline{2-4} \hline \multirow{3}{*}{$\l<0$}
            & $x_+$\;\hbox{increases} & ${\mc T}(x_+)$\;\hbox{decreases} & ${\mc T}(x_+)$\;\hbox{increases}
            \\   \cline{2-4}
            & $x_+\downarrow x^*$ & ${\mc T}(x_+)\uparrow +\infty$ & \\ \cline{2-4}
            & $x_+\downarrow 0$ &  & ${\mc T}(x_+)\downarrow 0$
            \\ \cline{2-4}
            & $x_+\uparrow +\infty$ & ${\mc T}(x_+)\downarrow 0$ & \\ \cline{2-4}
            & $x_+\uparrow \o_+$ &  & ${\mc T}(x_+)\uparrow +\infty$
            \\ \hline
        \end{tabular}
\caption{Behavior of the period map $\mc{T}(x_+)$}
\label{tab1}
\end{center}
\end{table}

By the results collected in Table \ref{tab1},  the period map $\mc{T}$ has the graph plotted in Figure \ref{periods} according to the sign of $\l$ and the size of $p$. In Table \ref{tab1}
and in Figure \ref{periods}, we agree that
$$
  \frac{2\pi}{\sqrt{\l}}\equiv +\infty\qquad \hbox{if}\;\; \l=0.
$$
Based on these properties, the set of solutions of \eqref{ii.1} can be easily constructed in the special case when $a(t)$ is a positive constant $\mu>0$. Subsequently, we use the notation \eqref{ii.2} and take into account that if $(u,u')$ is a periodic solution of \eqref{ii.3} with energy
$$
c=\frac{\l}{2}x_+^2+\frac{\mu}{p+1}x_+^{p+1}, \qquad x_+>0,
$$
then $M_+=x_+\equiv x_+(c)$. The following result holds.

\begin{figure}[h!]
    \centering
    \includegraphics[scale=0.6]{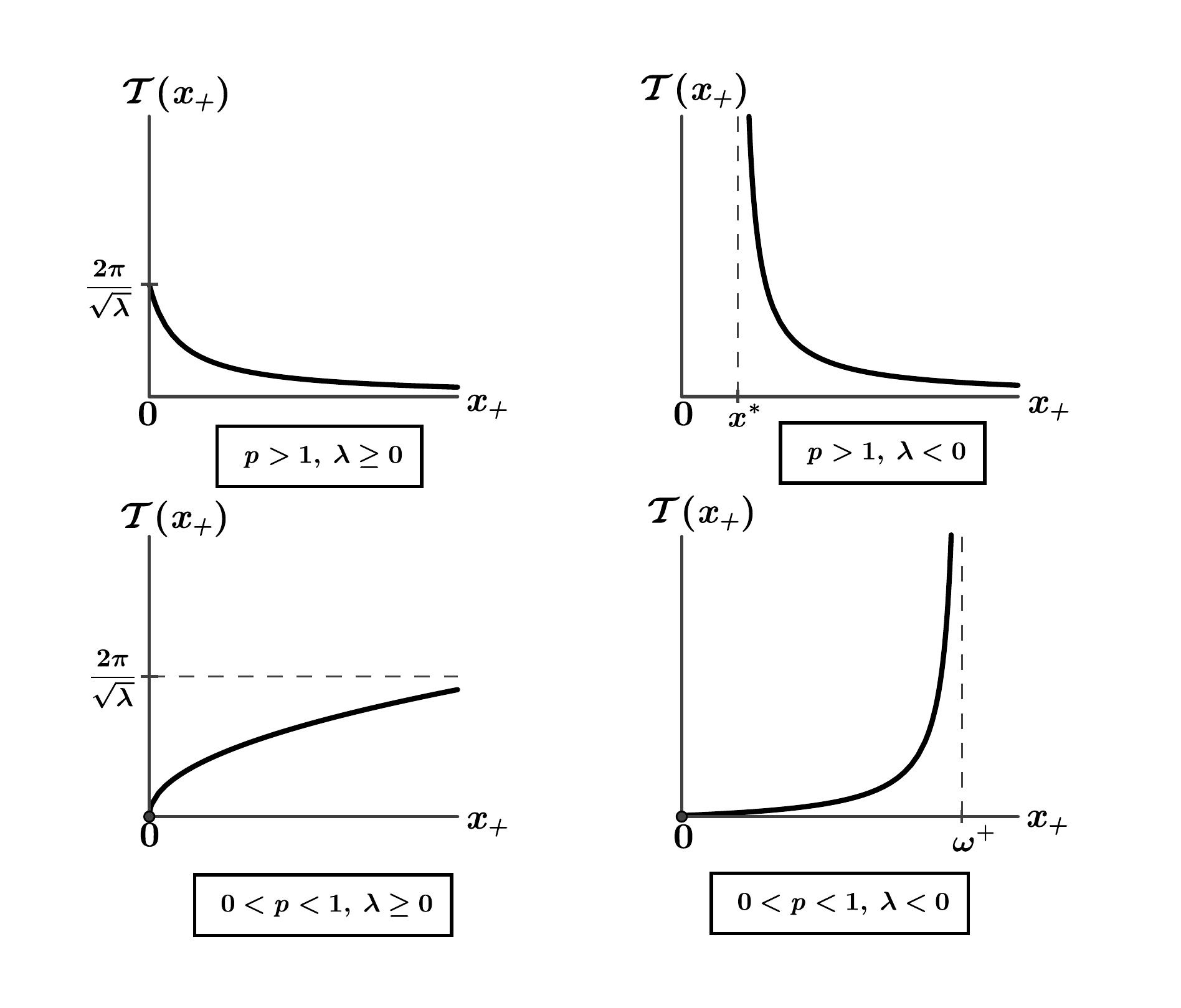}
    \caption{Plots of $\mc{T}$ according to Table \ref{tab1}.}
    \label{periods}
\end{figure}

\begin{proposition}
    \label{pr2.1}
    Suppose $a(t)\equiv\mu>0$ in $[0,L]$. Then, for every $p>0$, with $p\neq1$, and any integer $n\geq 1$, the problem \eqref{ii.1} admits a (unique) solution $(\l,u)$ with $u'(0)>0$ and $n-1$ interior zeros in $(0,L)$ if, and only if,
\begin{equation}
    \label{ii.9}
    \l<\sigma_{n}:=\left(\frac{n\pi}{L}\right)^2.
\end{equation}
In addition, when $n\geq 2$ the nodes are equidistant, i.e., they are located at the points
$$
  t_i:=\frac{i}{n}L,\qquad i\in\{1,...,n-1\}.
$$
Moreover, if, for every $\l<\s_n$, we denote by $u_{n-1,\l}$ the unique solution of \eqref{ii.1} with $n-1$ interior nodes such that $u'(0)>0$, then the map $(-\infty,\s_n)\to \mc{C}[0,L]$ defined by $\l\mapsto u_{n-1,\l}$ is continuous. Thus,
\begin{equation}
\label{ii.10}
  M_+(\l)\equiv M_+(n,\l):=\max_{t\in[0,L]}|u_{n-1,\l}(t)|,\qquad \l<\s_n,
\end{equation}
is also continuous and, actually, independent of $n\geq 1$. Furthermore, it satisfies:
\begin{enumerate}
\item[{\rm (a)}] In case $p>1$, $M_+(\l)$ decreases with respect to $\l<\s_n$, $M_+(\l)>x^*$ if $\l<0$, and
$$
   \lim_{\l \ua \s_n^-} M_+(\l)=0, \qquad \lim_{\l\da -\infty}M_+(\l)=+\infty.
$$
\item[{\rm (b)}] In case $p\in (0,1)$, $M_+(\l)$ increases with respect to $\l<\s_n$, $M_+(\l)<\o_+$ if $\l<0$, and
$$
   \lim_{\l \ua \s_n^-} M_+(\l)=+\infty, \quad \lim_{\l\da -\infty}M_+(\l)=0.
$$
\end{enumerate}
\end{proposition}

\begin{proof}
Note that a solution $(u,u')$ of \eqref{ii.3} passing through $(x_+,0)$ has $n-1$ interior zeroes if, and only if,  $\mc{T}(x_+)=\frac{2L}{n}$. Moreover, thanks to Table \ref{tab1}, we have that
\begin{equation*}
    \left(0,\tfrac{2\pi}{\sqrt{\l}}\right)= \mc{T}((0,+\infty))\;\,\hbox{if}\;\;\l\geq 0,\qquad \left(0,+\infty\right)=     \left\{
    \begin{array}{ll}
    \mc{T}((x^*,+\infty))&\;\,\hbox{if}\;\;\l< 0,\;p>1,
    \\
    \mc{T}((0,\o_+))&\;\,\hbox{if}\;\;\l< 0,\;p\in(0,1).
    \end{array}     \right.
\end{equation*}
Thus, a solution with $n-1$ zeros exists if, and only if, $\frac{2L}{n}<\frac{2\pi}{\sqrt{\l}}$,
which is equivalent to \eqref{ii.9}. The uniqueness follows from the fact that, for any fixed $\l<\sigma_{n}$, the period $\mc{T}(x_+)$ is monotone with respect to $x_+$. The distribution of the zeroes is a direct consequence of the symmetries of the problem, because the solution with $u'(0)>0$
and $n-1$ interior nodes in $(0,L)$ can be easily constructed from the positive solution in $(0,\frac{L}{n})$, by recursive translation and reflection. This explains also why the values of $M_+$ in \eqref{ii.10} are independent of $n\geq 1$.
\par
The continuity of the map $\l\mapsto u_{n-1,\l}$ is a byproduct of the uniqueness, by the continuous dependence of the solutions with respect to the parameter $\l$.
\par
Suppose  $p>1$. Then, thanks to \eqref{ii.5}, for any fixed $\l<\s_n$, $\mc{T}_\l(x_+)$ is decreasing with respect to $x_+$. Moreover, fixing $x_+>x^*$, $\mc{T}_\l(x_+)$ is decreasing with respect to $\l$. Thus, for any given $\l_1<\l_2<\s_n$ and $c_1, c_2>0$,
$\mc{T}_{\l_1}(x_+(c_1))=\mc{T}_{\l_2}(x_+(c_2))$ implies that $x_+(c_1)>x_+(c_2)$. Since $x_+(c)=M_+$, it becomes apparent that $M_+$ is decreasing in $(-\infty,\s_{n})$. Moreover, since
$\lim_{\l\ua \s_n^-}\frac{2\pi}{\sqrt{\l}}=\frac{2L}{n}$, it is easily seen that
$$
   \lim_{\l\ua \s_n^-}M_+(\l)=0,
$$
and that, for every $\l<0$, $0<x^*<M_+(\l)$. This inequality implies, in particular, that $M_{+}(\l)\ua +\infty$ if $\l\da -\infty$, which concludes the proof of Part (a).
\par
Suppose $p\in (0,1)$.  Then, $\mc{T}_\l(x_+)$ is increasing with respect to $x_+$ for each fixed $\l$.
Moreover, for every (fixed) $x_+<\o_+$, $\mc{T}_\l(x_+)$ is decreasing with respect to $\l$. Thus, as soon as $\l_1<\l_2<\s_{n}$ and $c_1,c_2>0$, the identity $\mc{T}_{\l_1}(x_+(c_1))=\mc{T}_{\l_2}(x_+(c_2))$ implies that $x_+(c_1)<x_+(c_2)$. Hence, $M_+(\l)$ is increasing with respect to $\l\in(-\infty,\s_{n})$. Moreover, since $\lim_{\l\ua\s_n^-}\frac{2\pi}{\sqrt{\l}}=\frac{2L}{n}$, we have that
$\lim_{\l\ua\s_n^-}M_+(\l)=+\infty$. Also, for every $\l<0$, it is easily seen that $0<M_+<\o_+$, which implies that $M_{+}(\l)\rightarrow 0^{+}$ if $\l\da -\infty$ and ends the proof of the theorem.
\end{proof}

Figure \ref{fig-bif} shows the global bifurcation diagrams of the non-zero solutions of
\eqref{ii.1} for the special case when $a(t)\equiv \mu>0$ in $[0,L]$. The case $p>1$ has been plotted above, and the case $p\in (0,1)$ below. Precisely, we are plotting the value of the parameter
$\l$, in abscissas, versus the value of $M_+$ if $u'(0)>0$, or $M_-$ if $u'(0)<0$.
As the solutions do arise by pairs, $(\l,\pm u)$, the two bifurcations diagrams are symmetric
with respect to the horizontal axis.
\par
Suppose $p>1$. Then, by Proposition \ref{pr2.1}, for every integer $n\geq 1$, the set of solutions with $n-1$ zeroes in $(0,L)$ consists of a continuous curve bifurcating from $u=0$ subcritically at $\l=\s_n$. Moreover, the upper curve increases up to $+\infty$ as $\l \da -\infty$. Each point on the superior curves represents a single solution, $(\l,u)$, of \eqref{ii.1} with $u'(0)>0$, while the opposite solutions, $(\l,-u)$, fill the  reflected inferior curves of the bifurcation diagram, providing us with all the solutions $(\l,w)$ such that $w'(0)<0$.
\par
Suppose $p\in (0,1)$. Then, by Proposition \ref{pr2.1}, for every integer $n\geq 1$,
the set of solutions with $n-1$ zeroes in $(0,L)$ consists of two continuous curves bifurcating from $\pm\infty$ subcritically at $\l=\s_n$. The upper curve increases from zero, while the inferior one is decreasing. As in the previous case, the solutions of  the form $(\l,u)$ with $u'(0)>0$ fill the superior curves, while the inferior ones consist of the solutions of the form $(\l,-u)$.
\par
In Figure \ref{fig-bif}, the curves of positive solutions have been plotted
by using red color, the solutions with one interior node fill the blue curves, the solutions with two interior nodes fill the orange curves, and the solutions with $n-1$ interior nodes have been plotted
using green color.

\begin{figure}[h!]
    \centering
    \includegraphics[scale=0.6]{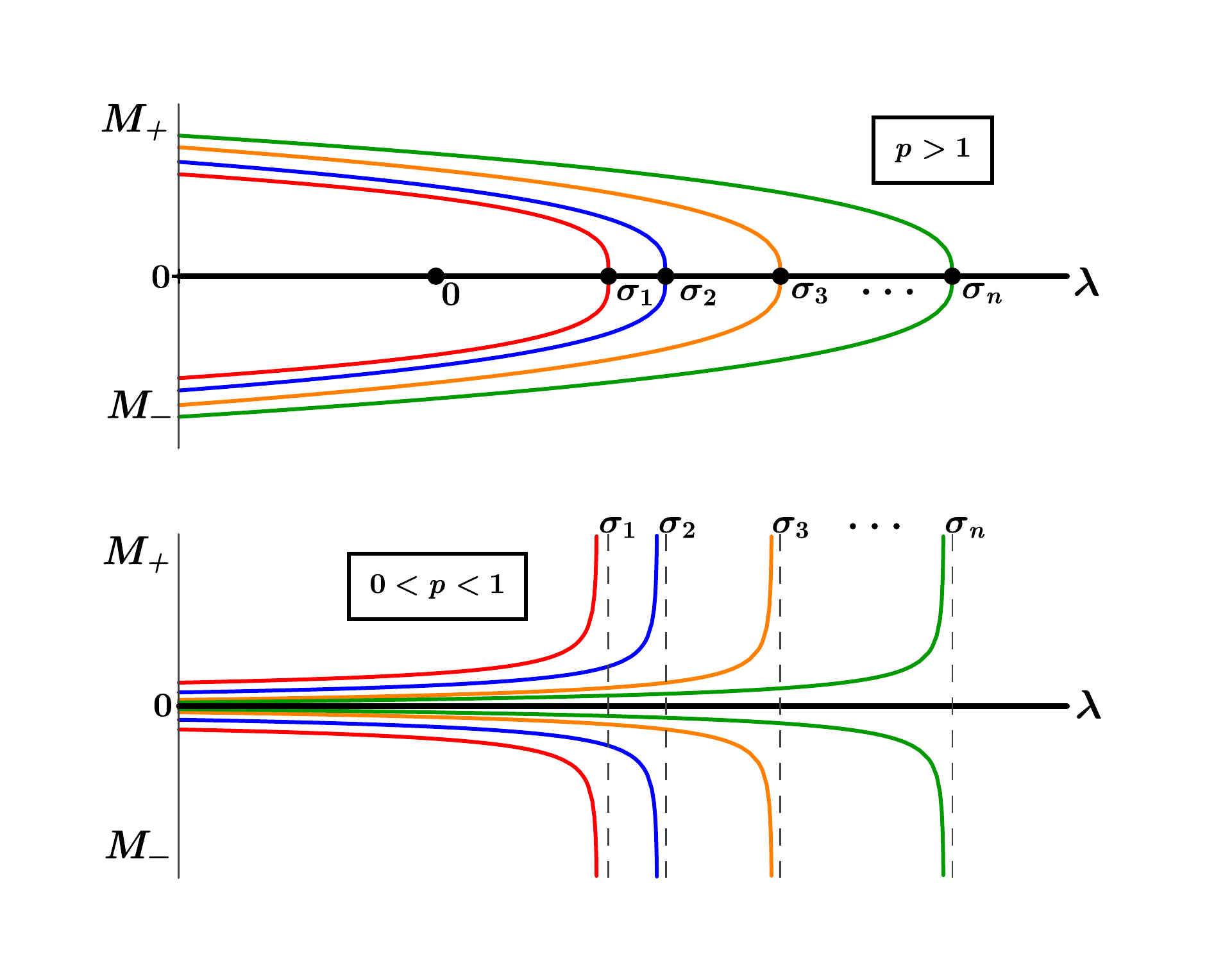}
    \caption{Bifurcation diagrams for $a\equiv\mu>0$}
    \label{fig-bif}
\end{figure}

\par
Since the continuous curves plotted in Figure \ref{fig-bif} consist of solutions with a prescribed number of zeroes, two different curves cannot meet. Actually, even projecting the solutions over $M_{\pm}$, the intersections are excluded by the construction of the nodal solutions from the positive solutions on smaller intervals.
\par
According to Figure \ref{fig-bif}, for every $\l<\s_1$ and any integer $n\geq 1$, the problem \eqref{ii.1}
has a unique solution $(\l,u)$ with $u'(0)>0$ having $n-1$ zeroes in $(0,L)$. Similarly,
for every $\l \in [\s_1,\s_2)$ and any integer $n\geq 2$, \eqref{ii.1} has a unique solution $(\l,u)$ with $u'(0)>0$ having $n-1$ zeroes in $(0,L)$, while it does not admit a positive solution. More generally, for every
integer $\kappa \geq 1$, $\l \in [\s_\kappa,\s_{\kappa+1})$ and $n\geq \kappa+1$, \eqref{ii.1} possesses a unique
solution $(\l,u)$ with $u'(0)>0$ having $n-1$ zeroes in $(0,L)$, while it does not admit any solution
with $\ell$ zeroes in $(0,L)$ if $\ell \leq \kappa-1$.
\par
The bifurcation diagrams of Figure \ref{fig-bif} should be compared with those of Figure 2 in Cubillos et al. \cite{CLGT-2022}, for the special case when $a(t)$ is a negative constant and $p>1$, which are the classical bifurcation diagrams of Rabinowitz \cite{Rab-71jde,Rab-74mm}.

\subsection{A generalized Moore--Nehari model}

In this section we consider a piecewise continuous weight function $a(t)\gneq0$ defined in  $[0,L]$,
which is divided into finitely many disjoint intervals where, alternately, either $a$ is positive (not necessarily constant) or $a\equiv0$. Precisely, we assume that $[0,L]$ splits out into
$$
    I^+_i= [t_i,s_i] \quad \text{and } \; I^0_j= (s_j,t_{j+1}),\quad i\in\{0,....,m\},\;\; j\in\{0,...,m-1\},
$$
where
\begin{equation}
\label{ii.11}
   0 = t_0 < s_0 < t_1 < \dots < t_i < s_i < \dots < s_{m-1} < t_m < s_m = L,
\end{equation}
in such a way that $a\in\mc{C}[t_i,s_i]$ satisfies $a(t)>0$ for all $t\in [t_i,s_i]$
and $i\in\{0,...,m\}$, while $a\equiv 0$ in $I^0_j$ for all $j\in\{0,...,m\}$.
Thus, there are $m+1$ intervals where the weight function $a(t)$ is positive,
separated away by $m$ intervals where $a\equiv 0$. Our class of weight functions extends considerably the one introduced by Moore and Nehari \cite{MN-1959}, later considered by Kajikiya \cite{Ka-2021,Ka-2022}, where $\l=0$ and, for some $\mu\in(0,\frac{L}{2})$,
$$
  a= 1-\chi_{(\frac{L}{2}-\mu,\frac{L}{2}+\mu)},
$$
where $\chi_J$ stands for the characteristic function of the interval $J\subset [0,L]$, i.e.,
$a(t)=0$ if $|t-\frac{L}{2}|<\mu$, while $a(t)=1$ if $|t-\frac{L}{2}|\geq \mu$.
\par
Essentially, the results of this section  establish that the global bifurcation diagrams already found
for the autonomous problem when $a\equiv \mu>0$, collected in Figure \ref{fig-bif}, are \lq \lq preserved\rq\rq\,  for the general class of weight functions considered in this section, though for some values of
$\l\in\R$ and some integer $n\geq 1$, the problem might have an arbitrarily large number of solutions with $n-1$ interior nodes (see Cubillos et al. \cite{CLGT-2022} for some examples in this direction when
$a\lneq 0$). The main technical tool to get these results is bifurcation theory
 with the a priori bounds of Amann and L\'{o}pez-G\'{o}mez \cite{ALG-98}.
As it will become apparent from the proofs in this section, the general bifurcation approach
combined with \cite{ALG-98}, applies to a broader class of weight functions than the ones
in which the intervals of positivity and degeneracy follow the split of \eqref{ii.11}. Indeed, the
location of the intervals of positivity is not important; our only  assumption here is that
$\min a(t)>0$ in each interval where $a(t)>0.$ In any case,
it is convenient to adopt \eqref{ii.11}
in order to have a comparison with the results in the other sections as well as with respect to
other cited articles.
\par
Subsequently, we adapt the analysis of Fencl and L\'{o}pez-G\'{o}mez \cite{FLG-2022} to our setting here. The next result is Lemma 2.1 of \cite{FLG-2022}. As usual, $W^{k,p}_0(0,L)$ denotes the closed subspace of the Sobolev space $W^{k,p}(0,L)$ consisting of the functions $u\in W^{k,p}(0,L)$ such that $u(0)=0=u(L)$. Similarly, $C^k_0[0,L]$ stands for the set of functions of class $\mc{C}^k[0,L]$ vanishing at $0$ and $L$.
\par

\begin{lemma}
\label{le2.1}
For every $f\in L^2(0,L)$, the function
\begin{equation*}
    u(t)\equiv Kf(t) :=\int_0^t(s-t)f(s)\,ds-\frac{t}{L}\int_0^L(s-L)f(s)\,ds
\end{equation*}
provides us with the unique solution of
\begin{equation*}
    \left\{
    \begin{array}{ll}
    -u''=f\quad\hbox{in}\;\,(0,L),\\
    u(0)=u(L)=0,
    \end{array}
    \right.
\end{equation*}
in $W^{2,2}(0,L)$. Moreover, the solution operator $K:L^2(0,L)\rightarrow W^{2,2}_0(0,L)$ is linear and continuous. Thus, when considering the compact embedding $\iota:W^{2,2}_0(0,L)\hookrightarrow C^1_0[0,L]$, the composition
\begin{equation*}
    \mc{K}:=\iota\circ K|_{C^1_0[0,L]}:C^1_0[0,L]\rightarrow C^1_0[0,L]
\end{equation*}
is a linear, continuous and compact operator.
\end{lemma}

Thanks to Lemma \ref{le2.1}, the problem \eqref{ii.1} can be expressed as a fixed point equation for a compact operator. Indeed, a function $u$ piecewise $\mc{C}^2$ in $[0,L]$ solves \eqref{ii.1} if, and only if,
$$
u=\mc{K}(\l u+a(t)|u|^{p-1}u).
$$
Thus,  the solutions of \eqref{ii.1} are the zeroes of the nonlinear operator
\begin{equation}
\label{ii.12}
\begin{array}{rcl}
\mathfrak{F}:\mathbb{R}\times C^1_0[0,L]&\longrightarrow & C^1_0[0,L]\\[5pt]
(\l,u)&\longmapsto & \mathfrak{F}(\l,u):=u-\mc{K}(\l u+a(t)|u|^{p-1}u),
\end{array}
\end{equation}
which can be decomposed in the form
$$
    \mathfrak{F}(\l,u)=\mathfrak{L}(\l)u+\mathfrak{N}(\l,u),
$$
where
\begin{equation}
\label{ii.13}
\mathfrak{L}(\l)u:=u-\l\mathcal{K}u\quad\hbox{and}\quad
\mathfrak{N}(\l,u):=-\mathcal{K}(a(t)|u|^{p-1}u).
\end{equation}
By \eqref{ii.13}, $\mf{L}(\l)$ is a compact perturbation of the identity $I$ in $C^1_0[0,L]$ and, hence,  a Fredholm operator of index $0$.
Moreover,
\begin{equation*}
\mf{N}(\l,u)=o(\|u\|)\quad\hbox{as}\quad
\left\{
\begin{array}{ll}
\|u\|\rightarrow 0&\quad\hbox{if}\;p>1,\\
\|u\|\rightarrow +\infty&\quad\hbox{if}\;0<p<1.
\end{array}
\right.
\end{equation*}
Thus, we can apply the abstract theory of \cite[Ch. 6]{LG01} for $p>1$, going back to
Rabinowitz \cite{Rab-71}, as well as Theorem 1.6 of Rabinowitz \cite{Rab-73} for $0<p<1$.

\par
The generalized spectrum, $\Sigma(\mf{L})$, of the Fredholm curve $\mf{L}(\l)$ is the set of $\l\in\R$ satisfying $u=\l\mc{K}u$ for some $u\neq0$,  $u\in C^1_0[0,L]$. It is folklore that
$$
\Sigma(\mf{L})=\left\{\s_n=\left(\frac{n\pi}{L}\right)^2\,:\,n\geq1\right\}.
$$
As in the general case when $p>1$, the nonlinearity is not of class $\mc{C}^2$, unless $p>2$,
the local theorem of Crandall and Rabinowitz \cite{Rab-71b} cannot be applied to show the existence of a
smooth curve of nodal solutions with $n-1$ interior nodes bifurcating from $u=0$ at $\l=\s_n$. Consequently, to establish the existence of a component of nodal solutions bifurcating from $u=0$ at $\l=\s_n$, we must find out the Leray--Schauder degree $\mathrm{Deg\,}(\mf{L}(\l),B_R)$ for every $\l\in\R\setminus\Sigma(\mf{L}(\l))$, where $B_R$ denotes the ball of radius $R>0$ centered at the origin of $C^1_0[0,L]$. The degree can be computed through the Schauder's formula
\[
{\rm Deg}(\mf{L}(\l),B_R)=(-1)^{m(\mf{L}(\l))},
\]
where $m(\mf{L}(\l))$ is the sum of the algebraic multiplicities of the negative eigenvalues of $\mf{L}(\l)$. According to the Schauder formula, the next result holds.

\begin{lemma}
\label{le2.2}
Setting $\s_0\equiv-\infty$, it turns out that, for every integer $n\geq 0$,
    \begin{equation}
    \label{ii.14}
    \mathrm{Deg}(\mf{L}(\l),B_R)=\left\{
    \begin{array}{ll}
    1&\quad\hbox{if}\;\,\l\in(\s_{2n},\s_{2n+1}),\\
    -1&\quad\hbox{if}\;\,\l\in(\s_{2n+1},\s_{2n+2}).
    \end{array}
    \right.
    \end{equation}
\end{lemma}
\begin{proof}
To determine $m(\mf{L}(\l))$, we should construct the set of $\mu\in\R$ such that, for some $u\in C^1_0[0,L]\setminus\{0\}$,
\begin{equation}
\label{ii.15}
  \mf{L}(\l)u=u-\l\mc{K}u=\mu u.
\end{equation}
Setting $\s_0:=-\infty$, for every $n\geq 1$, $\mathfrak{L}(\l)$ is invertible for $\l\in(\s_{n-1},\s_{n})$. Thus, by the homotopy invariance of the degree,
$$
    d_{n}:=\mathrm{Deg}(\mf{L}(\l),B_R)\;\; \hbox{is constant for all}\;\;\l\in(\s_{n-1},\s_n).
$$
In particular, $\mathrm{Deg}(\mf{L}(\l),B_R)$ is constant in $(-\infty,\s_1)$. Thus, to determine $d_1$ we can assume, without lost of generality, that $\l>0$. Then, by \eqref{ii.15}, $\mu\neq 1$ and hence, the equation \eqref{ii.15} can be equivalently written as
$$
    \left\{ \begin{array}{ll} -u''=\frac{\l}{1-\mu}u & \quad\hbox{in}\;\;[0,L],\\[2pt]
  u(0)=u(L)=0. & \end{array}\right.
$$
Thus, the set of classical eigenvalues of $\mf{L}(\l)$ is
\begin{equation*}
    \sigma(\mf{L}(\l)):=\left\{\mu_n:=1-\frac{\l}{\sigma_n}\;:\;n\geq 1 \right\}.
\end{equation*}
The $\mu_n$'s are algebraic simple eigenvalues of $\mf{L}(\l)$ because, for every $n\geq1$, setting $\psi_n(t):=\sin(\frac{n\pi}{L}t)$, $t\in [0,L]$, we have that
\begin{equation*}
    N[I-\l\mc{K}-\mu_nI]=N[\mf{L}(\s_n)]={\rm span}[\psi_n], \qquad \psi_n\notin R[I-\l\mc{K}-\mu_nI].
\end{equation*}
Indeed, if there exists $u\in C^1_0[0,L]$ such that
\[
    \psi_n=(1-\mu_n)u-\l\mc{K}u\Longleftrightarrow (1-\mu_n)u=\l\mc{K}u+\psi_n\in C^2_0[0,L],
\]
differentiating twice with respect to $t$, it follows from the definition of $\mu_n$ that
\begin{align*}
    -(1-\mu_n)u''=\l u-\psi_n''=\l u+\s_n\psi_n\Longleftrightarrow -u''-\s_nu=\frac{\s_n^2}{\l}\psi_n.
\end{align*}
Thus, multiplying by $\psi_n$ and integrating in $(0,L)$ yields
\[
    0<\frac{\s_n^2}{\l}\int_0^L\psi_n^2=\int_0^L(-u''-\s_nu)\psi_n=0,
\]
which is impossible. Therefore,  $m(\l)$ equals the number of negative eigenvalues of $\mf{L}(\l)$. Consequently, since for every $n\geq 1$ and any given  $\s_n<\l<\s_{n+1}$ we have that
$$
    \mu_1=1-\frac{\l}{\s_1}<\cdots<\mu_n=1-\frac{\l}{\s_n}<0<
    \mu_{n+1}=1-\frac{\l}{\s_{n+1}}<\cdots,
$$
the identity \eqref{ii.14} holds.
\end{proof}

Let  $\mathcal{S}$ denote  the set of non-trivial solutions of \eqref{ii.1}, i.e.,
\begin{equation*}
\mc{S}\equiv \{(\l,u)\in\mf{F}^{-1}(0):u\neq0\}\cup
\left\{
\begin{array}{ll}
\{(\s_n,0)\;:\;n\geq 1\}&\quad\hbox{if}\;\,p>1,\\[5pt]
\{(\s_n,\pm\infty)\;:\;n\geq 1\}&\quad\hbox{if}\;\,0<p<1,
\end{array}
\right.
\end{equation*}
where $\mf{F}$ is given by \eqref{ii.12}. In the case $0<p<1$ we are identifying, from a projective perspective, $(\s_n,+\infty)$ and $(\s_n,-\infty)$ to emphasize that solutions losing their bounds at $\l=\s_n$ belong also to $\mc{S}$. With this convention in mind, the next result holds.
By a component we mean a closed and connected subset maximal for the inclusion.

\begin{theorem}
\label{th2.1}
For every $n\geq1$, $\mc{S}$ possesses a component, $\mf{C}_n$, such that
\begin{equation*}
    \left\{
    \begin{array}{ll}
    (\s_n,0)\in\mf{C}_n&\quad\hbox{if}\;\,p>1,\\[5pt]
    (\s_n,\pm\infty)\in\mf{C}_n&\quad\hbox{if}\;\,0<p<1,
    \end{array}
    \right.
\end{equation*}
consisting of solutions  $(\l,u)$ with  $n-1$ interior zeroes.  Thus,
\begin{equation}
    \label{ii.16}
    \mf{C}_n\cap\mf{C}_m=\emptyset\quad\hbox{if}\;\, n\neq m.
\end{equation}
Moreover, since $\mf{F}$ is odd in $u$, for every $n\geq1$,
\begin{equation*}
    \mf{C}_n=\left\{
    \begin{array}{ll}
    \mf{C}^+_n\cup\mf{C}^-_n\cup\{(\s_n,0)\}&\quad\hbox{if}\;\,p>1,\\[5pt]
    \mf{C}^+_n\cup\mf{C}^-_n\cup\{(\s_n,\pm\infty)\}&\quad\hbox{if}\;\,0<p<1,
    \end{array}
    \right.
\end{equation*}
where
\begin{equation*}
    \mf{C}^\pm_n:=\{(\l,u)\in\mf{C}_n\;:\;\pm u'(0)>0\}.
\end{equation*}
Note that $\mf{C}^-_n:=\{(\l,-u)\;:\;(\l,u)\in\mf{C}_n^+\}$.
\end{theorem}
\begin{proof}
The existence of the component $\mf{C}_n$ in the case $p>1$  follows from Lemma \ref{le2.2} and \cite[Th. .2.1]{LG01}. Once proven the existence of $\mf{C}_n$, by \cite[Le. 6.4.1]{LG01}, we have that, for every $(\l,u)\in \mc{S}\setminus\{(\s_n,0)\}\cap B_{\d}(\s_n,0)$ with sufficiently small $\d>0$,
\begin{equation}
\label{ii.17}
    \l=\s_n+o(1)\;\; \hbox{and}\;\; u(s)=s[\psi_n+o(1)]\;\; \hbox{as}\;\,s\rightarrow0.
\end{equation}
Thus, the solutions of $\mf{C}_n$ in a neighborhood of $(\s_n,0)$ have $n-1$ interior zeroes.
Since the zeroes of the solutions of \eqref{ii.1} are simple, they vary continuously in $\R\times C^1_0[0,L]$. Therefore,  $\mf{C}_n\setminus\{(\s_n,0)\}$ consists of solutions with $n-1$ interior nodes. This implies also \eqref{ii.16}. Finally,  \cite[Pr. 6.4.2]{LG01} guarantees the existence of the subcomponents $\mf{C}_n^{\pm}$.
\par
Similarly, the result for $p\in (0,1)$  is a consequence of Rabinowitz \cite[Le. 1.3]{Rab-73} and the first part of \cite[Th. 1.6]{Rab-73}.
\end{proof}

\begin{remark}
\label{re2.1}
    \rm In the proof of Lemma \ref{le2.2} we have actually shown that the transversality condition of Crandall and Rabinowitz \cite{Rab-71b} holds for $\mf{L}(\s_n)=I-\s_n\mc{K}$ and $\mf{L}_1:=\mf{L}'(\s_n)=-\mc{K}$. Hence, in order to apply \cite[Th. 1.7]{Rab-71b}, it suffices that the operator $\mf{F}$ defined in \eqref{ii.12} be of class $\mc{C}^r$ with $r\geq2$. This holds provided $p>2$. In such case, $\mf{C}_n$ is a curve of class $\mc{C}^{r-1}$ in a neighborhood of $(\s_n,0)$ (see \cite[Sect. 2.2]{LG01} if necessary).
\end{remark}

Next, we will ascertain the global behavior of the components $\mf{C}_n$ for $n\geq1$. Essentially, we will show that $\mf{C}_n$ behaves much like the curve bifurcating from $(\s_n,0)$ in Figure \ref{fig-bif}, though in the general case covered in this section the uniqueness of the nodal solution with $u'(0)>0$ might fail.
Our main result will follow after a series of technical lemmas. The next one shows that \eqref{ii.1} cannot admit a solution $(\l,u)$ with $n-1$ interior zeroes if
$\l\geq \s_n$.

\begin{lemma}
\label{le2.3}
Suppose  $p>1$, or $p\in (0,1)$, and  \eqref{ii.1} has a solution with $n-1$ interior zeroes, $(\l,u)$. Then, $\l<\s_n$.
\end{lemma}
\begin{proof}
Let $(\l,u)$ be a solution with $n-1$ zeroes in $(0,L)$. Then,
$$
  \left\{ \begin{array}{l} \left(-D^2-a(t)|u|^{p-1}\right)u=\l u \quad \hbox{in}\;\; [0,L],\\
  u(0)=u(L)=0, \end{array} \right.
$$
and hence, $u$ is an eigenfunction with $n-1$ zeroes associated with the eigenvalue $\l$ of the
differential operator $-D^2-a(t)|u|^{p-1}$ under homogeneous Dirichlet boundary conditions. Consequently, by the uniqueness of these eigenvalues,
$$
    \l = \s_n\left[ -D^2-a(t)|u|^{p-1}\right]
$$
and, thanks to the monotonicity with respect to the potential (see Butazzo et al. \cite{BGH-98}), we find that
$$
  \l = \s_n\left[ -D^2-a(t)|u|^{p-1}\right]<\s_n[-D^2]=\s_n.
$$
This ends the proof.
\end{proof}

As a byproduct of Lemma \ref{le2.3}, $\mc{P}_\l (\mf{C}_n)\subset (-\infty,\s_1]$ for all $p>0$, where $\mc{P}_\l$ stands for the $\l$-projection operator, $\mc{P}_\l(\l,u):=\l$.
\par
The following lemma establishes that $(\s_n,0)$ is the unique bifurcation point of $\mf{C}_n$ from $u=0$ if $p>1$, and that  $(\s_n,\pm\infty)$ is the unique bifurcation point from infinity of $\mf{C}_n$ if $p\in (0,1)$.  As in \eqref{ii.2} for the constant case, we will denote
$$
   M_{\pm}(\l,u):=\pm\max_{t\in[0,L]}|u(t)|
$$
for every solution, $(\l,u)$, of \eqref{ii.1}.

\begin{lemma}
\label{le2.4}
Let $\mf{C}_\kappa$ be the component of nodal solutions of \eqref{ii.1} with $\kappa-1$ interior nodes
 whose existence is guaranteed by Theorem \ref{th2.1}. Then,
\begin{enumerate}
\item[\rm (i)] Whenever $p>1$, for every sequence of solutions of \eqref{ii.1}, $\{(\l_n,u_n)\}_{n\geq 1}$, such that
$\lim_{n\to \infty}\l_n=\l_*\in\R$, one has that
\begin{equation}
        \label{ii.18}
        \lim_{n\to \infty} M_{+}(\l_n,u_n) = 0
\end{equation}
if, and only if, $\l_*=\sigma_{\kappa}$ for some $\kappa\geq 1$.

\item[{\rm (ii)}] Similarly, when $p\in (0,1)$, for every sequence of solutions of \eqref{ii.1}, $\{(\l_n,u_n)\}_{n\geq 1}$, such that
$\lim_{n\to \infty}\l_n=\l_*\in\R$, one has that
\begin{equation}
\label{ii.19}
        \lim_{n\to\infty} M_{+}(\l_n,u_n)= +\infty
\end{equation}
if, and only if, $\l_*=\s_\kappa$ for some $\kappa\geq 1$.
\end{enumerate}
\end{lemma}
\begin{proof}
Let $\{(\l_n,u_n)\}_{n\geq1}$ be a sequence of solutions in $\mf{C}_k$ such that $\lim_{n\to\infty}\l_n=\l_*\in\R$ and
\begin{equation}
    \label{ii.20}
    \lim_{n\to\infty}M_+(\l_n,u_n) =    \left\{
    \begin{array}{ll}
    0 &\quad\hbox{if}\;\,p>1,\\[5pt]
    +\infty&\quad\hbox{if}\;\,0<p<1.
    \end{array}
    \right.
\end{equation}
Then, dividing by $M_+(\l_n,u_n)$ the differential equation satisfied by $(\l_n,u_n)$ and inverting $-D^2$, we find that, for every $n\geq 1$,
\begin{equation}
    \label{ii.21}
\frac{u_n}{M_+(\l_n,u_n)}=\mc{K}\left(\l_n\frac{u_n}{M_+(\l_n,u_n)}+
a(t)\frac{|u_n|^{p-1}u_n}{M_+(\l_n,u_n)}\right).
\end{equation}
Since $\lim_{n\to \infty}\l_n=\l_*\in\R$, by \eqref{ii.20}, the sequence
$$
   \l_n\frac{u_n}{M_+(\l_n,u_n)}+
a(t)\frac{|u_n|^{p-1}u_n}{M_+(\l_n,u_n)},\qquad n\geq 1,
$$
is bounded. Thus, since $\mc{K}$ is compact, there exists a convergent subsequence
$$
    \lim_{m\to\infty}\frac{u_{n_m}}{M_+(\l_{n_m},u_{n_m})}=\psi
$$
to some function $\psi$ with $\kappa-1$ interior zeroes such that $\|\psi\|_\infty=1$. Moreover,
 letting $m\to \infty$ in
 \begin{equation*}
\frac{u_{n_m}}{M_+(\l_{n_m},u_{n_m})}=\mc{K}\left(\l_{n_m}\frac{u_{n_m}}{M_+(\l_{n_m},u_{n_m})}+
a(t)\frac{|u_{n_m}|^{p-1}u_{n_m}}{M_+(\l_{n_m},u_{n_m})}\right)
\end{equation*}
it becomes apparent that
\begin{equation*}
    \left\{
    \begin{array}{ll}
    &-\psi''=\l_*\psi,\\
    &\psi(0)=0=\psi(L).
    \end{array}
    \right.
\end{equation*}
Therefore,  $\l_*=\s_k=\left(\frac{k\pi}{L}\right)^2$. The converses  follow straight ahead from
Theorem \ref{th2.1}.
\end{proof}

The next lemma provides us with a priori bounds in both cases, $p>1$ and $p\in (0,1)$, ensuring that the solutions of \eqref{ii.1} never blow up if $p>1$, and that they cannot emanate from $u=0$ if $p\in (0,1)$. These a priori bounds follow from Amann and L\'{o}pez-G\'{o}mez \cite[Le. 4.2]{ALG-98} for the case $p>1$, and from a counterpart of that result for the case $p\in (0,1)$.

\begin{lemma}
\label{le2.5}
Let $\mf{C}_\kappa$ be the component of nodal solutions of \eqref{ii.1} with $\kappa-1$ interior nodes    whose existence is guaranteed by Theorem \ref{th2.1}. Then, for every compact subinterval
$J\subset (-\infty,\s_\kappa)$, there exist $C=C(J)>0$ and $\o=\o(J)>0$ such that
\begin{enumerate}
\item[{\rm (i)}] For every $(\l,u)\in\mf{C}_\kappa$ with  $\l\in J$,
$M(\l,u)\leq C$ if $p>1$.

\item[\rm (ii)] For every $(\l,u)\in\mf{C}_\kappa$ with $\l\in J$,
$M(\l,u)\geq \o$ if $p\in(0,1)$.
\end{enumerate}
\end{lemma}
\begin{proof}
Assume that there exists a sequence  $\{(\l_n,u_n)\}_{n\geq1}\subset\mf{C}_\kappa^+$ such that in some interval $(\a,\b)$ where $a(t)>0$ there is a further sequence $\{x_n\}_{n\geq1}\subset(\a,\b)$ such that
\begin{equation*}
    \lim_{n\to+\infty}x_n=x_\infty\in (\a,\b),\qquad
    \lim_{n\to\infty}\l_n=\l_\infty\in(-\infty,\s_\kappa),
\end{equation*}
and
\begin{equation*}
    M(\l_n,u_n)=u_n(x_n):=\max_{t\in[\a,\b]}u_n(t)\overset{n\uparrow\infty}\longrightarrow
    \left\{
    \begin{array}{ll}
    +\infty&\quad\hbox{if}\;\,p>1,\\
    0&\quad\hbox{if}\;\,0<p<1.
    \end{array}
    \right.
\end{equation*}
Since $u_n(x_n)>0$, there exists $r_n>0$ such that
$$
  J_n:=(x_n-r_n,x_n+r_n)\subset (\a,\b)
$$
and $u_n(t)>0$ for all $t\in J_n$.  Then, setting $\varrho_n:=M_n^{-\frac{p-1}{2}}$, $n\geq 1$, it is apparent that, for every $p>0$, $p\neq1$,
\begin{equation}
\label{ii.22}
    \lim_{n\to+\infty}\varrho_n=0.
\end{equation}
Now, for every $t\in J_n$, by performing the change of variables
\begin{equation}
\label{ii.23}
    y:=\frac{t-x_n}{\varrho_n},\qquad\quad v_n(y):=\varrho_n^{\frac{2}{p-1}}u_n(t),\qquad n\geq 1,
\end{equation}
it becomes apparent that, for every $n\geq 1$,
\begin{equation}
\label{ii.24}
    -\ddot{v}_n(y)=\l_n\varrho_n^2v_n(y)+a(x_n+\varrho_n y)|v_n(y)|^{p-1}v_n(y).
\end{equation}
By \eqref{ii.23} and the definitions of $M_n$ and $\varrho_n$, it is easily seen that, for every $n\geq1$,
\begin{equation*}
    |v_n(y)|\leq 1\quad\hbox{and}\quad|v_n(0)|= 1.
\end{equation*}
Thus,  $v_n\not\equiv0$ is bounded for all $n\geq 1$. Consequently, by \eqref{ii.22} and arguing as in \cite[Le. 4.2]{ALG-98}, along some subsequence, relabeled by $n\geq 1$, we have that $v_n$ converges weakly to some positive function $w$ such that
\begin{equation}
\label{ii.25}
    -\ddot{w}(y)=a(x_\infty) |w(y)|^{p-1}w(y),\qquad w(0)=1.
\end{equation}
Since $a(x_\infty)>0$, the origin of \eqref{ii.25} is a center. Thus, all non-trivial solutions have
infinitely many zeroes in $\R$, contradicting the fact that $w$ is a non-trivial positive solution defined in $\R$. Adapting this argument, one can also get a contradiction when $x_\infty\in\{\a,\b\}$. In such case,
\eqref{ii.25} is not defined in $\R$ but only in $(-\infty,0]$, or $[0,\infty)$. This argument works out because we are assuming that $a(t)>0$ for all $t\in[\a,\b]$. The general case when $a(t)$ is continuous
in $[0,L]$ will be treated in a forthcoming paper.
\par
This concludes the proof of the theorem if the maximum of the $u_n$'s is attained in some of the finite intervals of $a^{-1}((0,+\infty))$, which is always the case if $\l_n<0$. Indeed, if $u_n$ reaches its maximum at $x_n\in (0,L)$, then
\begin{equation}
\label{ii.26}
    0\leq -u''_n(x_n)=\l_n u_n(x_n)+a(x_n)|u_n(x_n)|^{p-1}u_n(x_n).
\end{equation}
Thus, $a(x_n)>0$ if $\l_n<0$.
\par
To complete the proof, suppose that  $\l_n\geq 0$ and that the maximum of $u_n$ is attained in some
subinterval of $a^{-1}(0)$. If $0<p<1$ and $\lim_{n\to\infty} M_+(\l_n,u_n)=0$, then on each subinterval of $a^{-1}((0,+\infty))$ we also have that
$$
   \lim_{n\to \infty}\max_{t\in [t_i,s_i]}|u_n(t)|=0
$$
and the result holds very easily by adapting the argument already given before.
\par
Finally, suppose that
$p>1$ and that the $u_n$'s are bounded above in the intervals where $a>0$. Then, along some subsequence, labeled again by $n$, we have that
$$
   \lim_{n\to \infty}\max_{t\in [s_i,t_{i+1}]}|u_n(t)|=\infty
$$
for some $i\in\{0,...,m-1\}$. Thus, since
$$
  -u_n''(t)=\l_n u_n(t),
$$
with $\l_n\geq 0$, it is easily seen that also in a neighborhood of some point $x_n\in \{t_i, s_i:0\leq i\leq m\}$ the solutions go to infinity. According to the first part of the proof, the result holds true.  This ends the proof.
\end{proof}

\begin{remark}\label{rem-2.2}
\rm
In the special case when the weight function $a(t)$ is piecewise constant, the proof of Lemma \ref{le2.5} can be simplified very substantially. Indeed, assume the existence of some sequence $(\l_n,u_n)\in \mf{C}_\kappa^+$, $n\geq 1$, such that
$$
  \lim_{n\to\infty}\l_n = \l_*<\s_\kappa
$$
and
\begin{equation*}
    \lim_{n\to\infty}M_+(\l_n,u_n)= \left\{ \begin{array}{ll}
    +\infty&\quad\hbox{if}\;\,p>1,\\[5pt]
    0&\quad\hbox{if}\;\,0<p<1.  \end{array}
    \right.
\end{equation*}
Let  $[t_i,s_i]$ be any interval of length $\ell=s_i-t_i>0$ where $a$ is a constant $\mu>0$. By  \eqref{ii.5}, the period of the orbit through $(M_+(\l_n,u_n),0)$ in $[t_i,s_i]$ is given, for every $n\geq 1$, by
\begin{equation}
\label{ii.27}
\mathcal{T}(M_+(\l_n,u_n)) = 4\int_0^1\frac{ds}{\sqrt{\l_n(1-s^2)+\frac{2\mu}{p+1}M_+^{p-1}(\l_n,u_n)(1-s^{p+1})}}.
\end{equation}
Since $(\l_n,u_n)\in \mf{C}_\kappa$, for every $n\geq 1$, we have that
$$
\mathcal{T}(M_+(\l_n,u_n))>\frac{2\ell}{\kappa+1}.
$$
Indeed, if $\mathcal{T}(M_+(\l_n,u_n))\leq\frac{2\ell}{\kappa+1}$ for some $n\geq1$, then $(\l_n,u_n)$ must be a solution with more than $\kappa$ nodes and, hence,  $(\lambda_n,u_n)\notin\mathfrak{C}_{\kappa}$. However, letting $n\to \infty$ in \eqref{ii.27} yields to
$$
\lim_{n\to \infty} \mathcal{T}(M_+(\l_n,u_n)) = 0<\frac{2\ell}{\kappa+1},
$$
which is a contradiction.
\end{remark}

\par
According to Lemmas \ref{le2.4} and \ref{le2.5}, for every integer $\kappa\geq 1$, it follows from the global alternatives of Rabinowitz (see \cite[Th. 1.3]{Rab-71} and \cite[Th. 1.6]{Rab-73} for $p>1$ and $0<p<1$, respectively) that the component $\mathfrak{C}_\kappa$ must be unbounded in $\R\times \mc{C}^1_0[0,L]$. Thus, thanks to Lemma \ref{le2.3}, it becomes apparent that
\begin{equation}
\label{ii.28}
\mc{P}_{\l}\left(\mf{C}_\kappa\setminus\{(\s_\kappa,0)\}\right)=(-\infty,\s_\kappa).
\end{equation}
Note that $(\s_\kappa,0)\in\mf{C}_\kappa$ if $p>1$, while $(\s_\kappa,0)\notin\mf{C}_\kappa$ if $p\in (0,1)$. The next results collects these findings in a  compact way.

\begin{theorem}
\label{global}
For every integer $\kappa\geq1$, the component $\mf{C}_\kappa$ of nodal solutions with $\kappa-1$ interior nodes satisfies \eqref{ii.28}. In particular, it is unbounded in $\R\times \mc{C}_0^1[0,L]$. Moreover, $\mf{C}_\kappa$
bifurcates from $(\s_\kappa,0)$ if $p>1$, and from $(\s_\kappa,\pm\infty)$ if $p\in (0,1)$, and it satisfies Lemmas \ref{le2.4} and \ref{le2.5}. Furthermore, for every $(\l,u)\in\mf{C}_\kappa$ with $\l<0$,
we have that
\begin{equation}
\label{ii.29}
    M_{+}(\l,u)\geq \left(\frac{-\lambda}{\|a\|_\infty}\right)^{\frac{1}{p-1}}\quad \hbox{if}\;\;p>1,
\end{equation}
while
\begin{equation}
\label{ii.30}
    M_{+}(\l,u)\leq \left(\frac{-\lambda}{\|a\|_\infty}\right)^{\frac{1}{p-1}}\quad \hbox{if}\;\;p\in (0,1).
\end{equation}
\end{theorem}
\begin{proof}
The first assertions have been already proven. Suppose $\l<0$ and $(\l,u)\in\mf{C}_\kappa^+$ attains its maximum, which is positive, at $x_0\in(0,L)$. Then, since
$$
    0\leq -u''(x_0)=\l u(x_0)+a(x_0)|u(x_0)|^{p-1}u(x_0)=\left( \l +a(x_0)|u(x_0)|^{p-1}\right) u(x_0),
$$
we find that
$$
  \l +a(x_0)|u(x_0)|^{p-1} =\l +a(x_0)M_+^{p-1}(\l,u)\geq 0.
$$
Therefore, $a(x_0)>0$ and
\begin{equation*}
    M_+^{p-1}(\l,u)=u^{p-1}(x_0)\geq \frac{-\lambda}{a(x_0)}.
\end{equation*}
From this estimate, \eqref{ii.29} and \eqref{ii.30} follow readily.
\end{proof}

\begin{figure}[h!]
    \centering
    \includegraphics[scale=0.6]{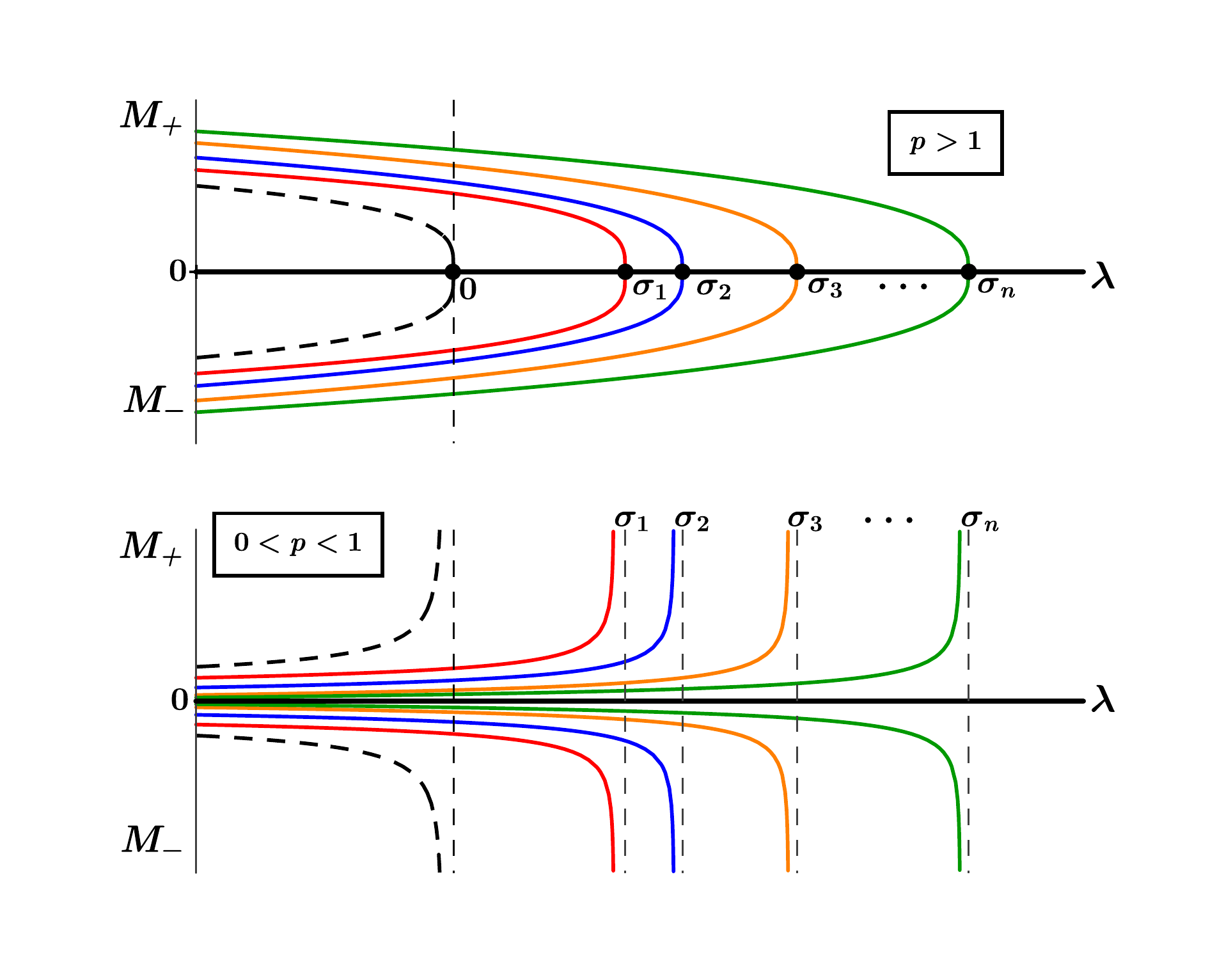}
    \caption{Two minimal global  bifurcation diagrams}
    \label{glob-bif}
\end{figure}

In Figure \ref{glob-bif} we have plotted the \emph{minimal} global bifurcation diagram of nodal solutions
of \eqref{ii.1} for the generalized Moore--Nehari model analyzed in this section. It is minimal in the sense, that, besides the components $\mf{C}_\kappa$, $\kappa\geq 1$, represented in it, the model might
possess some additional components, separated away from the $\mf{C}_\kappa$'s, that
has not been plotted on it, because its existence cannot be guaranteed unless some additional work is done.
Actually, the results of the next sections suggest the existence of
a more complicated bifurcation diagram for a piecewise constant weight satisfying \eqref{ii.11}. Indeed,
at least for $n$ sufficiently large we can provide several solutions with a fixed number of zeros
which are differently distributed in the intervals $I^+_i$ and $I^0_i$.
Moreover, the existence of an arbitrarily large number of turning
points along some of the components $\mf{C}_\kappa$, $\kappa\geq 1$,
cannot be excluded unless $a\equiv \mu>0$ in $[0,L]$.
Indeed, for the classical choice of Moore and Nehari \cite{MN-1959}, there are some negative ranges of the parameters where \eqref{ii.1} has three positive solutions. And,  at least when $a(t)<0$ for all $t\in [0,L]$, one can construct examples of weight functions, as close as wanted to some constant, for which \eqref{ii.1} has an arbitrarily large number of solutions with one interior node, distributed in finitely many isolas (see Cubillos et al. \cite{CLGT-2022}). Whether or not this multiplicity result holds also when $a\gneq 0$ remains an open problem yet.
\par
More precisely, in Figure \ref{glob-bif}, for any given solution $(\l,u)$,  we are plotting
$M_+(\l,u)$ versus $\l$ if $u'(0)>0$, while $\l$ is plotted versus $M_-(\l,u)$ if $u'(0)<0$.
Naturally, the \lq\lq fictitious\rq\rq \, intersections between the components $\mf{C}_\kappa$'s cannot be excluded to occur, neither the existence of turning points along the components $\mf{C}_\kappa$'s. However, as far as concerns the minimal structure of those global bifurcation diagrams, Figure \ref{glob-bif} preserves the main topological features of the global bifurcation diagrams already plotted in Figure \ref{fig-bif} for the special case when $a(t)\equiv \mu>0$ is a positive constant.  The dashed black curves in Figure \ref{glob-bif} represent the limiting curves of the estimates \eqref{ii.29} and \eqref{ii.30}.

\section{Some preliminaries for the quasilinear prototype}\label{section-2}

Throughout the rest of this paper we will focus attention into the quasilinear equation
\begin{equation}
\label{iii.1}
-(\phi(u'))'= \lambda u + a(t)g(u), \quad \lambda\in {\mathbb R},
\end{equation}
already introduced in Section 1, where $a(t)$ is a piece-wise constant
function in the class of \eqref{i.2}. Thus,
\begin{equation}
\label{iii.2}
a(t)\equiv \mu_i > 0, \;\; \forall \, t\in [t_i,s_i],\quad i\in\{0,...,m\}.
\end{equation}
So, $a(t)$ is a stepwise function with constant values $\mu_i>0$ on the $I^+_i$-intervals.
In the special case when $\phi(s)=s$ for all $s\in\R$, \eqref{iii.1}
is a semilinear equation of the form \eqref{ii.1}. Thus, \eqref{iii.1} can be regarded as a quasilinear generalized prototype of the generalized Moore--Nehari equation \eqref{ii.1}.
\par
Our study of \eqref{iii.1} in this paper will be performed by analyzing the dynamical system associated with the associated first order system
\begin{equation}
\label{iii.3}
\begin{cases}
x'= h(y),\\
y' = -\lambda x - a(t)g(x),
\end{cases}
\end{equation}
where $h(y):= \phi^{-1}(y)$, whose general properties were fixed in Section 1.
Under assumption \eqref{iii.2}, the system \eqref{iii.3} can be seen
as a \textit{switched system}, i.e., a system made of a superposition of
a fully nonlinear system,
$$
\begin{cases}
x'= h(y),\\
y' = - f_{i,\lambda}(x),
\end{cases}
\leqno{({\mathcal N}_{i,\lambda})}
$$
where, for every $i\in \{0,...,m\}$,
\begin{equation}
\label{iii.5}
  f_{i,\lambda}(x):= \lambda x + \mu_i g(x),\quad \lambda, x \in \R,
\end{equation}
together with a linear system in the $x$-variable
$$
\begin{cases}
x'= h(y),\\
y' = - \lambda x.
\end{cases}
\leqno{({\mathcal L}_{\lambda})}
$$
Since both $({\mathcal N}_{i,\lambda})$  and $({\mathcal L}_{\lambda})$
are autonomous systems, we can equivalently assume that
$({\mathcal N}_{i,\lambda})$ acts on a time-interval $[0,\tau_i]$
with $\tau_i= s_i-t_i$ and its effect is followed by that of the
system $({\mathcal L}_{\lambda})$ acting on a time-interval $[0,\varsigma_i]$
with $\varsigma_i= t_{i+1}-s_i$. In this framework, if we denote by $\Phi_i$
the Poincar\'{e} map associated with the system $({\mathcal N}_{i,\lambda})$
on the interval $[0,\tau_i]$ and by $\Psi_i$ that one associated with system $({\mathcal L}_{\lambda})$
on the interval $[0,\varsigma_i]$, we have that the Poincar\'{e} map of \eqref{iii.3} in $[0,L]$
splits as
$$
   {\mathscr P}_L:=\Phi_m\circ\Psi_{m-1}\circ\Phi_{m-1}\circ\dots\circ \Psi_{0}\circ \Phi_{0}.
$$
To make sure that the Poincar\'{e} map is well defined, we impose the uniqueness
of the solutions for the initial value problems associated with $({\mathcal N}_{i,\lambda})$
and $({\mathcal L}_{\lambda})$.  Since we are dealing with planar Hamiltonian systems, the
uniqueness holds under very general conditions of $h$ and $g$.  If some extra assumptions would be required, we will highlight them when needed. Actually, we can directly apply a uniqueness result from Rebelo \cite{Re-2000} if $\l>0$, and the same result applies when $\lambda=0$ for the system $({\mathcal N}_{i,\lambda})$,
while for the Cauchy problem associated to $({\mathcal L}_{\lambda})$ the uniqueness holds from an elementary direct argument. The situation might be more subtle  when $\lambda < 0$, where we need to prevent the
possibility that nontrivial solutions reach the origin in finite time. For instance, this could
happen for the system $({\mathcal L}_{\lambda})$ for $\lambda = -1$ and $h(y)=|y|^q {\rm sgn}(y)$
with $0 <q <1$.  In the case $\lambda <0$ the uniqueness can be inferred via the
strong maximum principle for $\phi$-Laplacian operators (see Pucci and Serrin  \cite[Th. 1.1.1]{PuSe-2007}).
\par
In this section, we are going to describe some general properties of the systems
$({\mathcal N}_{i,\lambda})$ and $({\mathcal L}_{\lambda})$. Throughout the rest of this paper,
we will consider the primitives
$$
   H(y):=\int_{0}^y h(s)\,ds, \qquad
F_{i,\lambda}(x):=\int_{0}^x f_{i,\lambda}(s)\,ds= \frac{\lambda}{2} x^2 + \mu_i G(x),
$$
where
$$
  G(x):=\int_{0}^x g(s)\,ds,
$$
and we define
$$
   H(\varrho_\pm)_=\lim_{y\to\varrho_\pm} H(y),
$$
where $(\varrho_{-},\varrho_{+})$ is the interval of definition of $h$.
Note that $H(\varrho_{\pm})=+\infty$ if $\varrho_{\pm}=\pm\infty$.

\subsection{A first glance at the system $({\mathcal N}_{i,\lambda})$}
\label{subsec-non}

Observe that $({\mathcal N}_{i,\lambda})$ is an Hamiltonian system with
$$
   {\mathcal H}(x,y)={\mathcal H}_{i,\lambda}(x,y):=
H(y) + F_{i,\lambda}(x),
\quad \text{for }\; (x,y)\in {\mathbb R}\times (\varrho_{-},\varrho_{+}).
$$
Hamiltonian functions \lq\lq with separable variables\rq\rq\,  of this form
have been already studied in Schaaf \cite{SC-1985,SC-1990}, Cima et al. \cite{CGM-2000}, and, more recently, in Villadelprat and Zhang \cite{VZ-2020}, where the reader is sent for some recent studies in this area.
Given $c\in\mathbb{R}$ such that
\begin{equation}
\label{iii.5}
0 <c < H^*:=\min\{H(\varrho_{-}), H(\varrho_{+})\},
\end{equation}
assume that there are constants $x_{\pm}(c)$,  with $x_-(c) < 0 < x_+(c)$, such that
\begin{equation}
\label{iii.6}
\begin{split}
& F_{i,\lambda}(x_-(c)) = F_{i,\lambda}(x_+(c)) =c,\\
& F_{i,\lambda}(x) < c \;\; \hbox{for all}\;\; x\in (x_-(c),x_+(c)),\\
& F'_{i,\lambda}(x_-(c))< 0 <F'_{i,\lambda}(x_+(c))
\end{split}
\end{equation}
Thus, the level line ${\mathcal H}(x,y) = c$ is a simple closed curve around the origin
intersecting  the $x$-axis at the points $(x_{\pm}(c),0)$ and the $y$-axis at the points $(0,y_{\pm}(c))$,
 where $y_-(c) < 0 < y_+(c)$ and $H(y_{\pm}(c))= c$.  Since $\mathcal{H}^{-1}(c)$ does not contain
equilibrium points of the system $({\mathcal N}_{i,\lambda})$,  it is a periodic orbit running clockwise.
The next result shows that the existence of a $c\in(0,H^*)$ for which \eqref{iii.6} holds is ensured by a more general condition.

\begin{lemma}
\label{le3.1}
Assume that $F_{i,\lambda}(x^{\pm})>0$ for some $x^-<0<x^+$. Then, there exists an interval $[c_1,c_2]\subset (0,H^*)$ such that, for every  $c\in[c_1,c_2]$, there are $x_{\pm}(c)$ satisfying \eqref{iii.6}.
\end{lemma}
\begin{proof}
Since $F_{i,\lambda}(0)=0$ and $F_{i,\lambda}$ is continuous, for every $c\in(0,\min\{H^*,F_{i,\l}(x^{\pm})\})$, there exist $x_{\pm}(c)$ satisfying the first two relations of \eqref{iii.6}. Moreover, since $F_{i,\lambda}$ is of class $C^1$, by applying Sard's theorem, the existence of a regular value $\tilde{c}$ of $F_{i,\lambda}$ in $(0,\min\{H^*,F_{i,\l}(x^{\pm})\})$ holds. Thus, there exist $x_{\pm}(\tilde{c})$ satisfying \eqref{iii.6}. Finally, since $F_{i,\lambda}$ is $C^1$, there exists a closed neighborhood $[c_1,c_2]$ of $\tilde{c}$ where \eqref{iii.6} holds.
\end{proof}

To determine the period of the orbit $\mathcal{H}^{-1}(c)$,  we can sum up the
times needed to cross each quadrant. First, from the relation
$H(y) + F_{i,\lambda}(x)= c$,  since $x'=h(y)$, we find  that
\begin{equation*}
x'=\left\{
\begin{array}{ll}
(h\circ H_{+}^{-1})(c-F_{i,\lambda}(x))\;\;&\text{if}\; x>0,\\[3pt]
(h\circ H_{-}^{-1})(c-F_{i,\lambda}(x))\;\;&\text{if}\; x<0,
\end{array}
\right.
\end{equation*}
where $H_{+}^{-1}$ and $H_{-}^{-1}$ are the inverse functions of $H$
restricted to $[0,\varrho_{+})$ and $(\varrho_{-},0]$, respectively. From this, we
obtain that the necessary time to move from $(0,y_+(c))$ to $(x_+(c),0)$ in the first quadrant
is given by
$$
{\mathcal T}_{i,I}(c):=\int_{0}^{x_+(c)} \frac{dx}{(h\circ H_{+}^{-1})(c-F_{i,\lambda}(x))}.
$$
Similarly, the times to cross the fourth, third and second quadrants, can be expressed as
$$
  {\mathcal T}_{i,IV}(c):=\int_{0}^{x_+(c)} \frac{dx}{-(h\circ H_{-}^{-1})(c-F_{i,\lambda}(x))},
$$
$$
   {\mathcal T}_{i,III}(c):=\int_{x_-(c)}^{0}\frac{dx}{-(h\circ H_{-}^{-1})(c-F_{i,\lambda}(x))}
$$
and
$$
  {\mathcal T}_{i,II}(c):=\int_{x_-(c)}^{0}  \frac{dx}{(h\circ H_{+}^{-1})(c-F_{i,\lambda}(x))},
$$
respectively. Obviously, the period of the closed orbit
corresponding to the level line ${\mathcal H} =c$ is given by
$$
   {\mathcal T}_{i}(c)= {\mathcal T}_{i,I}(c)+ {\mathcal T}_{i,IV}(c) + {\mathcal T}_{i,III}(c)
+{\mathcal T}_{i,II}(c).
$$
Observe that
\begin{align*}
& {\mathcal T}_{i,I}(c) = {\mathcal T}_{i,IV}(c)\;\; \text{and }\;\;
{\mathcal T}_{i,II}(c) = {\mathcal T}_{i,III}(c)\quad \text{if $h$ is odd},\\
& {\mathcal T}_{i,I}(c) = {\mathcal T}_{i,II}(c)\;\; \text{and }\;\;
{\mathcal T}_{i,III}(c) = {\mathcal T}_{i,IV}(c)\quad \text{if $g$ is odd},
\end{align*}
so that all the quarter-laps times coincide with ${\mathcal T}_{i,I}(c)$
if both $h$ and $g$ are odd.
The formulas for the times ${\mathcal T}_{i,j}(c)$ have been already given in
\cite[Pr. 2.4]{CGM-2000}.
\par
Observe that in the case $\lambda\geq0$ there is always an interval $[c_1,c_2]\subset(0,H^*)$ for which
\eqref{iii.6} holds on $x_{\pm}(c)$ for every $c\in[c_1,c_2]$. The case $\lambda<0$ is more delicate. First, as the existence of such an interval is not a priori ensured,  the existence of $x^\pm$ in the statement of Lemma \ref{le3.1} should be imposed when working in the interval $[t_i,s_i]$ with the problem $(\mathcal{N}_{i,\lambda})$. Second, the geometric structure of the orbits can change completely depending on the sublinear/superlinear behavior of $g(x)$ near zero and/or infinity, as one can see, by analyzing the function
$$
   f_{i,\lambda}(x) = \lambda x +\mu_i |x|^{p-1}x
$$
for $p>1$ or $0 < p <1$.

\subsection{A first glance at the system $({\mathcal L}_{\lambda})$}

In this section, we evaluate the necessary time to move along the trajectories
of the $x$-linear system $({\mathcal L}_{\lambda})$.
The trajectories lie on the level lines of the Hamiltonian
function
$$
{\mathcal E}(x,y)={\mathcal E}_{\lambda}(x,y):= H(y) + \frac{\lambda}{2} x^2,
\quad \text{for }\; (x,y)\in {\mathbb R}\times (\varrho_{-},\varrho_{+}).
$$
The  dynamics changes substantially according to the sign of the parameter $\lambda$. Thus,
we will consider three different cases.

\begin{itemize}
\item{} \textit{The case  $\lambda < 0$.}
\end{itemize}

\noindent Then, the origin is a saddle point whose stable and unstable
manifolds lie on the level line of zero-energy, given by $2H(y) = -\lambda x^2$. The level lines
with positive energy correspond to trajectories in the upper half-plane
moving from the second towards the first quadrant, and to trajectories in the lower half-plane
moving from the fourth towards the third quadrant. Let us consider now the energy level lines passing through the points $(0,y_+)$ and $(0,y_-)$ with $\varrho_{-}<y_-<0< y_+ <\varrho_{+}$,
namely
$$
   H(y) = H(y_{\pm}) - \frac{\lambda}{2} x^2.
$$
Then, either
$$
   x' = (h\circ H_{+}^{-1})\left(H(y_+) - \frac{\lambda}{2} x^2\right),\quad\hbox{or}
   \quad x' = (h\circ H_{-}^{-1})\left(H(y_-) - \frac{\lambda}{2} x^2\right).
$$
Thus, for any given $\ell<r$, the transition times from the straight line
$x=\ell$ to the  straight line $x=r$, with $y>0$, and from the  straight line
$x=r$ to the  straight line $x=\ell$, with $y<0$,  are given by
\begin{equation}
\label{iii.7}
{\mathcal S}^{\lambda}(\ell,r;(0,y_+)):=\int_{\ell}^{r}
\frac{dx}{(h\circ H_{+}^{-1})\left(H(y_+) - \frac{\lambda}{2} x^2\right)}
\end{equation}
and
\begin{equation}
\label{iii.8}
{\mathcal S}^{\lambda}(r,\ell;(0,y_-)):=\int_{\ell}^{r}
\frac{dx}{-(h\circ H_{-}^{-1})\left(H(y_-)- \frac{\lambda}{2} x^2\right)},
\end{equation}
respectively. In order to have the arc of trajectory well defined between these two straight lines,
it is required that
$$
  \max\left\{H(y_{\pm}) - \frac{\lambda}{2} \ell^2,H(y_{\pm}) - \frac{\lambda}{2} r^2\right\} < H(\varrho_{\pm}).
$$
The level lines with negative energy correspond to trajectories on the right half-plane
moving from the fourth towards the first quadrant, and to trajectories on the left half-plane
moving from the second towards  the third quadrant.
Figure \ref{fig-02b} shows the deformation of some regions under the flow associated
with $({\mathcal L}_{\lambda})$ for $\lambda <0$.

\begin{figure}[htbp]
\centering
\includegraphics[scale=0.3]{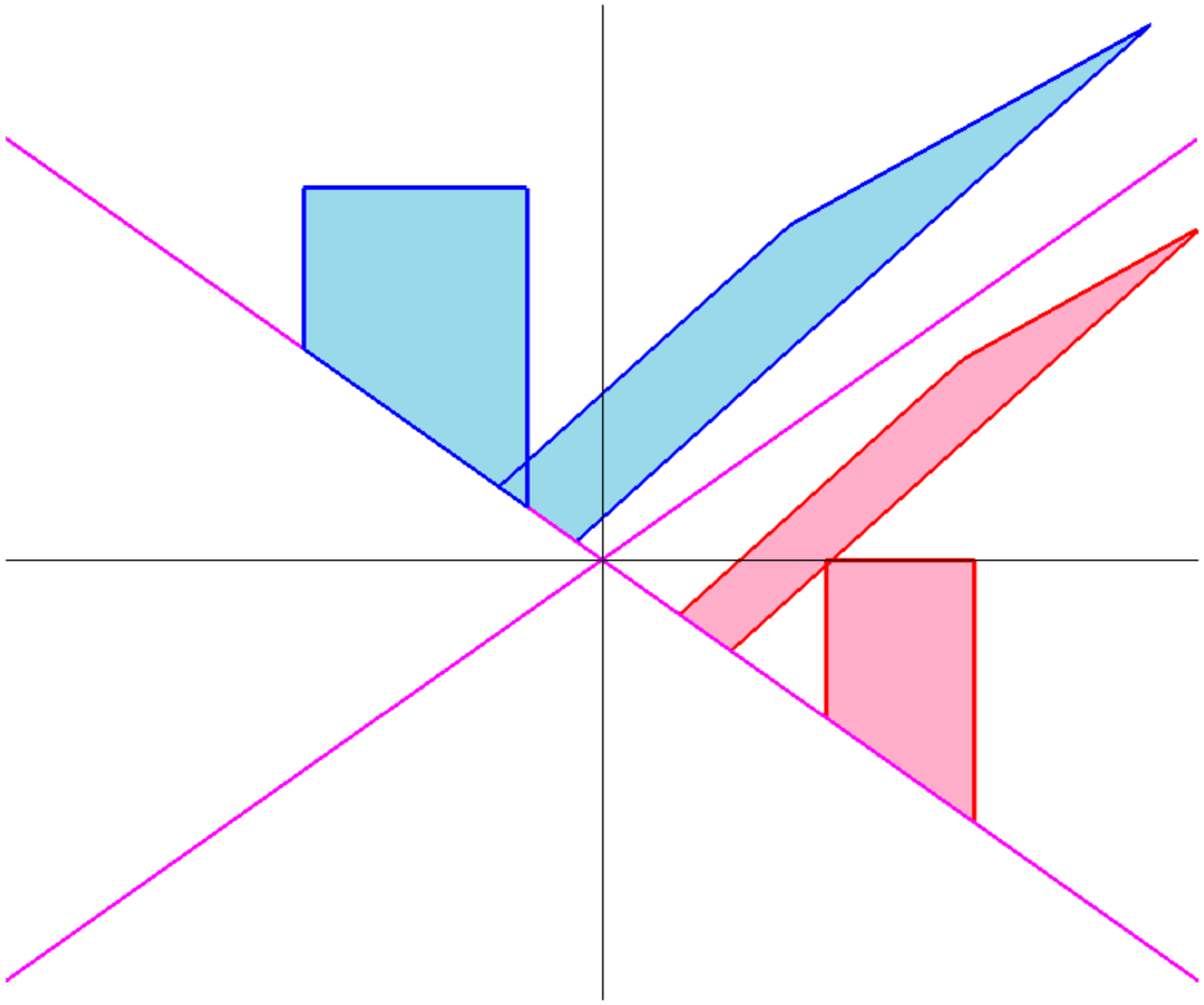}
\caption{The action of the Poicar\'{e} map.} \label{fig-02b}
\end{figure}

Precisely, Figure \ref{fig-02b} shows the deformation of two trapezoidal regions
under the action of the Poincar\'{e} map associated with $({\mathcal L}_{\lambda})$.
For this particular example we have chosen $h(y)=y$ and $\lambda=-1/2$, for
a time interval of length $\varsigma=1.5$. The two trapezoids are originally supported on each of the
two components of the stable manifold through the origin, which is the straight
line $y=-\sqrt{|\lambda|}x$. As an effect of the Poincar\'{e} map, the segments on the components of
the stable manifold are shifted towards the origin, while the remaining points of the trapezoids
move towards the unstable manifold $y=\sqrt{|\lambda|}x$.
\par
Next, we will consider the energy level lines passing through the points
$(x_{\pm},0)$, with $x_-<0<x_+$, namely
$$
   H(y) = \frac{\lambda}{2} (x_{\pm}^2-x^2).
$$
Then, either
$$
   x' = (h\circ H_{\pm}^{-1})\left(\frac{\lambda}{2} (x_+^2-x^2)\right),\quad\hbox{or}\quad x' = (h\circ H_{\pm}^{-1})\left(\frac{\lambda}{2} (x_-^2-x^2)\right),
$$
and hence, for any given $r_1, r_2>x_+$ and $\ell_1, \ell_2<x_-$, the transition time between the straight lines $x=r_1$, with $y<0$,  and $x=r_2$, with $y>0$, along the respective energy level lines, is given by
\begin{equation}
\label{iii.9}
{\mathcal S}^{\lambda}(r_1,r_2;(x_+,0)):=\!\!\int_{x_+}^{r_2}
\frac{dx}{(h\circ H_{+}^{-1})\left(\frac{\lambda}{2} (x_+^2-x^2)\right)}
\!+\! \int_{x_+}^{r_1}
\frac{dx}{-(h\circ H_{-}^{-1})\left(\frac{\lambda}{2} (x_+^2-x^2)\right)}.
\end{equation}
Similarly, the transition time between the straight line $x=\ell_1$, with $y>0$, and  $x=\ell_2$, with $y>0$,  is given through
\begin{equation}
\label{iii.10}
{\mathcal S}^{\lambda}(\ell_1,\ell_2;(x_-,0)):=\!\!\int_{\ell_1}^{x_-}
\frac{dx}{(h\circ H_{+}^{-1})\left(\frac{\lambda}{2} (x_-^2-x^2)\right)}\!+\!
\int_{\ell_2}^{x_-}
\frac{dx}{-(h\circ H_{-}^{-1})\left(\frac{\lambda}{2} (x_-^2-x^2)\right)}.
\end{equation}
The next compatibility condition
$$
  \max\left\{\frac{\lambda}{2} (x_+^2-r_j^2),\frac{\lambda}{2} (x_-^2-\ell_j^2)\right\}< H^*,
  \quad \text{for }\; j=1,2,
$$
makes sure that the arc of trajectory between each of these two lines is contained in the strip
${\mathbb R}\times (\varrho_{-},\varrho_{+})$.

\begin{itemize}
\item{} \textit{The case $\lambda = 0$.}
\end{itemize}

\noindent In this case, the system $({\mathcal L}_{0})$ reduces to $x' = h(y)$, $y'=0$. Thus,
the $x$-axis consists of equilibria. Moreover,
the motion has a constant speed, $h(y)$, on parallel lines to the $x$-axis
from left to the right, or the right to the left, according to whether $y>0$, or $y<0$.
Figure \ref{fig-02} shows the deformation of some rectangles by the action of the flow
induced by  $({\mathcal L}_{0})$.

\begin{figure}[htbp]
\centering
\includegraphics[scale=0.3]{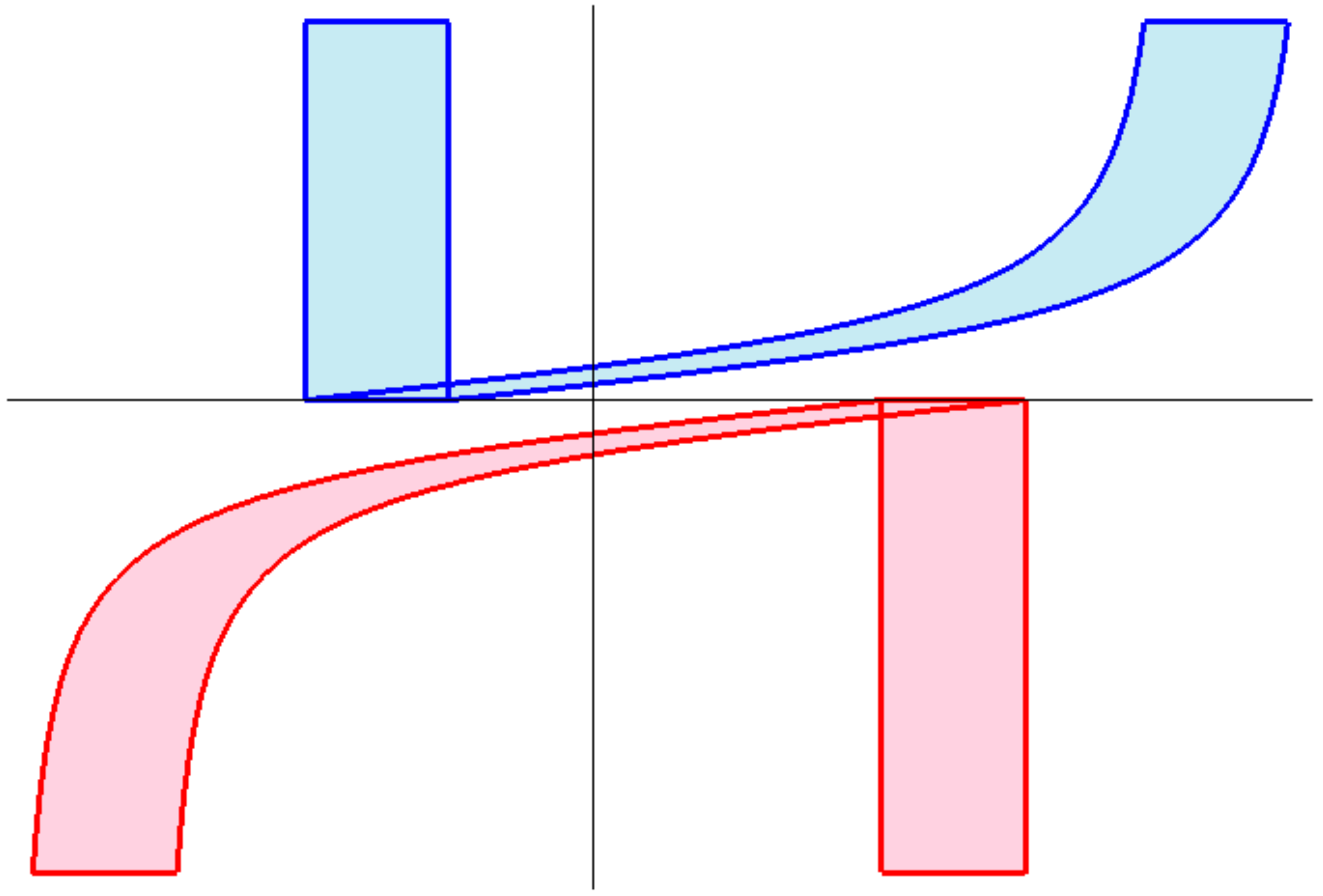}
\caption{The action of the flow induced by  $({\mathcal L}_{0})$.} \label{fig-02}
\end{figure}

Figure \ref{fig-02} shows the deformation of two rectangles
under the action of the Poincar\'{e} map associated with $({\mathcal L}_{0})$.
For this particular example we have chosen  $h(y)=y/\sqrt{1+y^2}$, which comes
from inverting the relativistic operator
$\frac{d}{dt}\bigl(\frac{x'}{\sqrt{1-(x')^2}}\bigr)$, for
a time interval of length $\varsigma=6$. The rectangles are $[-2,-1]\times[0,4]$ and $[2,3]\times[-5,0]$.
\par
Let us consider the orbits passing through the points
$(0,y_+)$ and $(0,y_-)$ with $\varrho_{-}<y_-<0< y_+ <\varrho_{+}$. From $x'=h(y_{\pm})$ we find that, for any given $\ell<r$, the transition times from the straight line
$x=\ell$ to the straight line $x=r$ with $y>0$, and from $x=r$ to $x=\ell$ with $y<0$, are given by
$$
  {\mathcal S}^{0}(\ell,r;(0,y_+)):=\frac{r-\ell}{h(y_+)} \quad\text{and }\quad
{\mathcal S}^{0}(r,\ell;(0,y_-)):=\frac{r-\ell}{-h(y_-)},
$$
respectively.

\begin{itemize}
\item{} \textit{The case  $\lambda > 0$.}
\end{itemize}

\noindent This situation is reminiscent of the case already analyzed for the system
$({\mathcal N}_{\lambda})$. Indeed, the origin is a center because, for every $c\in (0,H^*)$, the integral curve ${\mathcal E}_{\lambda}=c$ is a periodic orbit surrounding the origin that is symmetric with respect to the $y$-axis. This curve intersects the $x$-axis at the points $(x_{\pm}(c),0)$ with
$x_{\pm}(c)= \pm(2c/\lambda)^{1/2}$, and the $y$-axis at  $(0,H^{-1}_+(c))$ and $(0,H^{-1}_-(c))$.
Thus, by adapting the formulas for the times ${\mathcal T}_{i,j}$, $j\in\{I,II,III,IV\}$,  needed to cross each quadrant, after a simple change of variable, we find that
$$
    {\mathcal S}^{\lambda}_{I}(c):=(2c/\lambda)^{1/2}\int_{0}^{1}
\frac{dx}{(h\circ H_{+}^{-1})(c(1-x^2))}
$$
is the crossing time of the first quadrant, and
$$
  {\mathcal S}^{\lambda}_{IV}(c):=(2c/\lambda)^{1/2}\int_{0}^{1}
   \frac{dx}{-(h\circ H_{-}^{-1})(c(1-x^2))}
$$
is the crossing time of the fourth quadrant. By the symmetry with respect to the $y$-axis,
the crossing times of the third and second quadrants are given by
$$
   {\mathcal S}^{\lambda}_{II}(c)= {\mathcal S}^{\lambda}_{I}(c)
\quad \text{ and } \; \; {\mathcal S}^{\lambda}_{III}(c)= {\mathcal S}^{\lambda}_{IV}(c),
$$
respectively.

\section{Transformation of arcs in $({\mathcal N}_{i,\lambda})$}
\label{sec-non}
This  section analyzes the effect of the transformation of certain arcs through the Poincar\'{e} map
$\Phi_i$,  $i\in\{0,\ldots, m\}$.  To describe the movement of those arcs, we need to introduce some regions and a few notations. As these regions are always defined in the domain
${\mathbb R}\times (\varrho_-,\varrho_+)$, throughout this section \eqref{iii.5} is required to be satisfied.
\par
Let $0 < c_{1,i} < c_ {2,i}$ be such that, for every $c\in [c_{1,i},c_{2,i}]$, the condition
\eqref{iii.6} is satisfied. Then, the annulus
$$
   {\mathcal A}_{i}:= \{(x,y): c_{1,i}\leq {\mathcal H}(x,y)\leq c_{2,i}\}
$$
consists of the union of the periodic orbits ${\mathcal H}(x,y)=c$ of the system $({\mathcal N}_{i,\lambda})$ for $c\in [c_{1,i},c_{2,i}]$. As these orbits surround the origin,  also ${\mathcal A}_{i}$ is \lq\lq centered \rq\rq \ at the origin. Next, we split ${\mathcal A}_{i}$ into four rectangular
regions, ${\mathcal A}_{i,j}$, with $j\in\{I,II,III,IV\}$, where the index $j$ corresponds to the
intersection of ${\mathcal A}_{i}$ with the closed $j$-quadrant. For each of these
regions, we select an orientation, by fixing two opposite sides, ${\mathcal A}_{i,j}^-$, that we denote as follows:
\begin{align*}
  & \mathcal A^{\rm energy}_{i,j}:= ({\mathcal A}_{i,j},{\mathcal A}_{i,j}^-), \quad \hbox{where}
  \;\; {\mathcal A}_{i,j}^-:={\mathcal A}_{i,j}\cap \left(\{{\mathcal H}=c_{1,i}\}\cup \{{\mathcal H}=c_{2,i}\}\right),\\ &
  \mathcal A^{\rm axis}_{i,j}:= ({\mathcal A}_{i,j},{\mathcal A}_{i,j}^-),\quad \hbox{where} \;\;
  {\mathcal A}_{i,j}^- :={\mathcal A}_{i,j}\cap \bigl(\{xy=0\}\bigr).
\end{align*}
For each of the
four regions represented on each picture of Figure \ref{fig-01} we have
drawn also an arc connecting the two opposite sides of
${\mathcal A}_{i,j}^-$.

\begin{figure}[htbp]
\centering
\includegraphics[scale=0.35]{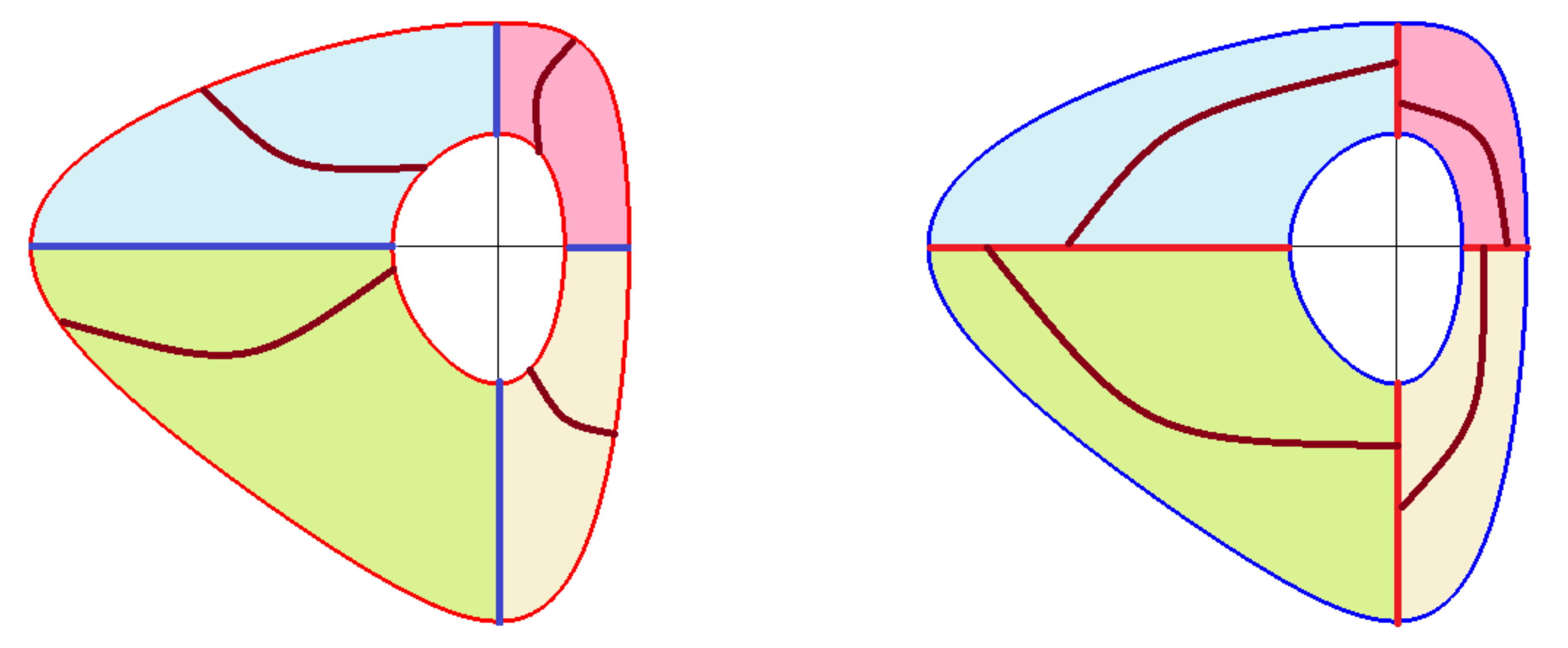}
\caption{The oriented regions
$\mathcal A^{\rm energy}_{i,j}$ (left) and
$\mathcal A^{\rm axis}_{i,j}$ (right). }
\label{fig-01}
\end{figure}

To generate Figure \ref{fig-01}, we  we have taken $h(y)=\log(\varrho_{+}/(\varrho_{+}-y))$, with $\varrho_{+}=4$,
which is an increasing homeomorphism $(\varrho_{-},\varrho_{+})\to {\mathbb R}$, for
$\varrho_{-}=-\infty$, with $h(0)=0$, $g(x)=-1+ e^x$, and $f_{\lambda}(x) := \lambda x + \mu g(x)$ with $\lambda=-1/20$ and $\mu=4/5$. For this particular choice of $h$, $g$ and $\lambda$, the
integral curves of ${\mathcal H}$ are closed trajectories around the origin contained
in the domain ${\mathbb R}\times(\varrho_{-},\varrho_{+})$ only for sufficiently small $c>0$.
\par
Throughout the rest of this paper, given two oriented rectangular
regions,  $({\mathcal R},{\mathcal R}^-)$ and $({\mathcal Q},{\mathcal Q}^-)$, and a homeomorphism $\eta$ between these two \emph{rectangles}, by setting
$$
    \eta: ({\mathcal R},{\mathcal R}^-)\overset{\mathfrak{m}}\stretchx ({\mathcal Q},{\mathcal Q}^-)
\quad (\mathfrak{m}\geq 1)
$$
we mean that there are $\mathfrak{m}$ compact disjoint sets $K_1,\ldots, K_{\mathfrak{m}}$
contained in ${\mathcal R}$ such that any arc $\gamma$ contained in
${\mathcal R}$ and linking the two opposite sides of ${\mathcal R}^-$
contains $\mathfrak{m}$ sub-arcs $\gamma_j\subset K_j$, $j\in\{1,...,\mf{m}\}$,
such that, for every $j\in\{1,...,\mf{m}\}$, $\eta(\gamma_j)$ is an arc contained in ${\mathcal Q}$
connecting the two opposite sides of ${\mathcal Q}^-$.  Thus, $\mathfrak{m}$
can be regarded as  a sort of \textit{crossing number}, as discussed by Kennedy and Yorke \cite{KeYo-2001}. When $\mathfrak{m}=1$, we omit it and simply set $\stretchx$.
\par
Once guaranteed that the annulus ${\mathcal A}_i$ is properly defined as a subset of ${\mathbb R}\times (\varrho_{-},\varrho_{+})$, the Poincar\'{e} map $\Phi_i$ of $({\mathcal N}_{i,\lambda})$ induces an homeomorphism of the annulus onto itself. Now, we will impose a suitable twist condition on the
boundaries of ${\mathcal A}_i$ in order to prove that an arc connecting
the inner with the outer boundaries of the annulus stretches into a spiral like curve
winding a certain number of times across the annulus itself, in order to prove
a property of the form
$$
  \mathcal A^{\rm energy}_{i,j'}\overset{\mathfrak{m}}\stretchx\,
\mathcal A^{\rm axis}_{i,j''}
$$
for a suitable choice of
$\mf{m}$, $j'$ and $j''$.  This is the main result of this section. Precisely, the following result holds. Recall that
the Poincar\'{e} map $\Phi_i$ acts on a time-interval of length $\tau_i=s_i-t_i$ for system
$({\mathcal N}_{i,\lambda})$.

\begin{proposition}
\label{pr4.1}
Given the cyclic clockwise order $IV \prec III \prec II \prec I\prec IV$ and $d, e\in \{c_{1,i},c_{2,i}\}$ with $d\not=e$, assume that,  for every $\k\in \{IV,III,II,I\}$, and, for some non-negative integers
$0\leq\alpha_{i} < \beta_{i}$, the following twist condition is satisfied
\begin{equation}
\label{iv.1}
\begin{split}
& {\mathcal T}_{i,\k+1}(d) + {\mathcal T}_{i,\k+2}(d) + \alpha_{i} {\mathcal T}_{i}(d) > \tau_i,\\
& {\mathcal T}_{i,\k}(e) + {\mathcal T}_{i,\k+1}(e)+\beta_{i} {\mathcal T}_{i}(e) <\tau_i.
\end{split}
\end{equation}
Then,
\begin{equation}
\label{iv.2}
\Phi_i:\mathcal A^{\rm energy}_{i,\k}\overset{\beta_i-\alpha_i}\stretchx\,
\mathcal A^{\rm axis}_{i,\k-1}
\cup \mathcal A^{\rm axis}_{i,\k+1}.
\end{equation}
Furthermore, fixing $\kappa=I\;\hbox{or}\;III$, for every $j=\{1,2,\ldots,2(\beta_i-\alpha_i)\}$, there is at least one solution of $(\mathcal{N}_{i,\l})$ starting at $\mathcal A_{i,\k}$ and having exactly
\begin{itemize}
\item[{\rm (i)}] $2\alpha_i+j$ zeroes in the $x$-component in the interval $(0,\tau_i)$, and
\item[{\rm (ii)}] $2\alpha_i+j+1$ zeroes in the $y$-component in $[0,\tau_i]$.
\end{itemize}
\end{proposition}

The rest of this section is devoted to the proof of Proposition \ref{pr4.1}, which is divided into two parts. The first part shows \eqref{iv.2}. The second one holds from a technical lemma where
the nodal behavior of the solutions starting at $\mathcal{A}_{i,\k}$ for $\k\in\{I,III\}$ is described. This provides us with the  multiplicity result of the zeroes in the $x$ and $y$ components
stated in the proposition. A similar result holds for solutions starting at $\mathcal{A}_{i,\k}$ with $\k\in\{II,IV\}$. However, as it will become apparent in later sections,  we are just interested in the nodal behavior of the solutions starting at $\mathcal{A}_{i,I}$ or $\mathcal{A}_{i,III}$.

\begin{proof}
Suppose \eqref{iv.1}. We will detail the proof of \eqref{iv.2}  in the case $\kappa=I$, as in the
remaining cases follows easily by symmetry. Just to fix ideas, suppose that
$d= c_{1,i}$ and $e= c_{2,i}$ (the other case is analogous).
\par
Using a system of polar coordinates with the angle
counted in the clockwise sense starting from the positive $y$-axis, we can associate to the solutions of
$(\mathcal{N}_{i,\l})$ an angular component $\theta(t;z)$, with $t\in[0,\tau_i]$ and $z\in {\mathcal A}_{i,I}$ representing the angular coordinate of the solution at time $t$, starting
at the point $z$ at time $0$.  The map $\theta$ is continuous, and $0 \leq  \theta(0,z) \leq \pi/2$, for all $z\in {\mathcal A}_{i,I}$.
\par
Let $\gamma$ be a continuous arc of curve contained in ${\mathcal A}_{i,I}$ linking the two sides of
${\mathcal A}_{i,I}^-$ which are the part of the level lines ${\mathcal H}=c_{1,i}$ and
${\mathcal H}=c_{2,i}$ in the first quadrant. The arc $\gamma$
is the homeomorphic image of the compact interval $[0,1]$ through a homeomorphism,
still denoted by $\gamma:[0,1]\to {\mathcal A}_{i,I}$,  such that $\gamma(0)=z_0\in
\{{\mathcal H}=c_{1,i}\}$ and $\gamma(1)=z_1\in \{{\mathcal H}=c_{2,i}\}$. Without loss of generality,
we can assume that $c_{1,i}<{\mathcal H}(\gamma(s))<c_{2,i}$ for all $s\in (0,1)$.
Now, consider the image of $\gamma=\gamma([0,1])$ through the Poincar\'{e} map $\Phi_i$. Then,
$$
    \vartheta(t,s):= \theta(t;\gamma(s)), \quad t\in[0,\tau_i],\;\; s\in [0,1],
$$
represents the angular coordinate of the solution of $({\mathcal N}_{i,\lambda})$ at time $t$, starting
at $\gamma(s)$ (at time $0$).  The map $\vartheta$ is continuous and satisfies
$0 \leq  \vartheta(0,s) \leq \pi/2$ for all $s\in [0,1]$. Moreover, for each nonnegative integer $k$, as soon as $c=\mathcal{H}(\gamma(s))$, we have that
\begin{equation}
\label{iv.3}
{\mathcal T}_{i,IV}(c) + {\mathcal T}_{i,III}(c) + k{\mathcal T}_{i}(c)>t\;
\Longrightarrow\; \vartheta(t,s) < \frac{3}{2}\pi + 2k\pi,
\end{equation}
\begin{equation}
\label{iv.4}
{\mathcal T}_{i,I}(c) + {\mathcal T}_{i,IV}(c) + k{\mathcal T}_{i}(c)<t\;
\Longrightarrow\; \vartheta(t,s) > \pi + 2k\pi.
\end{equation}
Thus, from \eqref{iv.3} and the first condition of \eqref{iv.1}, we find that
$$
\vartheta(\tau_i,0) < \frac{3}{2}\pi +  2\alpha_{i}\pi.
$$
On the other hand, from \eqref{iv.4} and the second condition of  \eqref{iv.1}, it becomes apparent
that
$$
   \vartheta(\tau_i,1) > \pi +  2\beta_{i}\pi.
$$
Thus, the continuous map
$[0,1]\ni s\mapsto \vartheta(\tau_i,s)$ covers the whole compact interval
$$
  \left[\tfrac{3}{2}\pi + 2\alpha_{i}\pi,\pi+2\beta_{i}\pi\right].
$$
Hence, there are $2(\beta_{i}-\alpha_{i})$ pairwise disjoint
compact subintervals of $[0,1]$ such that this map covers, for every $k\in\{\alpha_{i},\dots, \beta_{i}-1\}$,  the intervals
$$
  \left[\tfrac{3}{2}\pi + 2k\pi, 2(k+1)\pi\right]\quad\text{and}\quad
  \left[\tfrac{1}{2} \pi +2(k+1)\pi,\pi+ 2(k+1)\pi\right].
$$
Now, we introduce the sets
\begin{equation}
\label{iv.5}
\begin{split}
& K_{j,II}:= \left\{z\in {\mathcal A}_{i,I}: \, \theta(\tau_i;z)\in \left[\tfrac{3}{2}\pi +
2(\alpha_{i}+j-1)\pi,2(\alpha_{i}+j)\pi\right]\right\},\\[3pt]
& K_{j,IV}:= \left\{z\in {\mathcal A}_{i,I}: \, \theta(\tau_i;z)\in \left[\tfrac{1}{2}\pi +
2(\alpha_{i}+j)\pi,\pi+2(\alpha_{i}+j)\pi\right]\right\},
\end{split}
\end{equation}
for $j= 1, \dots,\beta_{i}-\alpha_{i},$ which are compact, nonempty and pairwise
disjoint. By a simple continuity argument, given an arbitrary arc $\gamma$
in ${\mathcal A}_{i,I}$ connecting the two opposite sides of $\mathcal A^{\rm energy}_{i,I}$,
it is easily seen that, for every $\k\in\{II,IV\}$, there are $\beta_i-\alpha_i$ sub-arcs $\gamma_{j,\k}\subset K_{j,\k}$ such that $\Phi_i(\gamma_{j,\k})$
is an arc contained in ${\mathcal A}_{i,\k}$ linking the two sides $y=0$ and  $x=0$ of
$\mathcal A^{\rm axis}_{i,\k}$. This proves that
$$\Phi_i:\mathcal A^{\rm energy}_{i,I}\overset{\beta_i-\alpha_i}\stretchx\,
\mathcal A^{\rm axis}_{i,II}
\cup \mathcal A^{\rm axis}_{i,IV}$$ according to our definition. The multiplicity results of the zeroes
of the solutions stated in the second part of the proposition holds as a direct consequence of the next technical lemma, which completes the proof of the proposition.
\end{proof}

\begin{lemma}
\label{le4.1}
For every $\k\in\{I,III\}$, and any arc $\gamma\subset\mathcal{A}_{i,\k}^{\rm energy}$ joining the opposite sides of $\mathcal{A}_{i,\k}^{-}$, the solutions of $({\mathcal N}_{i,\lambda})$
with initial value
on any sub-arc $\gamma_{j,\tilde{\k}}\subset K_{j,\tilde{\k}}$ of $\gamma$ have exactly, for every $j=1,2,\ldots,\beta_i-\alpha_i$,
\begin{equation*}
\left.
\begin{array}{ll}
2(\alpha_i+j)-1&\;\hbox{zeroes in the x-component in}\;\;(0,\tau_i)\\
2(\alpha_i+j)&\;\hbox{zeroes in the y-component in}\;\;[0,\tau_i]
\end{array}
\right\}
\;\hbox{if}\;\;\tilde{\k}=\k-1
\end{equation*}
and
\begin{equation*}
\left.
\begin{array}{ll}
2(\alpha_i+j)&\;\hbox{zeroes in the x-component in}\;\;(0,\tau_i)\\
2(\alpha_i+j)+1&\;\hbox{zeroes in the y-component in}\;\;[0,\tau_i]
\end{array}
\right\}
\;\hbox{if}\;\;\tilde{\k}=\k+1.
\end{equation*}
\end{lemma}
\begin{proof}
We focus our attention on case  $\k=I$, since the proof for $\k=III$ is completely analogous. By the definition of $\vartheta(t,s)$, we have that  $0\leq \vartheta(0,s)\leq \frac{\pi}{2}$ for all $s\in [0,1]$.
Moreover, if the arc $\gamma_{j,II}\subset K_{j,II}$ is parameterized on a subinterval
$[s^0_{j,II},s^1_{j,II}]$ of $[0,1]$, then
$$
\tfrac{3}{2} \pi +
2(\alpha_{i}+j-1)\pi \leq \vartheta(\tau_i,s) \leq 2(\alpha_{i}+j)\pi,
$$
with
$$
\vartheta(\tau_i,s^0_{j,II})=\tfrac{3}{2}  \pi +
2(\alpha_{i}+j-1)\pi,\qquad \vartheta(\tau_i,s^1_{j,II})=2(\alpha_{i}+j)\pi.
$$
Thus, for every  $z\in \gamma_j$, the solution $(x(t;z),y(t;z))$ satisfies:
\begin{itemize}
\item  $x(t;z)$ has exactly $2(\alpha_{i}+j)-1$ zeroes in the interval $(0,\tau_i)$,
\item  $y(t,z)$, as well as $x'(t;z)$, has exactly $2(\alpha_{i}+j)$ zeroes in the interval $[0,\tau_i]$.
\end{itemize}
Similarly, if the arc $\gamma_{j,IV}\subset K_{j,IV}$ is parameterized on a subinterval
$[s^0_{j,IV},s^1_{j,IV}]$ of $[0,1]$,
$$
\tfrac{1}{2} \pi +
2(\alpha_{i}+j)\pi \leq \vartheta(\tau_i,s) \leq \pi+2(\alpha_{i}+j)\pi,
$$
with
$$
   \vartheta(\tau_i,s^0_{j,IV})=\tfrac{1}{2}\pi +
   2(\alpha_{i}+j)\pi,\qquad \vartheta(\tau_i,s^1_{j,IV})=\pi+2(\alpha_{i}+j)\pi.
$$
Therefore, for every  $z\in \gamma_j$, the solution $(x(t;z),y(t;z))$ satisfies:
\begin{itemize}
\item $x(t,z)$ has exactly $2(\alpha_{i}+j)$ zeroes in the interval $(0,\tau_i)$,
\item $y(t,z)$, as well as $x'$, has exactly $2(\alpha_{i}+j)+1$ zeroes in the interval $[0,\tau_i]$.
\end{itemize}
This end the proof.
\end{proof}

Adapting the proof of Proposition \ref{pr4.1}, the next result holds.  It will be invoked at the final stages of our analysis for the associated boundary value problem in $[0,L]$.

\begin{corollary}
\label{co4.1}
Under the same assumptions of Proposition \ref{pr4.1}, assume that, for every  $\kappa\in \{IV,III,II,I\}$, the stronger twist condition
\begin{equation}
\label{iv.6}
\begin{split}
& {\mathcal T}_{i,\k+1}(d) + \alpha_{i} {\mathcal T}_{i}(d) > \tau_i, \\
& {\mathcal T}_{i,\k}(e) + {\mathcal T}_{i,\k+1}(e)+\beta_{i} {\mathcal T}_{i}(e) <\tau_i,
\end{split}
\end{equation}
holds. Then, for every $\kappa'\in \{IV,III,II,I\}$,
\begin{equation}\label{eq-st2-unified-strong}
\Phi_i:\mathcal A^{\rm energy}_{i,\k}\overset{\beta_i-\alpha_i}\stretchx\,\mathcal A^{\rm axis}_{i,\k'}.
\end{equation}
\end{corollary}

\begin{remark}
\label{re4.1}
\rm By slightly changing the twist conditions of  Proposition \ref{pr4.1}, we can get some different results  for the Poincar\'e map $\Phi_i$. For instance, imposing the twist condition
$$
{\mathcal T}_{i,IV}(d) + {\mathcal T}_{i,III}(d) + \alpha_{i} {\mathcal T}_{i}(d) > \tau_i,\quad \beta_{i} {\mathcal T}_{i}(e) <\tau_i,
$$
it follows that
$$
\Phi_i:\mathcal A^{\rm energy}_{i,I}\overset{\beta_i-\alpha_i}\stretchx\,
\mathcal A^{\rm axis}_{i,II}.
$$
However, in our application, it suffices  to apply Proposition \ref{pr4.1}.
\end{remark}

\section{Transformation of arcs in $(\mathcal{L}_{\l})$}
\label{sec-lin}
This section considers the Poincar\'{e} map $\Psi_i$ associated with the
$x$-linear system $({\mathcal L}_{\lambda})$ in order to link the dynamics in ${\mathcal A}_i$
and ${\mathcal A}_{i+1}$. The map $\Psi_i$  moves the points along the level lines of the Hamiltonian
${\mathcal E}(x,y)$. As the dynamics vary with the sign of $\lambda$,
we will discuss each of these cases separately.

\subsection{The case $\l<0$}

Throughout this section, for every $i\in\{0,...,m\}$, we denote by $y_+(c_{\kappa,i})>0$, $\kappa=1, 2$, the unique positive value of $y$ for which $\mathcal{H}(0,y_+(c_{\kappa,i}))=c_{\kappa,i}$, $\kappa=1, 2$, and set
\begin{equation}
\label{v.1}
y_{+,i}:= \min\{y_+(c_{1,i}),y_+(c_{1,i+1})\}>0,\qquad i\in\{0,...,m-1\}.
\end{equation}
Then, we consider the energy level line  $\mathcal{E}(x,y)=\mathcal{E}(0,y_{+,i})$ and its intersections with the outer boundaries of ${\mathcal A}_{i,II}$
and ${\mathcal A}_{i+1,I}$, whose abscissas are
\begin{align*}
   x_{\ell,i} & :=G^{-1}_{-}( \mu_i^{-1}[H(y_+(c_{2,i})) - H(y_{+,i})])<0, \\
   x_{r,i+1} & :=G^{-1}_{+}( \mu_{i+1}^{-1}[H(y_+(c_{2,i+1})) - H(y_{+,i}) ])>0,
\end{align*}
respectively, where $G^{-1}_-$ and $G^{-1}_+$ are the inverse functions of $G$ restricted to $(-\infty,0)$ and $(0,+\infty)$, respectively. Then,  by \eqref{iii.7}, the transition time between the straight lines
$x=x_{\ell,i}$ and $x=x_{r,i+1}$ along $\mathcal{E}(x,y)=\mathcal{E}(0,y_{+,i})$  is given by
\begin{equation}
\label{v.2}
{\mathcal S}_{i,II\to I}:=
{\mathcal S}^{\lambda}(x_{\ell,i},x_{r,i+1};(0,y_{+,i}))=
\int_{x_{\ell,i}}^{x_{r,i+1}}
\frac{dx}{(h\circ H_{+}^{-1})\left(H(y_{+,i})- \frac{\lambda}{2} x^2\right)}\,.
\end{equation}
Similarly, for every $i\in\{0,...,m\}$, we denote by $y_-(c_{\kappa,i})<0$, $\kappa=1, 2$, the unique negative value of $y$ for which $\mathcal{H}(0,y_-(c_{\kappa,i}))=c_{\kappa,i}$, $\kappa=1, 2$, and set
\begin{equation}
\label{v.3}
y_{-,i}:= \max\{y_-(c_{1,i}),y_-(c_{1,i+1})\}<0,\qquad i\in\{0,...,m-1\}.
\end{equation}
Arguing as before, the energy level line passing through $(0,y_{-,i})$
intersects the outer boundaries of ${\mathcal A}_{i,IV}$
and ${\mathcal A}_{i+1,III}$ at the points with abscissas
\begin{align*}
   x_{r,i} & :=G^{-1}_{+}( \mu_{i}^{-1}[H(y_-(c_{2,i})) - H(y_{-,i}) ])>0,\\
   x_{\ell,i+1} & :=G^{-1}_{-}( \mu_{i+1}^{-1}[H(y_-(c_{2,i+1})) - H(y_{-,i})])<0,
\end{align*}
respectively. Thus, due to \eqref{iii.8}, the transition time between  the straight lines
$x=x_{r,i}$ and $x=x_{\ell,i+1}$ along $\mathcal{E}(x,y)=\mathcal{E}(0,y_{-,i})$ on the lower half-plane is given by
\begin{equation}
\label{v.4}
{\mathcal S}_{i,IV\to III}:=
{\mathcal S}^{\lambda}(x_{r,i},x_{\ell,i+1};(0,y_{-,i}))=
\int_{x_{\ell,i+1}}^{x_{r,i}}
\frac{dx}{-(h\circ H_{-}^{-1})\left(H(y_{-,i}) - \frac{\lambda}{2} x^2\right)}\,.
\end{equation}
Analogously, for every $i\in\{0,...,m\}$, we denote by $x_+(c_{\kappa,i})>0$, $\kappa=1, 2$, the unique positive value of $x$ for which $\mathcal{H}(x_+(c_{\kappa,i}),0)=c_{\kappa,i}$, $\kappa=1, 2$, set
$$
   x_{+,i}:= \min\{x_+(c_{1,i}),x_+(c_{1,i+1})\}>0, \qquad i\in\{0,...,m-1\},
$$
and consider the  energy level line passing through $(x_{+,i},0)$, as well as  its
intersections with the outer boundaries of ${\mathcal A}_{i,IV}$
and ${\mathcal A}_{i+1,I}$, whose abscissas are
\begin{align*}
  x_{r_1,i} & :=G^{-1}_{+}\left( G(x_+(c_{2,i}))+\frac{\lambda}{2\mu_i}
[x_+^2(c_{2,i}) -x_{+,i}^2 ]\right)>0,\\ x_{r_2,i+1} & :=G^{-1}_{+}\left( G(x_+(c_{2,i+1}))+\frac{\lambda}{2\mu_{i+1}}
[x_+^2(c_{2,i+1}) -x_{+,i}^2 ]\right)>0,
\end{align*}
respectively. Thanks to \eqref{iii.9}, the transition time between the straight lines
$x=x_{r_1,i}$ and  $x=x_{r_2,i+1}$ along the level line
$\mathcal{E}(x,y)=\mathcal{E}(x_{+,i},0)$ on the right half-plane is given by
\begin{equation}
\label{v.5}
\begin{split}
{\mathcal S}_{i,IV\to I}  := &
{\mathcal S}^{\lambda}(x_{r_1,i},x_{r_2,i+1};(x_{+,i},0))\\
=& \int_{x_{+,i}}^{x_{r_2,i+1}} \frac{dx}{(h\circ H_{+}^{-1})\left(\frac{\lambda}{2} (x_{+,i}^2-x^2)\right)}
+\int_{x_{+,i}}^{x_{r_1,i}}
\frac{dx}{-(h\circ H_{-}^{-1})\left(\frac{\lambda}{2} (x_{+,i}^2-x^2)\right)}.
\end{split}
\end{equation}
Finally, for every $i\in\{0,...,m\}$, we denote by $x_-(c_{\kappa,i})>0$, $\kappa=1, 2$, the unique negative value of $x$ for which $\mathcal{H}(x_-(c_{\kappa,i}),0)=c_{\kappa,i}$, $\kappa=1, 2$, set
$$
   x_{-,i}:= \max\{x_-(c_{1,i}),x_-(c_{1,i+1})\}<0, \qquad i\in\{0,...,m-1\},
$$
and consider the  energy level line passing through $(x_{-,i},0)$, and its intersections with the outer boundaries of ${\mathcal A}_{i,II}$ and ${\mathcal A}_{i+1,III}$, whose abscissas are
\begin{align*}
  x_{\ell_1,i} & :=G^{-1}_{-}\left( G(x_-(c_{2,i}))+\frac{\lambda}{2\mu_{i}}
[x_-^2(c_{2,i}) -x_{-,i}^2 ]\right),\\ x_{\ell_2,i+1} & :=G^{-1}_{-}\left( G(x_-(c_{2,i+1}))+\frac{\lambda}{2\mu_{i+1}}
[x_-^2(c_{2,i+1}) -x_{-,i}^2 ]\right),
\end{align*}
respectively. Then, by \eqref{iii.10},
we obtain that the transition time between the straight lines
$x=x_{\ell_1,i}$ and $x=x_{\ell_2,i+1}$ along the integral curve
$\mc{E}(x,y)=\mc{E}(x_{-,i},0)$ on the left half-plane is given by
\begin{equation}
\label{v.6}
\begin{split}
{\mathcal S}_{i,II\to III}:= &
{\mathcal S}^{\lambda}(x_{\ell_1,i},x_{\ell_2,i+1};(x_{-,i},0))\\ =&
\int_{x_{\ell_1,i}}^{x_{-,i}}
\frac{dx}{(h\circ H_{+}^{-1})\left(\frac{\lambda}{2} (x_{-,i}^2-x^2)\right)}
+ \int_{x_{\ell_2,i+1}}^{x_{-,i}}
\frac{dx}{-(h\circ H_{-}^{-1})\left(\frac{\lambda}{2} (x_{-,i}^2-x^2)\right)}.
\end{split}
\end{equation}

\subsection{The case $\l=0$}

The abscissas in the upper half-plane of the intersection points of  the straight line $y=y_{+,i}$ (see \eqref{v.1})   with the outer boundaries of ${\mathcal A}_{i,II}$
and ${\mathcal A}_{i+1,I}$ are
\begin{align*}
  x_{\ell,i} & :=G^{-1}_{-}( \mu_i^{-1}[c_{2,i} - H(y_{+,i}) ])<0,\\
  x_{r,i+1} & :=G^{-1}_{+}( \mu_{i+1}^{-1}[c_{2,i+1} - H(y_{+,i})])>0,
\end{align*}
respectively. Thus, the transition time between $x=x_{\ell,i}$ and $x=x_{r,i+1}$ along $y= y_{+,i}$
is
\begin{equation}
\label{v.7}
{\mathcal S}_{i,II\to I}:=
{\mathcal S}^{0}(x_{\ell,i},x_{r,i+1},(0,y_{+,i})):=\frac{x_{r,i+1}-x_{\ell,i}}{h(y_{+,i})}
<\frac{x_+(c_{2,i+1})-x_-(c_{2,i})}{h(y_{+,i})}.
\end{equation}
Similarly, the abscissas in the lower half-plane of the  intersection points of the straight line $y=y_{-,i}$ (see \eqref{v.3}) with the outer boundaries of ${\mathcal A}_{i,IV}$
and ${\mathcal A}_{i+1,III}$ are
\begin{align*}
  x_{r,i} & :=G^{-1}_{+}( \mu_i^{-1}[c_{2,i} - H(y_{-,i}) ])>0,\\
  x_{\ell,i+1} & :=G^{-1}_{-}( \mu_{i+1}^{-1}[c_{2,i+1} - H(y_{-,i}) ])<0,
\end{align*}
respectively. Hence, the transition time between $x=x_{r,i}$ and $x=x_{\ell,i+1}$ along $y= y_{-,i}$ equals
\begin{equation}
\label{v.8}
{\mathcal S}_{i,IV\to III}:=
{\mathcal S}^{0}(x_{r,i},x_{\ell,i+1},(0,y_{-,i})):=\frac{x_{r,i}-x_{\ell,i+1}}{-h(y_{-,i})}
<\frac{x_+(c_{2,i})-x_-(c_{2,i+1})}{-h(y_{-,i})}\,.
\end{equation}

\subsection{The case $\l>0$}
\label{subsec-l>0}

To carry out our analysis in this case, we must previously guarantee the existence of a closed orbit of $({\mathcal L}_{\lambda})$ crossing the two annuli ${\mathcal A}_i$ and ${\mathcal A}_{i+1}$ as
described in Figure \ref{fig-03}. This  figure shows an admissible configuration of two annuli
${\mathcal A}_{i}$ and ${\mathcal A}_{i+1}$ which are crossed by a periodic orbit
of the $x$-linear system $({\mathcal L}_{\lambda})$ represented by the level line
${\mathcal E}(x,y)=c$. For the particular example generating Figure \ref{fig-03}
we have chosen $h(y)=y^3/(1+y^2)$, which comes from the inverse of a differential operator
behaving like a $p$-Laplacian near the origin that it is asymptotically linear at infinity, $g(x)=x|x|$, $\lambda=1$, $\mu_i=1$ and $\mu_{i+1} =2$. The annular regions are determined by the values of the
parameters
$$
  c_{1,i}=H(6),\quad c_{2,i}=H(7),\quad c_{1,i+1}=H(10),\quad
  c_{2,i+1}=H(11),
$$
and we have taken $c=11/2$.

\begin{figure}[h!]
\centering
\includegraphics[scale=0.3]{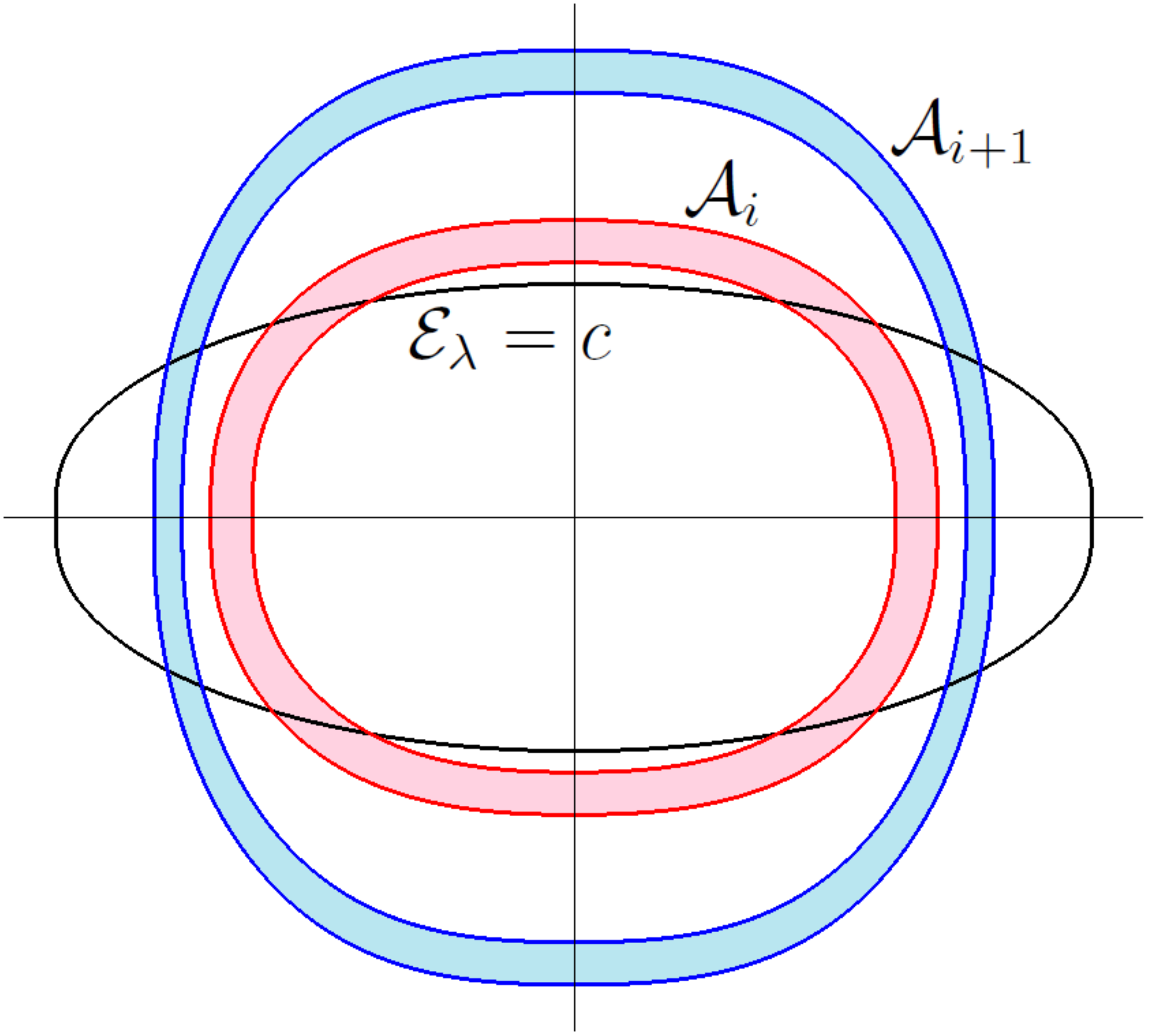}
\caption{A closed orbit of $({\mathcal L}_{\lambda})$ crossing ${\mathcal A}_i$ and ${\mathcal A}_{i+1}$.} \label{fig-03}
\end{figure}

To accomplish such a configuration, let us consider the level line ${\mathcal E}(x,y) =c$
with $c>0$ such that
\begin{equation}
\label{v.9}
0 <c\leq H_{*}:=\min_{j\in\{i,i+1\}}H(y_{\pm}(c_{1,j}))=\min_{j\in\{i,i+1\}}c_{1,j}.
\end{equation}
By the definition of $\mathcal{E}(x,y)$, we have that
$$
   -\sqrt{\frac{2c}{\lambda}}\leq x\leq \sqrt{\frac{2c}{\lambda}}\quad
\hbox{for all} \;\; (x,y)\in \mathcal{E}^{-1}(c).
$$
Thus, assuming that
\begin{equation}
\label{v.10}
\min\left\{G(-\sqrt{2c/\lambda}),G(\sqrt{2c/\lambda})\right\}>
\max_{j\in\{i,i+1\}} \frac{c_{2,j}-c}{\mu_j},
\end{equation}
it becomes apparent that $\mc{E}^{-1}(c)$ intersects the outer boundaries of ${\mathcal A}_{i,II}$
and ${\mathcal A}_{i+1,I}$, as illustrated in Figure \ref{fig-03}.
Note that \eqref{v.10} is equivalent to
\begin{equation}
\label{v.11}
-\sqrt{\frac{2c}{\lambda}} < \min_{j\in\{i,i+1\}}x_-(c_{2,j})<0<\max_{j\in\{i,i+1\}}x_+(c_{2,j})<\sqrt{\frac{2c}{\lambda}}.
\end{equation}
Indeed, since for every $x\in\R$
\begin{align*}
  & xG'(x)=xg(x)>0,\\
  & x\left(\frac{\l}{2}x^2+\mu_jG(x) \right)'=\l x^2+x\mu_j g(x)>0,\\
  & \frac{\lambda}{2}x_{\pm}^2(c_{2,j})+\mu_jG(x_{\pm}(c_{2,j}))=c_{2,j}>c,\qquad\hbox{for}\;j\in\{i,i+1\},
\end{align*}
it becomes apparent that, under condition \eqref{v.10}, the intersections of the level line
${\mathcal E}^{-1}(c)$ with the $x$-axis are outside both annuli ${\mathcal A}_{i}$ and
${\mathcal A}_{i+1}$. Thus, they must across, as claimed above. Moreover, the abscissas of the
intersection points are
\begin{align*}
   x_{\ell,c_{2,i}} & :=G^{-1}_{-}((c_{2,i}-c)/\mu_i)< 0,\\  x_{\ell,c_{1,i}} & :=G^{-1}_{-}((c_{1,i}-c)/\mu_i)\leq 0,\\
  x_{r,c_{2,i+1}} & :=G^{-1}_{+}((c_{2,i+1}-c)/\mu_{i+1})> 0,\\
  x_{r,c_{1,i+1}} & :=G^{-1}_{+}((c_{1,i+1}-c)/\mu_{i+1})\geq 0,
\end{align*}
respectively. Thus, the transition times between the straight lines $x=x_{\ell,c_{2,i}}$ and $x=x_{r,c_{2,i+1}}$, and between $x=x_{\ell,c_{1,i}}$ and  $x=x_{r,c_{1,i+1}}$, along $\mc{E}^{-1}(c)$ in the upper half-plane are given by
\begin{align*}
{\mathcal S}^{c_{2,i\to i+1}}_{II\to I}(c) & :=
{\mathcal S}^{\lambda}(x_{\ell,c_{2,i}},x_{r,c_{2,i+1}},H_{+}^{-1}(c))=
\int_{x_{\ell,c_{2,i}}}^{x_{r,c_{2,i+1}}}
\frac{dx}{(h\circ H_{+}^{-1})\left(c- \frac{\lambda}{2} x^2\right)},\\
{\mathcal S}^{c_{1,i\to i+1}}_{II\to I}(c) & :=
{\mathcal S}^{\lambda}(x_{\ell,c_{1,i}},x_{r,c_{1,i+1}},H_{+}^{-1}(c))=
\int_{x_{\ell,c_{1,i}}}^{x_{r,c_{1,i+1}}}
\frac{dx}{(h\circ H_{+}^{-1})\left(c- \frac{\lambda}{2} x^2\right)},
\end{align*}
respectively. Similarly, the transition time from the straight line $x=x_{\ell,c_{1,i}}$ to the point $(\sqrt{2c/\lambda},0)$ along the curve ${\mathcal E}^{-1}(c)$ in the upper half-plane is given by
\begin{equation}
\label{v.12}
{\mathcal S}^{c_{1,i}\to\sqrt{}}_{II\to I}(c):=
{\mathcal S}^{\lambda}(x_{\ell,c_{1,i}},\sqrt{2c/\lambda},H_{+}^{-1}(c))=
\int_{x_{\ell,c_{1,i}}}^{\sqrt{2c/\lambda}}
\frac{dx}{(h\circ H_{+}^{-1})\left(c- \frac{\lambda}{2} x^2\right)}.
\end{equation}
The same level line intersects
the outer and inner boundaries of ${\mathcal A}_{i,IV}$
and ${\mathcal A}_{i+1,III}$ at the points with abscissas
\begin{align*}
   x_{r,c_{2,i}} & :=G^{-1}_{+}((c_{2,i}-c)/\mu_i)>0,\\
   x_{r,c_{1,i}} & :=G^{-1}_{+}((c_{1,i}-c)/\mu_i)\geq 0,\\
   x_{\ell,c_{2,i+1}} & :=G^{-1}_{-}((c_{2,i+1}-c)/\mu_{i+1})< 0,\\
   x_{\ell,c_{1,i+1}} & :=G^{-1}_{-}((c_{1,i+1}-c)/\mu_{i+1})\leq 0,
\end{align*}
respectively. Hence, the transition times between the straight lines
$x=x_{r,c_{2,i}}$ and $x=x_{\ell,c_{2,i+1}}$, and between
$x=x_{r,c_{1,i}}$ and $x=x_{\ell,c_{1,i+1}}$,  along ${\mathcal E}^{-1}(c)$
in the lower half-plane, are given by
\begin{align*}
{\mathcal S}^{c_{2,i\to i+1}}_{IV\to III}(c) & :=
{\mathcal S}^{\lambda}(x_{r,c_{2,i}},x_{\ell,c_{2,i+1}},H_{-}^{-1}(c))=
\int_{x_{\ell,c_{2,i+1}}}^{x_{r,c_{2,i}}}
\frac{dx}{-(h\circ H_{-}^{-1})\left(c- \frac{\lambda}{2} x^2\right)},\\
{\mathcal S}^{c_{1,i\to i+1}}_{IV\to III}(c) & :=
{\mathcal S}^{\lambda}(x_{r,c_{1,i}},x_{\ell,c_{1,i+1}},H_{-}^{-1}(c))=
\int_{x_{\ell,c_{1,i+1}}}^{x_{r,c_{1,i}}}
\frac{dx}{-(h\circ H_{-}^{-1})\left(c- \frac{\lambda}{2} x^2\right)},
\end{align*}
respectively. Similarly, the transition time from the straight line
$x=x_{r,c_{1,i}}$ to the point $(-\sqrt{2c/\lambda},0)$, along the level line ${\mathcal E}^{-1}(c)$ in the lower half-plane, is given by
\begin{equation}\label{v.13}
{\mathcal S}^{c_{1,i}\to \sqrt{}}_{IV\to III}(c):=
{\mathcal S}^{\lambda}(x_{r,c_{1,i}},-\sqrt{2c/\lambda},H_{-}^{-1}(c))=
\int_{-\sqrt{2c/\lambda}  }^{x_{r,c_{1,i}}}
\frac{dx}{-(h\circ H_{-}^{-1})\left(c- \frac{\lambda}{2} x^2\right)}.
\end{equation}

Finally, we need to introduce two further transition times as follows.
Recall that $x_{\ell,c_{2,i}}$ equals the abscissa of the intersections between $\mc{E}^{-1}(c)$ and the outer boundary of ${\mathcal A}_{i,III}$. From this, it is easily seen that the transition time between
the straight lines $x=x_{\ell,c_{2,i}}$ and $x=x_{r,c_{1,i+1}}$ passing through $(-\sqrt{{2c}/{\lambda}},0)$
along ${\mathcal E}^{-1}(c)$ is given by
\begin{equation}\label{eq-S-II-I>0outin}
{\mathcal S}_{III\to I}(c):=
\int_{-\sqrt{2c/\lambda}}^{x_{\ell,c_{2,i}}}
\frac{dx}{-(h\circ H_{-}^{-1})\left(c- \frac{\lambda}{2} x^2\right)}+
\int_{-\sqrt{{2c}/{\lambda}}}^{x_{r,c_{1,i+1}}}
\frac{dx}{(h\circ H_{+}^{-1})\left(c- \frac{\lambda}{2} x^2\right)}.
\end{equation}
Similarly, since $x_{r,c_{2,i}}$ provides the abscissa of the intersections
between $\mc{E}^{-1}(c)$ and the outer boundary of ${\mathcal A}_{i,I}$, the transition time between the straight lines $x=x_{r,c_{2,i}}$ and $x=x_{\ell,c_{1,i+1}}$ passing
through the point $(\sqrt{{2c}/{\lambda}},0)$ along $\mc{E}^{-1}(c)$ is given by
\begin{equation}\label{eq-S-I-III>0outin}
{\mathcal S}_{I\to III}(c):= \int_{x_{r,c_{2,i}}}^{\sqrt{2c/\lambda}}
\frac{dx}{(h\circ H_{+}^{-1})\left(c- \frac{\lambda}{2}x^2\right)}+
\int_{x_{\ell,c_{1,i+1}}}^{\sqrt{{2c}/{\lambda}}}
\frac{dx}{-(h\circ H_{-}^{-1})\left(c- \frac{\lambda}{2}
x^2\right)}.
\end{equation}
This concludes our analysis in this case.

\section{Linking dynamics in general Sturm--Liouville BVP's}\label{section-3}

In this section we show some applications of the preceding results to the existence
of multiple solutions to \eqref{i.9} satisfying prescribed boundary conditions. In our analysis, we will restrict ourselves to consider regular Sturm--Liouville boundary value problems.
\par
From a geometrical point of view, the homogeneous boundary conditions of Sturm--Liouville type can be described by, e.g.,  assigning two lines $r_0$ and $r_L$ passing through the origin and looking for solutions of \eqref{i.9} such that  $(x(0),y(0))\in r_0$ and $(x(L),y(L))\in r_L$, according to the particular boundary value problem considered. In this setting, also the case of non-homogeneous boundary conditions can be treated by considering
two affine lines not passing through the origin provided they satisfy appropriate compatibility conditions. Namely, the lines $r_0$ and $r_L$  must cross the annuli ${\mathcal A}_0$ and ${\mathcal A}_m$, respectively. More generally, we can also study problems
associated with nonlinear boundary conditions, by taking as $r_0$ and $r_L$
two general curves crossing ${\mathcal A}_0$ and ${\mathcal A}_m$, respectively.
See Papini and Zanolin \cite{PZ-2000} and the references therein for further
details on generalized boundary conditions of Sturm--Liouville type.
\par
More precisely, we will study the boundary value problem
\begin{equation}
\label{vi.1}
\left\{
\begin{array}{ll}
x' = h(y),\\
y' = -\lambda x - a(t)g(x),\\
x(0) \in r_0,\; \; x(L)\in r_L,
\end{array}
\right.
\end{equation}
where $r_0$ is an arc contained in ${\mathcal A}_{0}$ and
$r_L$ is an arc contained in ${\mathcal A}_{m}$,
\textit{both arcs connecting the
inner with the outer boundary of the corresponding annulus}, while the weight function $a(t)$ satisfies the general properties of Section \ref{section-2}  in the interval $[0,L]$, as well as the functions $h=\phi^{-1}: (\varrho_-,\varrho_+)\to {\mathbb R}$ and $g:{\mathbb R}\to {\mathbb R}$. Moreover, we suppose the annular regions considered to satisfy the compatibility conditions with respect to the domain
${\mathbb R} \times (\varrho_-,\varrho_+)$ and, when $\lambda >0$,
with respect to \eqref{v.10}, as illustrated in
Figure \ref{fig-03}.
\par
In order to analyze \eqref{vi.1} under the above assumptions we will proceed as follows. First, they are studied, for every $i\in\{0,1,\ldots,m\}$, the linking dynamics between the Poincaré maps $\Phi_i$ and $\Psi_i$ associated to $(\mathcal{N}_{i,\l})$ and $(\mathcal{L}_{\l})$, respectively. Then, when these linking dynamics hold, by applying a combinatorial scheme, there will be obtained a number of multiplicity results for solutions having a prescribed nodal behavior. As the sign of $\lambda$ changes completely the structure of $(\mathcal{N}_{i,\l})$ and $(\mathcal{L}_{\l})$, as well as  the linking between them, the cases $\l\leq0$ and $\l>0$ will be dealt with throughout separately.

\subsection{The case $\l\leq0$}

We start analyzing the case $\l\leq0$ by delivering a result in the vain of Proposition \ref{pr4.1}. Namely, by using the times of transition calculated in Section \ref{sec-lin} for the linear problem $(\mathcal{L}_{\l})$, linking two non-linear problems $(\mathcal{N}_{j,\l})$ for $j\in\{i,i+1\}$, we get a property on the arc transformation of the composite map $\Psi_i\circ\Phi_i$ as well as the nodal properties of such arcs.

\begin{proposition}
\label{pr6.1}
Given the cyclic clockwise ordering $IV \prec III \prec II \prec I\prec IV$ we fix $\k=I$ or $III$,
pick $d, e\in \{c_{1,i},c_{2,i}\}$ with $d\not=e$, and assume that there are two non-negative integers
$0\leq\alpha_{i} < \beta_{i}$ satisfying the twist condition \eqref{iv.1}. Suppose, in addition, that the transition times satisfy either
\begin{enumerate}
\item[\rm (1)] ${\mc S}_{i,\k-1\to \k} \leq \varsigma_{i}$ and ${\mc S}_{i,\k+1\to \k+2} \leq \varsigma_{i}$, or

\item[\rm (2)] ${\mc S}_{i,\k+1\to \k} \leq \varsigma_{i}$ and ${\mc S}_{i,\k+3\to \k+2} \leq \varsigma_{i}$.
\end{enumerate}
Then,
\begin{equation}
\label{vi.2}
\Psi_i\circ\Phi_i:\mathcal {A}^{\rm energy}_{i,\k}\overset{\b_i-\a_i}\stretchx\,\mathcal {A}^{\rm energy}_{i+1,\k}\cup \mathcal {A}^{\rm energy}_{i+1,\k+2}.
\end{equation}
Furthermore, depending on whether the assumption {\rm (1)} or {\rm (2)} occurs, for every arc $\gamma\subset\mathcal{A}_{i,\k}^{\rm energy}$ joining the opposite sides of $\mathcal{A}_{i,\k}^{-}$, the solutions of $({\mathcal N}_{i,\lambda})$ with initial value on a sub-arc $\gamma_{j,\k}\subset K_{j,\k}$ of $\gamma$ and ending at $\mathcal{A}_{i,\tilde\k}$ have exactly, for every $j=1,2,\ldots,\beta_i-\alpha_i$,

\begin{equation*}
(1)\left\{
\begin{array}{llll}
(a)\;\tilde{\k}=\k&\left\{
\begin{array}{ll}
2(\alpha_i+j)-1&\;\hbox{zeroes in the x-component in}\;\;(0,\tau_i)\equiv(t_i,s_i)\\[2pt]
1&\;\hbox{zeroes in the x-component in}\;\;[\tau_i,\tau_i+\varsigma_i]\equiv[t_i,\varsigma_i]
\\[10pt]
2(\alpha_i+j)&\;\hbox{zeroes in the y-component in}\;\;[0,\tau_i]\equiv[t_i,s_i]\\[2pt]
0&\;\hbox{zeroes in the y-component in}\;\;[\tau_i,\tau_i+\varsigma_i]\equiv[t_i,\varsigma_i]
\end{array}
\right.
\\[5pt]
(b)\;\tilde{\k}=\k+2&\left\{
\begin{array}{ll}
2(\alpha_i+j)&\;\hbox{zeroes in the x-component in}\;\;(0,\tau_i)\equiv(t_i,s_i)\\[2pt]
1&\;\hbox{zeroes in the x-component in}\;\;[\tau_i,\tau_i+\varsigma_i]\equiv[t_i,\varsigma_i]
\\[10pt]
2(\alpha_i+j)+1&\;\hbox{zeroes in the y-component in}\;\;[0,\tau_i]\equiv[t_i,s_i]\\[2pt]
0&\;\hbox{zeroes in the y-component in}\;\;[\tau_i,\tau_i+\varsigma_i]\equiv[t_i,\varsigma_i]
\end{array}
\right.
\end{array}
\right.
\end{equation*}
and
\begin{equation*}
(2)\left\{
\begin{array}{llll}
(a)\;\tilde{\k}=\k&\left\{
\begin{array}{ll}
2(\alpha_i+j)&\;\hbox{zeroes in the x-component in}\;\;(0,\tau_i)\equiv(t_i,s_i)\\[2pt]
0&\;\hbox{zeroes in the x-component in}\;\;[\tau_i,\tau_i+\varsigma_i]\equiv[t_i,\varsigma_i]
\\[10pt]
2(\alpha_i+j)+1&\;\hbox{zeroes in the y-component in}\;\;[0,\tau_i]\equiv[t_i,s_i]\\[2pt]
1&\;\hbox{zeroes in the y-component in}\;\;[\tau_i,\tau_i+\varsigma_i]\equiv[t_i,\varsigma_i]
\end{array}
\right.
\\[5pt]
(b)\;\tilde{\k}=\k+2&\left\{
\begin{array}{ll}
2(\alpha_i+j)-1&\;\hbox{zeroes in the x-component in}\;\;(0,\tau_i)\equiv(t_i,s_i)\\[2pt]
0&\;\hbox{zeroes in the x-component in}\;\;[\tau_i,\tau_i+\varsigma_i]\equiv[t_i,\varsigma_i]
\\[10pt]
2(\alpha_i+j)&\;\hbox{zeroes in the y-component in}\;\;[0,\tau_i]\equiv[t_i,s_i]\\[2pt]
1&\;\hbox{zeroes in the y-component in}\;\;[\tau_i,\tau_i+\varsigma_i]\equiv[t_i,\varsigma_i]
\end{array}
\right.
\end{array}
\right.
\end{equation*}
\end{proposition}

By the analysis already done in Section \ref{sec-lin}, it is clear that the case (1) can occurs for $\l\leq 0$ with ${\mc S}_{i,j\to j+1}$ given by \eqref{v.2}, \eqref{v.4}, \eqref{v.6}, or \eqref{v.7}, respectively. However, the case (2) can only occurs for $\l<0$.

\begin{proof}
We focus our efforts in proving the case (1) in the sub-case $\k=I$, since the proofs of the
remaining cases can be completed with a rather similar argument. According to Proposition \ref{pr4.1}, there are $\beta_i-\a_i$ pairwise disjoint compact sets $K_{1,II},\dots,K_{\b_i-\a_i,II}$ such that each arc $\gamma$ in  ${\mathcal A}_{i,I}$ connecting the level lines $\{{\mathcal H}=c_{i,1}\}$ and $\{{\mathcal H}=c_{i,2}\}$ has $\ell_i$ sub-arcs $\gamma_{j,II}\subset K_{j,II}$ such that $\Phi_i(\gamma_{j,II})$ is an arc contained in ${\mathcal A}_{i,II}$ which connects the $x$-axis with the $y$-axis.
\par
We first discuss the case $\lambda < 0$. For each of the arcs $\Phi_i(\gamma_{j,II})$ we consider
a suitable sub-arc $\Phi_i(\gamma_j)$, where $\gamma_{j}$ is a compact sub-arc of $\gamma_{j,II}$ that intersects in the second quadrant at a unique point $w_1$ the level line ${\mathcal E} = 0$ and
at a unique point $w_2$ the level line ${\mathcal E}=  H(y_{+,i})$. Furthermore, we suppose that $\Phi_i(\gamma_{j})$ is contained in the region
$$
  \{(x,y): 0 \leq {\mathcal E}(x,y) \leq H(y_{+,i})\}.
$$
By definition, $w_1$ lies on the component of the stable manifold of the origin of
$({\mathcal L}_{\lambda})$ in the second quadrant. Therefore, $\Psi_i(w_1)$ is moved toward the origin, but it remains in the second quadrant. On the other hand, the condition ${\mc S}_{i,II\to I} \leq \varsigma_{i}$ implies that the point $w_2$ moves under the action of $\Psi_i$  along the level line ${\mathcal E} = H(y_{+,i})$ to the first quadrant and outside the annulus ${\mathcal A}_{i+1}$. To be more precise, one should prevent the possibility that there is a blow up at some time $t\leq \varsigma_i$ in the sense that the $y$-component of the solutions achieves the value $\varrho_{+}$,
where we might abandon the domain of
definition of the differential equation. This, however, is not a difficulty as we are interested only on the pieces of $\Phi_i(\gamma_{j})$ which stay in the annulus ${\mathcal A}_{i+1}$ after applying $\Psi_i$.
For this reason, we can just repeat the same argument by restricting ourselves to a further
sub-arc of the form $\Phi_i(\tilde{\gamma}_j)$ lying between ${\mathcal E}=0$ and ${\mathcal E} = c$ for a suitable $c \in (0,y_{+,i}))$.  Accordingly, from now on, we identify $\tilde{\gamma}_j$ with $\gamma_j$ and assume, without loss of generality, that $\Psi_i$ is well defined on  the arc $\Phi_i(\gamma_j)$. In this manner, we can conclude that $\Psi_i(\Phi_i(\gamma_j))$ is an arc connecting $\Psi_i(w_1)$, which is in the region $\{{\mathcal H}_{i+1,\lambda}< c_{1,{i+1}}\}$ with $x<0$,  to $\Psi_i(w_2)$, which is in the region $\{{\mathcal H}_{i+1,\lambda}\geq c_{2,{i+1}}\}$ with $x>0$. Moreover, all the points of $\Psi_i(\Phi_i(\gamma_j))$ are contained in the set $0\leq {\mathcal E} \leq H(y_{+,i})$ with $y>0$.
\par
Next, we consider the case when  $\lambda = 0$. As before, for each of the arcs $\Phi_i(\gamma_{j,II})$ we consider a suitable sub-arc $\Phi_i(\gamma_j)$, where $\gamma_{j}$ is a compact sub-arc of $\gamma_{j,II}$, that intersects at a unique point $w_1$ the $x$-axis, which is the level line ${\mathcal E} = 0$, and at a unique point $w_2$ the level line ${\mathcal E} = H(y_{+,i})$, which is the line $y =y_{+,i}$. Recalling that $({\mathcal L}_{0})$ reduces to $x'=h(y)$ and $y'=0$, we have that  $\Psi_i(w_1)=w_1$. On the other hand, the condition  ${\mc S}_{i,II\to I} \leq \varsigma_{i}$ implies that $w_2$ moves under the action of $\Psi_i$ along the line $y =y_{+,i}$ to the first quadrant outside the annulus ${\mathcal A}_{i+1}$ (see Figure \ref{fig-02}).  Hence, $\Psi_i(\Phi_i(\gamma_j))$ is an arc connecting $w_1$ to $\Psi_i(w_2)$ (which is in the region $\{{\mathcal H}_{i+1,\lambda}\geq  c_{2,{i+1}}\}$ with $x>0$). Moreover, all the points of
$\Psi_i(\Phi_i(\gamma_j))$ are contained in the set $0\leq y\leq y_{+,i}$.
\par
At this point, we have already proven that, in either case $\lambda<0$ and $\lambda=0$, there is a sub-arc $\gamma_j$ of $\gamma_{j,II}$ such that $\Psi_i(\Phi_i(\gamma_j))$ crosses the annulus ${\mathcal A}_{i+1}$ in the fist quadrant. Passing to a further sub-arc $\hat{\gamma}_j$ of $\gamma_j$ contained in $K'_j$, we have that $\Phi_i(\hat{\gamma}_j)$ is a sub-arc of $\Phi_i(\gamma_j)$ contained in the set
$$
  \mc{A}_{i,II}\cap \{ 0\leq \mc{E}\leq H(y_{+,i}) \}
$$
such that $(\Psi_j\circ  \Phi_j)(\hat{\gamma}_j)$ is an arc contained in ${\mathcal A}_{i+1,I}$ linking the two level lines     $\mc{H}^{-1}(c_{i+1,1})$ and $\mc{H}^{-1}(c_{i+1,2})$. This proves \eqref{vi.2} and shows how the corresponding solutions in the interval $[\tau_i,\tau_i+\varsigma_i]$ have exactly  one (interior) zero in the $x$-component, while the $y$-component remains positive.
\par
The remaining nodal properties of the solutions come directly from Proposition
\ref{pr4.1}. Precisely, for every $j=1,\dots,\beta_i-\alpha_i$, we have the following classification, where we are taking $\k\in\{I,III\}$:
\begin{itemize}
\item[$(1.a)$] \underline{For the transition between ${\mathcal A}_{i,\k}$ and ${\mathcal A}_{i+1,\k}$}:\\[3pt]
There exists a solution with $2(\alpha_i + j)$ zeroes in the $x$-component in the interval $(t_i,t_{i+1})$. Moreover, this solution has $2(\alpha_i + j)$ zeroes in the $x'/y$-component in the interval $[t_i,t_{i+1}]$. Thus, the number of zeroes of the $x$ and $y$ components  of the
solution, indicated over the arrows,  follow the next patterns, respectively:
\[
{\mathcal A}_{i,\k}\overset{2(\alpha_i + j)-1}\longrightarrow {\mathcal A}_{i,\k-1}\overset{1}\longrightarrow{\mathcal A}_{i+1,\k}\quad\hbox{and}\quad {\mathcal A}_{i,\k}\overset{2(\alpha_i + j)}\longrightarrow {\mathcal A}_{i,\k-1}\overset{0}\longrightarrow{\mathcal A}_{i+1,\k}
\]
\item[$(1.b)$]
\underline{For the transition between ${\mathcal A}_{i,\k}$ and ${\mathcal A}_{i+1,\k+2}$}:
\\[3pt]
There exists a solution with $2(\alpha_i + j)+1$ zeroes in the $x$-component in the interval $(t_i,t_{i+1})$, and this solution has $2(\alpha_i + j)+1$ zeroes in the $x'/y$-component in the interval $[t_i,t_{i+1}]$. Thus,  the number of zeroes of the $x$ and $y$ components  of the
solution, indicated over the arrows,  follow the next patterns, respectively:
\[
{\mathcal A}_{i,\k}\overset{2(\alpha_i + j)}\longrightarrow {\mathcal A}_{i,\k+1}\overset{1}\longrightarrow{\mathcal A}_{i+1,\k+2}\quad\hbox{and}\quad {\mathcal A}_{i,\k}\overset{2(\alpha_i + j)+1}\longrightarrow {\mathcal A}_{i,\k+1}\overset{0}\longrightarrow{\mathcal A}_{i+1,\k+2}
\]
\item[$(2.a)$]
\underline{For the transition between ${\mathcal A}_{i,\k}$ and ${\mathcal A}_{i+1,\k}$}:
\\[3pt]
There exists a solution with $2(\alpha_i + j)$ zeroes in the $x$-component in the interval $(t_i,t_{i+1})$. Moreover, this solution can be chosen to have $2(\alpha_i + j)+2$ zeroes in the $x'/y$-component in the interval $[t_i,t_{i+1}]$. Hence, the number of zeroes of the $x$ and $y$ components  of the solution, indicated over the arrows,  follow the next patterns, respectively:
\[
{\mathcal A}_{i,\k}\overset{2(\alpha_i + j)}\longrightarrow {\mathcal A}_{i,\k+1}\overset{0}\longrightarrow{\mathcal A}_{i+1,\k}\quad\hbox{and}\quad {\mathcal A}_{i,\k}\overset{2(\alpha_i + j)+1}\longrightarrow {\mathcal A}_{i,\k+1}\overset{1}\longrightarrow{\mathcal A}_{i+1,\k}
\]
\hspace{2mm}\item[$(2.b)$]
\underline{For the transition between ${\mathcal A}_{i,\k}$ and ${\mathcal A}_{i+1,\k+2}$}:
\\[3pt]
There exists a solution with $2(\alpha_i + j)-1$ zeroes in the $x$-component in the interval $(t_i,t_{i+1})$ with $2(\alpha_i + j)+1$ zeroes in the $x'/y$-component in the interval $[t_i,t_{i+1}]$. Hence, the number of zeroes of the $x$ and $y$ components  of the solution, indicated over the arrows,  follow the next patterns, respectively:
\[
{\mathcal A}_{i,\k}\overset{2(\alpha_i + j)-1}\longrightarrow {\mathcal A}_{i,\k-1}\overset{0}\longrightarrow{\mathcal A}_{i+1,\k+2}\quad\hbox{and}\quad {\mathcal A}_{i,\k}\overset{2(\alpha_i + j)}\longrightarrow {\mathcal A}_{i,\k-1}\overset{1}\longrightarrow{\mathcal A}_{i+1,\k+2}
\]
\end{itemize}
This concludes the proof.
\end{proof}

Figure \ref{fig-04} (left panel)
illustrates the admissible transitions in the context of Proposition \ref{pr6.1} (above) when $\lambda\leq0$ and Propositions \ref{pr6.2} and \ref{pr6.3} (further in this section) when $\lambda>0$.
If $\lambda \leq 0$, the intermediate blocks are given by ${\mathcal
    B}_{i,II}={\mathcal A}_{i,II}$ and ${\mathcal B}_{i,IV}={\mathcal
    A}_{i,IV}$, while, for $\lambda >0$, we have that
$$
{\mathcal B}_{i,II}={\mathcal A}_{i,II}\cup {\mathcal A}_{i,III},\quad
{\mathcal B}_{i,IV}={\mathcal A}_{i,IV}\cup {\mathcal A}_{i,I}.
$$
When $\lambda \geq 0$ we have only two connections available for
$\Psi_i$, namely from ${\mathcal B}_{i,II}$ to ${\mathcal A}_{i+1,I}$
and from ${\mathcal B}_{i,IV}$ to ${\mathcal A}_{i+1,III}$ (the black
arrows), while, for $\lambda <0$ all four connections for $\Psi_i$
are allowed (black and red arrows). In the horizontal and diagonal transitions,  $\b_i-\a_i$ is the associated crossing
number of the linking $\Phi_i$. Concerning the maps $\Psi_i$, they have $1$ as crossing
number, except if $\lambda >0$, where the crossing number
could equal $\kappa_i$ under the assumptions described in Remark \ref{re6.3}.

\begin{figure}[htbp]
\centering
\includegraphics[scale=1]{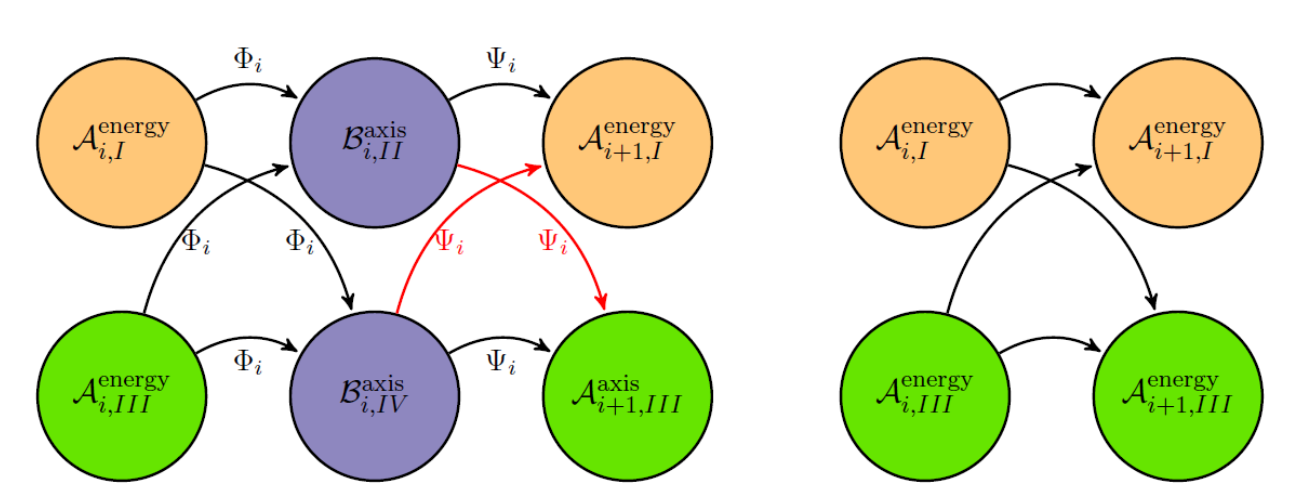}
\caption{\small{Admissible transitions of of $\Psi_i\circ\Phi_i$ (right panel). In the left
panel we have put in evidence the intermediate transitions across the sets
$\mathcal B^{\rm axis}$.}}
\label{fig-04}
\end{figure}

Now, by applying recursively Proposition \ref{pr6.1} we get the next multiplicity results of nodal solutions for the boundary value problem \eqref{vi.1}. By simplicity, the case $\l=0$ is analyzed at the first place. The case $\l<0$ holds in a very similar way.

\begin{theorem}
\label{th6.1} Suppose that $\lambda=0$ and that either $r_0\in {\mathcal A}_{0,I}$, or
$r_0\in {\mathcal A}_{0,III}$ and either $r_L\in {\mathcal A}_{m,II}$ or
$r_L\in {\mathcal A}_{m,IV}$. Assume, in addition, that, for every $i\in\{0,...,m\}$ and $\kappa\in\{I,III\}$, the twist conditions \eqref{iv.1} are satisfied together with
$$
   {\mc S}_{i,\k-1\to \k}\leq \varsigma_i \quad\text{and } \;\;
{\mc S}_{i,\k+1\to \k+2}\leq \varsigma_i\quad\hbox{for all}\;\; i\in\{0,\dots, m-1\}.
$$
Then, for each choice of $r_0$ and $r_L$, the problem \eqref{vi.1} has, at least,
\[
2^{m}\prod_{i=0}^m( \b_i-\a_i)
\]
nodal solutions. Moreover, each of these solutions can be classified by assigning it
a precise counter describing its nodal properties on each of the intervals $(t_i,s_i)$ and $(s_i,t_{i+1})$, according to any admissible itinerary
chosen in the Markov-type diagram of Figure \ref{fig-04} (left panel) at $\lambda=0$.
\end{theorem}
\begin{proof}
The proof employs a shooting type argument. Its goal is to prove the
existence of points $P_0\in r_0$ such that ${\mathscr P}_L(P_0)\in r_L$, where ${\mathscr P}_L$ denotes the Poincar\'{e} map in $[0,L]$. To do this, we factorize ${\mathscr P}_L$ as
\begin{equation}
\label{vi.3}
{\mathscr P}_L = \Phi_m\circ {\mathscr Q}_{m},\quad \text{where }\;
{\mathscr Q}_{m}:=(\Psi_{m-1}\circ \Phi_{m-1})\circ
\dots \circ (\Psi_{0}\circ \Phi_{0}).
\end{equation}
This follows by applying Proposition \ref{pr6.1} (1), which covers the case $\lambda=0$.
The possible transitions are those illustrated in Figure \ref{fig-04}.
\par
Precisely,  if $\gamma=r_0\in {\mathcal A}_{0,I}$, by \eqref{vi.2}, we can get $\b_0-\a_0$ sub-arcs of $\gamma=r_0$ whose images through $\Psi_{0}\circ \Phi_{0}$ provide us with arcs of ${\mathcal A}_{1,I}$ linking the two components $\mathcal A^{-}_{1,I}$ of $\mathcal A^{\rm energy}_{1,I}$, plus another $\b_0-\a_0$ sub-arcs of $\gamma=r_0$ whose images through $\Psi_{0}\circ \Phi_{0}$
are arcs of ${\mathcal A}_{1,III}$ linking the two components $\mathcal A^{-}_{1,III}$
of $\mathcal A^{\rm energy}_{1,III}$.
\par
Now, using again the same property, it is apparent that each of these arcs in
$\mathcal A^{\rm energy}_{1,I}\cup \mathcal A^{\rm energy}_{1,III}$
generates another $\b_1-\a_1$ sub-arcs whose images through $\Psi_{1}\circ \Phi_{1}$
are arcs of ${\mathcal A}_{2,I}$ linking the two components $\mathcal A^{-}_{2,I}$
of $\mathcal A^{\rm energy}_{2,I}$ as well as another $\b_1-\a_1$ sub-arcs whose  images through $\Psi_{1}\circ \Phi_{1}$ are arcs of  ${\mathcal A}_{2,III}$ linking the two components $\mathcal A^{-}_{2,III}$ of $\mathcal A^{\rm energy}_{2,III}$.
\par
Set
\begin{equation}
\label{vi.4}
M:=(\b_0-\a_0)\prod_{i=1}^{m-1}2(\b_i-\a_i)
\end{equation}
and observe that, arguing inductively, there are $M$ sub-arcs of $r_0$
whose images through ${\mathscr Q}_{m}$ are pairwise disjoint sub-arcs
of ${\mathcal A}_{m-1,I}$ linking the two components $\mathcal A^{-}_{m-1,I}$
of $\mathcal A^{\rm energy}_{m-1,I}$, as well as $M$ sub-arcs of $r_0$
whose images through ${\mathscr Q}_{m}$ are pairwise disjoint sub-arcs of
${\mathcal A}_{m-1,III}$ linking the two components $\mathcal A^{-}_{m-1,III}$
of $\mathcal A^{\rm energy}_{m-1,III}$.
\par
Now, define
$$
   \tilde{\gamma}_{m-1}:={\mathscr Q}_{m}(\tilde{\gamma}_0),
$$
where $\tilde{\gamma}_0$ is anyone of the previous $2M$ sub-arcs of $r_0$. Recall that  $M$ among these  $\tilde{\gamma}_{m-1}$ arcs lie in ${\mathcal A}_{m-1,I}$ and another $M$ are contained in $\mc{A}_{m-1,III}$. As a last step, we apply Proposition \ref{pr4.1} to $\Phi_m$
and observe that, for each of the $M$ arcs $\tilde{\gamma}_{m-1}$ contained in ${\mathcal A}_{m-1,I}$, there are $\b_m-\a_m$ sub-arcs whose images
are arcs of ${\mathcal A}_{m,II}$ linking the two components $\mathcal A^{-}_{m,II}$ of $\mathcal A^{\rm axis}_{m,II}$ as well as another $\b_m-\a_m$ sub-arcs whose images
are arcs of ${\mathcal A}_{m,IV}$ linking the two components $\mathcal A^{-}_{m,IV}$ of $\mathcal A^{\rm axis}_{m,IV}$. And this can be repeated for each of the $M$
arcs $\tilde{\gamma}_{m-1}$ of ${\mathcal A}_{m-1,III}$. Actually, each of these
arcs generates two families of $\b_m-\a_m$ sub-arcs of ${\mathcal A}_{m}$ that cross
the second and fourth quadrants, respectively. Thus, coming back to the first quadrant with
${\mathscr Q}_{m}^{-1}$, we obtain a total of
$$
   M 2(\b_m-\a_m)=(\b_0-\a_0)\prod_{i=1}^{m}2(\b_i-a_i)=2^m\prod_{i=0}^{m}(\b_i-a_i)
$$
sub-arcs
$\hat{\gamma}_0$ of $r_0$, which are simultaneously sub-arcs of the $\tilde{\gamma}_0$'s, whose images through ${\mathscr P}_L$ are pairwise disjoint sub-arcs
of ${\mathcal A}_{m,II}$ linking the two components $\mathcal A^{-}_{m,II}$ of ${\mathcal A}^{\rm axis}_{m,II}$. Naturally, there are another $ M 2(\b_m-\a_m)$ (different) sub-arcs of $r_0$
whose images through ${\mathscr  P}_L$ are pairwise disjoint sub-arcs
of ${\mathcal A}_{m,IV}$ linking the components of $\mc{A}_{m,IV}^-$ of ${\mathcal A}^{\rm axis}_{m,IV}$.
Each arc $\hat{\gamma}_m:={\mathscr P}_L(\hat{\gamma}_0)$ is continuous, it lies in ${\mathcal A}_{m,II}$
(respectively in ${\mathcal A}_{m,IV}$), and it links the two opposite sides of the corresponding
rectangular region. Moreover, also $r_L$ is a continuum of ${\mathcal A}_{m,II}$, or ${\mathcal A}_{m,IV}$, linking the other two opposite sides. Thus, by a well-known result in plane topology
(see, for instance, Lemma 3 of  Muldowney and Willett \cite{MuWi-1974})), we find that, for each
pairwise disjoint $M2(\b_m-\a_m)$ arcs like $\hat{\gamma}_m$, there exists, at least, an intersection point between $\hat{\gamma}_m$ and $r_L$.
\par
Repeating the same argument with $r_0\in {\mathcal A}_{0,III}$, we can find another $M2(\b_m-\a_m)$ intersection points between the arcs defined as $\hat{\gamma}_m$ and $r_L$. Therefore, for every alternative in the choice of $r_0\in {\mathcal A}_{0,I}
\cup {\mathcal A}_{0,III}$ and $r_L\in {\mathcal A}_{0,II}
\cup {\mathcal A}_{0,IV}$, the problem \eqref{vi.1} has at least
$$
2^m\prod_{i=0}^{m}(\b_i-a_i)
$$
solutions.
\par
Finally, by the nodal classification (1.a) and (1.b) done in Proposition \ref{pr6.1}, given any itinerary of admissible transitions
$$
  \Psi_{i}\circ \Phi_{i}:\mathcal A^{\rm energy}_{i,\k_1}\overset{\b_i-\a_i}\stretchx\,\mathcal {A}^{\rm energy}_{i+1,\k_2},
$$
with $\k_1,\k_2\in\{I,III\}$, we can provide, rather similarly,  a precise description of the nodal properties of all the given solutions. This ends the proof.
\end{proof}

\begin{remark}
\label{re6.1}
\rm According to the scheme diagram of Figure \ref{fig-04} (left panel), we can consider some simplified twist
and cross conditions in order to cover only the horizontal transitions. For instance, if
we assume that $r_{0}\in {\mathcal A}_{0,I}$ and $r_{L}\in {\mathcal A}_{m,II}$ and that
$$
  {\mc S}_{i,II\to I}\leq \varsigma_i \quad \hbox{for all}\;\; i\in\{0,\dots, m-1\}
$$
is the only cross condition which is satisfied,
we can just move horizontally in the upper part of the diagram.
In this case we could consider a simplified twist condition as the one
in Remark \ref{re4.1}  to get up to $\prod_{i=0}^m(\b_i-\a_i)$ solutions
with their corresponding nodal properties.
Note that the coefficient $2^m$ appearing in Theorem \ref{th6.1}
comes from the extra choice of having two options at each $i$-th step,
concerning the
transition from ${\mathcal A}_{i,I}$ to either ${\mathcal A}_{i,I}$ or
${\mathcal A}_{i,III}$, as well as the
transition from ${\mathcal A}_{i,III}$ to either ${\mathcal A}_{i,III}$ or
${\mathcal A}_{i,I}$.
\end{remark}

We conclude our analysis in this section by analyzing the case $\lambda <0$.
In this situation, in view of the diagrams of Figure \ref{fig-04}, it is apparent that
the complexity may further increase, because, at each step we have
two more possibilities of choice for the transitions. Indeed, all the available transitions
in the diagram of Figure \ref{fig-04} (left panel) are permitted, so that, at each step,
we can apply, in two different ways, the path-stretching property from
$\mathcal A^{\rm energy}_{i,k_1}$ to
$\mathcal A^{\rm energy}_{i,k_2}$ for each choice
of $k_1,k_2\in\{I,III\}$. As the associated crossing number of the map $\Phi_i$ is
$\b_i-\a_i$, we conclude that
$$
   \Psi\circ\Phi_i:\mathcal A^{\rm energy}_{i,k_1}\overset{2(\b_i-\a_i)}\stretchx\,
\mathcal A^{\rm energy}_{i,k_2}\quad\hbox{for all}\;\; i\in\{0,\dots,m-1\}
$$
for every choice of $k_1, k_2\in\{I,III\}$.  Thus, we can repeat the proof of
Theorem \ref{th6.1}, using the crossing number $2(\b_i-\a_i)$, instead of $\b_i-\a_i$,
at each intermediate transition. For the last step of the proof,
we will apply only $\Phi_i$ (following the same argument of
Theorem \ref{th6.1}). The analysis of the nodal properties of the solutions follow directly from Proposition \ref{pr6.1}. With these preliminary observations, the following result holds.

\begin{theorem}
\label{th6.2} Suppose that $\lambda<0$ and that either $r_0\in {\mathcal A}_{0,I}$,  or
$r_0\in {\mathcal A}_{0,III}$, and either $r_L\in {\mathcal A}_{m,II}$, or
$r_L\in {\mathcal A}_{m,IV}$.    Assume, in addition,  that, for every $i\in\{0,\dots, m\}$ and $\k\in\{I,III\}$,
the twist conditions    \eqref{iv.1} are satisfied  together with
$$
   {\mc S}_{i,\cdot}\leq \varsigma_i\quad \hbox{for all}\;\; i\in\{0, \dots, m-1\},
$$
where $S_{i,\cdot}$ stand for  the transition times defined in    \eqref{v.2}, \eqref{v.4}, \eqref{v.5} and \eqref{v.6}.  Then, for each admissible choice of $r_0$ and $r_L$,
the problem \eqref{vi.1} has, at least,
$$
   4^{m}\prod_{i=0}^{m}(\b_i-\a_i)
$$
nodal solutions. Moreover, each of these solutions can be classified by assigning it a counter
describing its nodal properties on each of the intervals  $(t_i,s_i)$ and $(s_i,t_{i+1})$,
for any admissible itinerary in the  Markov-type of diagram sketched in Figure \ref{fig-04}
(left panel) for  $\l<0$.
\end{theorem}
\begin{proof}
The proof follows, almost step by step, the proof of Theorem \ref{th6.1}. The main difference being
that in the Markov-type transition diagram of Figure \ref{fig-04}
(right panel) each connection has an associated crossing number $2(\b_i-\a_i)$.
Indeed, if we consider the transition between  ${\mathcal A}_{i,I}$  and ${\mathcal A}_{i+1,I}$, we can move from ${\mathcal A}_{i,I}$ to
${\mathcal A}_{i,II}$ with $\Phi_i$, by using $\b_i-\a_i$ sub-arcs, and then from ${\mathcal A}_{i,II}$ to ${\mathcal A}_{i+1,I}$ with $\Psi_i$, with one zero in $x$ and no zero in $y$, or, alternatively,
we can move from ${\mathcal A}_{i,I}$ to ${\mathcal A}_{i,IV}$ with $\Phi_i$, using $\b_i-\a_i$ sub-arcs, and then from ${\mathcal A}_{i,IV}$ to ${\mathcal A}_{i+1,I}$ with $\Psi_i$, with one zero in $y$ and no zero in $x$.
The remaining transitions of Figure \ref{fig-04} (right panel) can be described in a similar manner.
Setting
$$
  M:=2 (\b_0-\a_0)\prod_{i=1}^{m-1}4(\b_i-\a_i)
$$
and arguing recursively,  at the end of the day we will get $M$ sub-arcs $\tilde{\gamma}_0$ of $r_0$
whose images through ${\mathscr Q}_{m}$ are pairwise disjoint sub-arcs of ${\mathcal A}_{m-1,I}$, either linking the two components of $\mathcal A^{-}_{m-1,I}$ in $\mathcal A^{\rm energy}_{m-1,I}$, or
the two components of ${\mathcal A}^{-}_{m-1,III}$  in $\mathcal A^{\rm energy}_{m-1,III}$.
As according to the last step we are adding $2(\b_m-\a_m)$ arcs to each of the arcs of the previous step, the proof of Theorem \ref{th6.1}, with some minor modifications, can be updated to complete the proof of this one.
\par
Finally, for every $i\in\{0,\ldots,m\}$, the nodal properties of the different transitions between $\mathcal{A}^{\rm energy}_{i,\k_1}$ and  $\mathcal{A}^{\rm energy}_{i+1,\k_2}$, with $\k_1,\k_2\in\{I,III\}$, can be completely explained by following (1.a), (1.b), (2.a) and (2.b) of Proposition \ref{pr6.1}.
\end{proof}

\begin{remark}
\label{re6.2}
\rm To get the results of this section, we have chosen the initial and final arcs, $r_0$ and $r_L$,
to be contained in a certain quadrant. By imposing the  more restrictive twist condition
\eqref{iv.6} for $i\in\{0,m\}$, one can take $r_0$ and $r_L$ in an arbitrary quadrant. The condition
\eqref{iv.6} can be easily adapted according to the quadrant where the arcs $r_0$ and $r_L$ are located.
For instance, if $r_0\in {\mathcal A}_{0,II}$, then \eqref{iv.6} becomes
$$
    {\mathcal T}_{0,I}(d) + \alpha_{0} {\mathcal T}_{0}(d) > \tau_0,
    \qquad {\mathcal T}_{0,II}(e)
    +\beta_{0} {\mathcal T}_{0}(e) <\tau_0.
$$
These conditions entail that the image of $r_0$ through $\Phi_0$ crosses, as a spiral like curve,
at least $\beta_0-\alpha_0$ times both the second, the third and the fourth quadrants.
Then, applying $\Psi_0$ we can move some sub-arcs of $\Phi_0(r_0)$ to the first, or the third, quadrant
when $\lambda \leq 0$ with the technique already described in the previous proofs. As for the
intermediate steps $i\in\{1,\dots,m-1\}$ we should repeat the same arguments as in the
proofs above, producing a sufficiently large number of arcs, like the $\tilde{\gamma}_{m-1}$'s obtained in the proof of Theorem \ref{th6.1}. These arcs lie either in the first, or in the third, quadrant
and cross the  regions ${\mathcal A}_{m,I}$, or ${\mathcal A}_{m,III}$, respectively. A final application of $\Phi_m$ will move these arcs across the annulus ${\mathcal A}_m$ and, according to the position of
$r_L$, we should use a twist condition like \eqref{iv.6}, possibly modified, to make ensure
that the images  $\Phi_m(\tilde{\gamma}_{m-1})$ intersect $r_L$.  An example of how this remark applies will be given in Section \ref{section-4}.
\end{remark}

The next diagram shows all the possible connections, represented by the arrows, for the map $\Psi_i\circ\Phi_i$. The crossing number of these transitions is $\ell_i$, as defined in Proposition \ref{pr4.1}.

\subsection{The case $\l>0$}

This section discusses the case $\lambda>0$. As in this case two
centers are interacting, this is a much more delicate case (see Figure \ref{fig-03}). For this reason,
we will focus our attention into two particular cases; the most useful from the point of view of the applications. Precisely, we will consider separately: (a) the case when the period map of the orbits of
the system $({\mathcal L}_{\lambda})$ presents a certain kind of monotonicity, or, more generally,
there are suitable annular domains with the level lines of ${\mathcal E}$
satisfying a sufficiently strong twist condition on the boundary,
and (b) the case where the center of the system $({\mathcal L}_{\lambda})$
does not produce a sufficiently strong twist, which also covers the case of
isochronous centers. Accordingly, we present the following two results.

\begin{proposition}[Case (a)]
\label{pr6.2}
Given the cyclic clockwise order $IV \prec III \prec II \prec I\prec IV$, fix $\k\in\{I,III\}$ and assume that, for $d, e\in \{c_{1,i},c_{2,i}\}$ with $d\not=e$ and two non-negative integers
$0\leq\alpha_{i} < \beta_{i}$,  the twist condition \eqref{iv.1} is satisfied. Assume, in addition,
that the compatibility condition \eqref{v.10} is satisfied and that there are $\hat{d}>0$ and $\hat{e}>0$,  with   $\hat{d}\not= \hat{e}$, such that
\begin{equation}
\label{vi.5}
    \mc{S}^{c_{2,i\to i+1}}_{\k-1\to \k}(\hat{e}) \leq \varsigma_{i}
    \leq \mc{S}^{c_{1,i\to i+1}}_{\k-1\to \k}(\hat{d}).
\end{equation}
Then
\begin{equation}
\label{vi.6}
    \Psi_i\circ\Phi_i:\mathcal A^{\rm energy}_{i,\k}\overset{\b_i-\a_i}\stretchx\,
    \mathcal A^{\rm energy}_{i+1,\k}.
\end{equation}
Furthermore, for every arc $\gamma\subset\mathcal{A}_{i,\k}^{\rm energy}$ linking the opposite sides of $\mathcal{A}_{i,\k}^{-}$, the solution of $({\mathcal N}_{i,\lambda})$  with initial value
on a sub-arc $\gamma_{j,\k}\subset K_{j,\k}$ of $\gamma$ has exactly, for every $j=1,2,\ldots,\beta_i-\alpha_i$,
\begin{equation*}
    \left\{
    \begin{array}{ll}
    2(\alpha_i+j)-1&\;\hbox{zeroes in the x-component in}\;\;(0,\tau_i)\equiv(t_i,s_i),\\[2pt]
    1&\;\hbox{zero in the x-component in}\;\;[\tau_i,\tau_i+\varsigma_i]\equiv[t_i,\varsigma_i],
    \\[10pt]
    2(\alpha_i+j)&\;\hbox{zeroes in the y-component in}\;\;(0,\tau_i)\equiv(t_i,s_i),\\[2pt]
    0&\;\hbox{zeroes in the y-component in}\;\;[\tau_i,\tau_i+\varsigma_i]\equiv[t_i,\varsigma_i].
    \end{array}
    \right.
\end{equation*}
\end{proposition}

\begin{proof}
Arguing as in Proposition \ref{pr6.1}, from Proposition \ref{pr4.1} we find that there are $\b_i-\a_i$ pairwise disjoint compact sets, $K_{1,II},\dots,K_{\b_i-\a_i,II}$, such that every arc $\gamma$ of
${\mathcal A}_{i,I}$ linking the level lines $\{{\mathcal H}=c_{i,1}\}$ and $\{{\mathcal H}=c_{i,2}\}$
possesses $\b_i-\a_i$ sub-arcs, $\gamma_{j,II}\subset K_{j,II}$, such that $\Phi_i(\gamma_{j,II})$ is an arc
of ${\mathcal A}_{i,II}$ linking the two components of $\mathcal{A}^{-}_{i,II}$ in $\mathcal{A}^{\rm axis}_{i,II}$. To fix ideas, we can suppose that $\hat{d} < \hat{e}$, so that, according to
\eqref{vi.5}, the level line ${\mathcal E}_{\lambda}=\hat{d}$, where the points move at a lower speed, is inside the region enclosed by the level line ${\mathcal E}_{\lambda}=\hat{e}$, where the points move faster.
The case when $\hat{e} < \hat{d}$ can be treated similarly.
\par
For each of the arcs $\Phi_i(\gamma_{j,II})$ we consider  a suitable sub-arc $\Phi_i(\gamma_j)$, where $\gamma_{j}$ is a compact sub-arc of $\gamma_{j,II}$, which intersects, in the second quadrant,
the level line ${\mathcal E}_{\lambda} = \hat{d}$ at a unique point,    $w_1$, and the level line ${\mathcal E}_{\lambda} = \hat{e}$ at a (unique) $w_2$. We can assume that $\Phi_i(\gamma_j)$ is contained in the set
$$
    \{(x,y): \hat{d} \leq {\mathcal E}_{\lambda}(x,y) \leq \hat{e},\, x<0, \,y>0\}.
$$
Observe that the annular region
$$
    {\mathcal E}[\hat{d},\hat{e}]:=\{(x,y): \hat{d} \leq {\mathcal E}(x,y) \leq \hat{e}\}
$$
is invariant under the action of $\Psi_i$, i.e., $(\Psi_i\circ\Phi_i)(\gamma_j)\in {\mathcal E}[\hat{d},\hat{e}]$. Moreover, by the second inequality of \eqref{vi.5}, we know that
$\Psi_i(w_1)$, which is on the level line ${\mathcal E}=\hat{d}$, has an abscissa less or equal than $x_{r,c_{1,i+1}}$, as defined in Section \ref{subsec-l>0}, and it lies in the
region $\{{\mathcal H}\leq c_{1,i+1}\}$. On the other hand, by the first inequality in \eqref{vi.5},
the point    $w_2$ moves under the action of $\Psi_i$ in the clockwise sense along the level line ${\mathcal E}=\hat{e}$, in such a way that $\Psi_i(w_2)$ is beyond the point $(x_{r,c_{2,i+1}},y_{r,c_{2,i+1}})$
of intersection between ${\mathcal E}=\hat{e}$ and ${\mathcal H}=c_{2,i+1}$ in the first quadrant.
More precisely, with respect to a polar coordinates reference system, counting the angles in the
clockwise sense starting at the positive $y$-axis  (see the argument in the proof of Proposition \ref{pr4.1}),  if we denote by $\vartheta(t)$ the angular coordinate associated to the
solution of $({\mathcal L}_{\lambda})$ with initial point at $w_2$, we have that
$-\pi/2<\vartheta(0)<0$ and
$$
   \vartheta(\varsigma_i) \geq \arctan\left(\frac{x_{r,c_{2,i+1}}}{y_{r,c_{2,i+1}}}\right).
$$
Note that $\vartheta(\varsigma_i)$ is the angular coordinate of $\Psi_i(w_2)$. Thus,
using a continuity argument, we can determine a sub-arc $\tilde{\gamma}_j$ of
$\gamma_{j,II}$, contained in $K_{j,II}$, such that $\Phi_i(\tilde{\gamma}_j)$ is a sub-arc
of $\Phi_i(\gamma_j)$ contained in the second quadrant and such that
$\Psi_j(\Phi_i(\tilde{\gamma}_j)$ is an arc lying in the component of the
intersection
$$
  \{c_{1,i+1}\leq{\mathcal H}\leq c_{2,i+1}\}\cap {\mathcal E}[\hat{d},\hat{e}]
$$
contained in the first quadrant.
Moreover, by construction, $(\Psi_j\circ\Phi_i)(\tilde{\gamma}_j)$ connects the
level lines ${\mathcal H}=c_{1,i+1}$ and ${\mathcal H}=c_{2,i+1}$. This proves
\eqref{vi.6}.
\par
Concerning the nodal properties of the solutions with initial value on $\tilde{\gamma}_j$, note that
they are precisely the same as those obtained in case (1.a) of Proposition \ref{pr6.1}. Actually,
as in the previous case, this follows by applying Proposition \ref{pr4.1} and observing
that in the interval    $[0,\varsigma_j]\equiv[s_i,t_{i+1}]$ there is exactly one crossing of the
$y$-axis    with $y>0$.
\end{proof}

\begin{remark}
\label{re6.3}
\rm
One can give a counterpart of Proposition \ref{pr6.2} by replacing the condition \eqref{vi.5} with
\begin{equation}
\label{vi.7}
    \mc{S}^{c_{2,i\to i+1}}_{\k-1\to \k}(\hat{e})+ \xi_i{\mathcal T}_{\lambda}(\hat{e})
    \leq \varsigma_{i}\leq
    \mc{S}^{c_{1,i\to i+1}}_{\k-1\to \k}(\hat{d}) + \zeta_i{\mathcal T}_{\lambda}(\hat{d}),
\end{equation}
for some nonnegative integers $\zeta_i\leq \xi_i$, where we are denoting by
${\mathcal T}_{\lambda}(c)$ the period of the orbit of $({\mathcal L}_{\lambda})$ with
energy ${\mathcal E} = c$.  In this case, instead of \eqref{vi.6}, we find that
\begin{equation}
\label{vi.8}
    \Psi_i\circ\Phi_i:\mathcal A^{\rm energy}_{i,I}\overset{\mathfrak{m}_i(\b_i-\a_i)}\stretchx\,
    \mathcal A^{\rm energy}_{i+1,I} \qquad\text{for } \;
    \mathfrak{m}_i:=\xi_i-\zeta_i+1.
\end{equation}
Moreover, concerning \eqref{vi.8}, for every arc $\gamma\subset {\mathcal A}_{i,I}$ linking
the two components  of ${\mathcal A}_{i,I}^-$, the solution of $({\mathcal N}_{i,\lambda})$
    with initial value on a sub-arc $\gamma_{j,m}\subset K_{j,m}$, for $j=1,\dots,\b_i-\a_i$ and
    $m=0,\dots, \xi_i-\zeta_i$, has exactly
\begin{equation*}
\left\{
\begin{array}{ll}
2(\alpha_i+j)-1&\;\hbox{zeroes in the $x$-component in}\;\;(0,\tau_i)\equiv(t_i,s_i),\\[2pt]
2(\zeta_i+m)+1&\;\hbox{zeroes in the $x$-component in}\;\;[\tau_i,\tau_i+\varsigma_i]\equiv[t_i,\varsigma_i],
\\[10pt]
2(\a_i+j)&\;\hbox{zeroes in the $y$-component in}\;\;(0,\tau_i)\equiv(t_i,s_i),\\[2pt]
2(\zeta_i+m)&\;\hbox{zeroes in the $y$-component in}\;\;[\tau_i,\tau_i+\varsigma_i]\equiv[t_i,\varsigma_i].
\end{array}
\right.
\end{equation*}
\end{remark}

The results established by Proposition \ref{pr6.2} and Remark \ref{re6.3} are closely related
to the theory of \lq\lq Linked Twist Maps\rq\rq \, concerning the composition of two maps which act on
two linked annuli, each map providing a twist effect at the boundary of
an annulus (see Devaney \cite{De-1978}, Wiggins and Ottino \cite{WO-2004}, Margheri et al. \cite{MRZ-2010}, and the references therein).
\par
The next case concerns a situation with two annular domains
linked through an intermediate orbit. The twist condition is required only
on the first of the two composite maps. In the applications, this case
will include the possibility that $({\mathcal L}_{\lambda})$ defines an isochronous center.
This situation appears less investigated
in the literature from the point of view of the theory of dynamical systems. As in the proof of Proposition \ref{pr6.2}, we will denote by $\vartheta_{(x,y)}(t)$ the angular coordinate associated to the solution of $(\mathcal{L}_{\l})$ with initial point $(x,y)$.

\begin{proposition}[Case (b)]
\label{pr6.3}
Given the cyclic clockwise order $IV \prec III \prec II \prec I\prec IV$, fix $\k=\{I,III\}$ and assume that, for $d, e\in \{c_{1,i},c_{2,i}\}$ with $d\not=e$ and two non-negative integers
$0\leq\alpha_{i} < \beta_{i}$, the next twist condition holds
\begin{equation}
\label{vi.9}
{\mathcal T}_{i,\k+1}(d) + \alpha_{i} {\mathcal T}_{i}(d) > \tau_i,
\qquad \beta_{i} {\mathcal T}_{i}(e) <\tau_i.
\end{equation}
Suppose, in addition, that there is $c>0$ satisfying \eqref{v.10} such that
\begin{equation}
\label{vi.10}
    \mc{S}^{c_{2,i\to i+1}}_{\k-1\to \k}(c) < \varsigma_i
    < {\mathcal S}^{c_{1,i\to\sqrt{}}}_{\k-1\to \k}(c),\qquad
    {\mathcal S}_{i,\k+2\to \k}(c) > \varsigma_i,
\end{equation}
and that, for every $(x,y)\in\{\mathcal{H}(x,y)=c_{1,i}\}\cap\{\mathcal{E}(x,y)\leq c\}$,
\begin{equation}
\label{vi.11}
\vartheta_{(x_0,y_0)}(\s_i)<\vartheta_{(x_1,y_1)}(\s_i)\quad\hbox{if}\quad \vartheta_{(x_0,y_0)}(0)<\vartheta_{(x_1,y_1)}(0).
\end{equation}
Then
\begin{equation}
\label{vi.12}
    \Psi_i\circ\Phi_i:\mathcal A^{\rm energy}_{i,\k}\overset{\b_i-\a_i}\stretchx\,
    \mathcal A^{\rm energy}_{i+1,\k}.
\end{equation}
Moreover, concerning \eqref{vi.12}, for every arc $\gamma\subset\mathcal{A}_{i,\k}^{\rm energy}$
linking the opposite sides of $\mathcal{A}_{i,\k}^{-}$, the solution of $({\mathcal N}_{i,\lambda})$
with initial value    on any sub-arc $\gamma_{j,\k}\subset K_{j,\k}$ of $\gamma$ have exactly, for every $j=1,2,\ldots,\beta_i-\alpha_i$,
\begin{equation*}
    \left\{
    \begin{array}{ll}
    2(\alpha_i+j)-1&\;\hbox{zeroes in the x-component in}\;\;(0,\tau_i)\equiv(t_i,s_i),\\[2pt]
    1&\;\hbox{zero in the x-component in}\;\;[\tau_i,\tau_i+\varsigma_i]\equiv[t_i,\varsigma_i],
    \\[10pt]
    2(\alpha_i+j)&\;\hbox{zeroes in the y-component in}\;\;(0,\tau_i)\equiv(t_i,s_i),\\[2pt]
    0&\;\hbox{zeroes in the y-component in}\;\;[\tau_i,\tau_i+\varsigma_i]\equiv[t_i,\varsigma_i].
    \end{array}
    \right.
\end{equation*}
\end{proposition}

Although \eqref{vi.10} and \eqref{vi.11} might look a bit artificial,
they hold when considering the isochronous center associated to the linear system $x'=y$, $y'=-\l x$,
which will be considered in Section \ref{section-4}.

\begin{proof}
We focus attention into the case $\k=I$, since the proof for $\k=III$ follows similar patterns.
Arguing as in Proposition \ref{pr4.1} and assuming \eqref{vi.9},  instead of  \eqref{iv.1},
it becomes apparent that there are $\b_i-\a_i$ pairwise disjoint compact sets $K_1,\dots,K_{\b_i-\a_i}$
such that, for every arc $\gamma$ of    ${\mathcal A}_{i,I}$ linking the level lines
$\{{\mathcal  H}=c_{i,1}\}$ and $\{{\mathcal H}=c_{i,2}\}$, there are $\b_i-\a_i$ sub-arcs
$\gamma_{j,II\cup III}\subset K_{j,II\cup III}$ such that $\Phi_i(\gamma_{j,II\cup III})$ is an arc
of ${\mathcal A}_{i,II}\cup {\mathcal A}_{i,III}$ linking the $y$-axis with $y<0$ to the $y$-axis with $y>0$, passing across the    $x$-axis with $x<0$.
\par
With reference to Figure \ref{fig-03}, for each of the arcs $\Phi_i(\gamma_{j,II\cup III})$ we consider a suitable sub-arc $\Phi_i(\gamma_j)$, where $\gamma_{j}$ is a compact sub-arc of $\gamma_{j,II\cup III}$, which intersects the level line ${\mathcal E}= c$ in the third and the second quadrants
at the (unique) points $w_1$ and $w_2$, respectively. Thus, the arc $\Phi_i(\gamma_j)$ lies in
$$
   \{x<0\}\cap\{{\mathcal E} \leq c\} \cap \{c_{i,1}\leq {\mathcal H}\leq c_{i,2}\}
$$
and it intersects the level line ${\mathcal E} =c$ at the points $w_1$ in the third quadrant and $w_2$ in the second quadrant. According to \eqref{vi.10}, the point $w_2$ moves, by the action of $\Psi_i$,
towards a point $\Psi_i(w_2)$ that lies in the first quadrant, on the level curve ${\mathcal E} =c$,
outside the annulus ${\mathcal A}_{i+1}$. Similarly, also by \eqref{vi.10}, the point $w_1$ travels
to a point $\Psi_i(w_1)$ that lies on the level line ${\mathcal E} =c$. Moreover, the angular component of $\Psi_i(w_1)$ is less than the angle associated to $x_{r,c_{1,i+1}}$. Thus, since, due to \eqref{vi.11}, the image trough $\Psi_i$ of all arcs contained in the set
\[
\{(x,y)\,:\,  x<0,\,c_{1,i}\leq\mathcal{H}(x,y)\leq c_{2,i}\;\hbox{and}\;\mathcal{E}(x,y)\leq c\}
\]
is contained in the first, second or third quadrants, $\Psi_i(\gamma_j)$, which does not cross the origin, must intersect the region ${\mathcal A}_{i+1,I}$.  Therefore, repeating once more the
continuity argument already exploited in the    previous proofs, there must exist a sub-arc, $\tilde\gamma_j$, of $\gamma_j$ such that $\Psi_i(\Phi_i)(\tilde\gamma_j)$ is an arc
of ${\mathcal A}_{i+1,I}$ linking the level lines  ${\mathcal H}=c_{1,i+1}$ and ${\mathcal H}=c_{2,i+1}$.
This shows \eqref{vi.12}. The nodal properties of the corresponding solutions follow easily, with some minor changes,  from the construction of the arcs in the proof of Proposition \ref{pr4.1}. On this occasion one should use \eqref{vi.9}, instead on \eqref{iv.1}.
\end{proof}

The proof of Theorem \ref{th6.1}  can be easily adapted to cover the general case when $\lambda>0$. Actually, by using the same argument one can get the following results. The first one is related to the case
analyzed in Proposition \ref{pr6.2}, while the second one is related to Proposition \ref{pr6.3}.

\begin{theorem}[Case (a)]
\label{th6.3}
Suppose that either $r_0\in {\mathcal A}_{0,I}$, or    $r_0\in {\mathcal A}_{0,III}$,  and either $r_L\in {\mathcal A}_{m,II}$, or  $r_L\in {\mathcal A}_{m,IV}$, and that, for every $i\in\{0,\dots, m\}$
and $d, e\in \{c_{1,i},c_{2,i}\}$ with $d\not=e$, there exists a pair of integers
$0\leq \alpha_{i} < \beta_{i}$ satisfying the twist conditions \eqref{iv.1} for $\k=\{I,III\}$. Assume, in addition, that \eqref{vi.5} holds    for all $i\in\{0,\dots, m-1\}$. Then, the same
conclusions of Theorem \ref{th6.1}  hold.
\end{theorem}

\begin{theorem}
[Case (b)]
\label{th6.4}
Suppose that either $r_0\in {\mathcal A}_{0,I}$, or $r_0\in {\mathcal A}_{0,III}$, and either $r_L\in {\mathcal A}_{m,II}$, or $r_L\in {\mathcal A}_{m,IV}$, and that, for every $i\in\{0,\dots, m\}$ and
$d, e\in \{c_{1,i},c_{2,i}\}$ with $d\not=e$, there are two integers $0\leq \alpha_{i} < \beta_{i}$
satisfying the twist conditions \eqref{vi.9} for $\k\in\{I,III\}$. Assume also that \eqref{vi.10} and \eqref{vi.11} hold true for all  $i\in\{0,\dots, m-1\}$. Then,
the same conclusions of Theorem \ref{th6.1}  hold.
\end{theorem}

The proof of Theorems \ref{th6.3} and \ref{th6.4} is omitted, by repetitive,
as they only require a direct application of the argument already described in the
proof of Theorem \ref{th6.1}. Actually, they use, once more, all the admissible connections
described in the diagrams of Figure \ref{fig-04}, with the sole difference
that the transition properties of the map $\Psi_i$ are those collected by Proposition \ref{pr6.2}
in the proof of Theorem \ref{th6.3} and by Proposition \ref{pr6.3} for Theorem \ref{th6.4}.

\begin{remark}
\label{re6.4}
\rm
According to these findings, the number of solutions obtained in Theorems \ref{th6.3} and  \ref{th6.4}
can be further expanded by a factor of the form $\prod_{i=0}^{m-1}\mathfrak{m}_i$ by imposing
a stronger twist condition of the form of \eqref{vi.7} to the maps $\Psi_i$'s,  as already
observed in Remark \ref{re6.3}, where it was already explained how to find solutions satisfying
prescribed nodal properties in  $(t_i,s_i)$ and  $(s_i,t_{i+1})$.
\end{remark}

\section{An application}
\label{section-4}

In this section we apply the abstract theory developed in Section 6 to the generalized Sturm--Liouville problem
\begin{equation}
\label{vii.1}
\begin{cases}
-u''= \lambda u + a(t)|u|^{p-1}u\\
(u(0),u'(0))\in \hat{r}_0\\
(u(L),u'(L))\in \hat{r}_L\\
\end{cases}
\end{equation}
where $\hat{r}_0$ and $\hat{r}_L$ are unbounded arcs linking the origin to infinity
and contained in a single quadrant, $a(t)$ satisfies the general assumptions at the beginning of Section 3, and $p>0$, $p\neq 1$.
Thus, we are making the special choices
\begin{equation}
\label{vii.2}
  h(y)=y,\qquad g(x)= |x|^{p-1}x,
\end{equation}
in \eqref{vi.1}.
A phase-plane analysis of the equation in \eqref{vii.1} has been performed in
\cite{LGTZ-2014, Te-2018} in connection with the search of positive solutions for a sign-indefinite weight.
Even if our approach is also related to plane analysis techniques, the different conditions on the weight
produce a complete different geometry.
Since \eqref{ii.5} can be further simplified depending on whether $\l=0$ or not, these two possible choices for $\lambda$ will be discussed separately.

\subsection{The case $\l=0$}

In this case, since
$$
x_+(c)=\left(\frac{(p+1)c}{\mu_i}\right)^{\frac{1}{p+1}}\quad\hbox{and}\quad \int_{0}^{1}
\frac{ds}{\sqrt{1- s^{p+1}}} = \frac{1}{p+1}B(\tfrac{1}{p+1},\tfrac{1}{2}),
$$
where $B$ is the Euler beta-function, the period \eqref{ii.5} can be expressed as
\begin{equation}
\label{vii.3}
\frac{{\mathcal T}_{i}(x_+)}{4} = \mu_i^{-\frac{1}{p+1}}c^{\frac{1-p}{2(p+1)}} C(p),
\qquad C(p)\equiv \frac{B(\tfrac{1}{p+1},\tfrac{1}{2})}{\sqrt{2} (p+1)^{\frac{p}{p+1}}}.
\end{equation}
\eqref{vii.3} provides us with a precise estimate on the growth of the period as a function of $c$.
Thus, for each step $i\in\{0,\dots,m-1\}$, we can construct annular regions ${\mathcal A}_i$ whose
boundaries are the energy levels $c_{1,i}$ and $c_{2,i}>c_{1,i}$ for which
the twist conditions imposed in Section 6.1 are satisfied.
\par
Next, we estimate the transition times \eqref{v.7} and \eqref{v.8}, which coincide by the oddity of $h$ and $g$. It turns out that
\begin{equation}
\label{vii.4}
S_{i,II\to I} = S_{i,IV\to III} <
\frac{ \left(\frac{(p+1)c_{2,i}}{\mu_i}\right)^{\frac{1}{p+1}}+
    \left(\frac{(p+1)c_{2,i+1}}{\mu_{i+1}}\right)^{\frac{1}{p+1}} }
{\sqrt{2\min\{c_{1,i},c_{1,i+1}\}}}.
\end{equation}
The next result is a byproduct of Theorem \ref{th6.1}. It provides us with a multiplicity result for the generalized Sturm--Liouville boundary value problem \eqref{vii.1} for $\l=0$.

\begin{theorem}
\label{th7.1}
Assume that, for every $i\in\{0,\dots, m\}$, there are $d_i>0$ and $e_i>0$, with $d_i\not= e_i$,
satisfying the following twist conditions
\begin{equation}
\label{vii.5}
\begin{split}
    (\alpha_i + \tfrac{1}{2} ) {\mathcal T}_{i}(x_+(d_i)) & > {\tau_i}
    > (\beta_i + \tfrac{1}{2} ) {\mathcal T}_{i}(x_+(e_i)), \quad i\in\{1, \dots, m-1\},\\
    ( \alpha_i + \tfrac{1}{4} ){\mathcal T}_{i}(x_+(d_i)) & > {\tau_i}
    > ( \beta_i + \tfrac{1}{2} ) {\mathcal T}_{i}(x_+(e_i)), \quad i=0, m,
\end{split}
\end{equation}
for some integers $0\leq \alpha_i< \beta_i$.  Suppose, in addition, that, for every $i\in\{0,\dots, m-1\}$,
\begin{equation}
\label{vii.6}
    \varsigma_i \geq     \frac{
        \left(\frac{(p+1)c_{2,i}}{\mu_i}\right)^{\frac{1}{p+1}}+
        \left(\frac{(p+1)c_{2,i+1}}{\mu_{i+1}}\right)^{\frac{1}{p+1}}
    }
    {\sqrt{2\min\{c_{1,i},c_{1,i+1}\}}}
\end{equation}
holds, where $c_{1,j}:= \min\{d_j,e_j\}$ and $c_{2,j}:= \max\{d_j,e_j\}$.
Then, the
problem \eqref{vii.1} has, at least, $2^{m}\prod_{i=0}^{m}(\b_i-\a_i)$ solutions,
and each of these solutions can be classified by assigning it  a counter describing
its nodal properties on each of the intervals
$(t_i,s_i)$ and $(s_i,t_{i+1})$, according to its itinerary in the Markov-type diagram of Figure \ref{fig-04}
(right panel).
\end{theorem}
\begin{proof}
Fix a sequence ${\mathcal A}_{i}$ of annular regions, for $i\in\{0,\dots,m\}$, with $c_{1,i}:= \min\{d_i,e_i\}$ and $c_{2,i}:= \max\{d_i,e_i\}$, where $d_i$ and $e_i$ are chosen to satisfy \eqref{vii.5}.
Then, the twist condition \eqref{iv.1} holds for $\k_i\in\{I,III\}$. Moreover, by \eqref{vii.4},
the condition $S_{i,II\to I} = S_{i,IV\to III}\leq \varsigma_i$  follows from \eqref{vii.6}.  Finally, for every $\k_0,\k_m\in\{IV,III,II,I\}$, let $r_0$ be a component of the intersection of the curve
$\hat{r}_0$ with ${\mathcal A}_{0,\k_0}$, and  $r_L$ a component of the intersection of $\hat{r}_L$ with ${\mathcal A}_{m,k_m}$. Then, by Theorem \ref{th6.1} and Remark \ref{re6.2}, for every choice of  $r_0$ and $r_L$, the problem \eqref{vii.1} possesses, at least, $2^{m}\prod_{i=0}^{m}(\b_i-\a_i)$ solutions.
\end{proof}

\begin{remark}
\label{re7.1} \rm
Choose an initial set ${\mathcal A}_{0,\k_0}$ and a final set ${\mathcal A}_{m,\k_m}$ corresponding
to the quadrants of $\hat{r}_0$ and $\hat{r}_L$, respectively, as well as an arbitrary itinerary from
${\mathcal A}_{1,\k_1}$ to ${\mathcal A}_{m-1,\k_{m-1}}$, where, at each step $i\in\{1,\dots,m-1\}$,
we can choose $\k_i\in\{I,III\}$. Finally, for every $i\in\{0,\dots,m\}$, let $j_i\in \{1,\dots\ell_i\}$ be an arbitrary integer. Then, by  Proposition \ref{pr6.1},  Theorem \ref{th7.1} guarantees the existence of at least one solution $u(t)$ to problem \eqref{vii.1} with $(u(0),u'(0))\in {\mathcal A}_{0,k_0}$ and
$(u(L),u'(L))\in {\mathcal A}_{m,k_m}$ such that, for each $i\in\{1,\dots,m-1\}$, $u(t)$ has exactly:
\begin{equation*}
    \left\{
    \begin{array}{llll}
    \left.
    \begin{array}{ll}
    2(\alpha_i+j_i)-1&\;\hbox{zeroes in}\;\;(t_i,s_i)\\
    1&\;\hbox{zero with $u'(t)\neq 0$ in}\;\;(s_i,t_{i+1})
    \end{array}
    \right\}
    \;\hbox{if}\;\;\k_i=\k_{i+1}
    \\[10pt]
    \left.
    \begin{array}{ll}
    2(\alpha_i+j_i)&\qquad\hbox{zeroes in}\;\;(t_i,s_i)\\
    1&\qquad\hbox{zero with $u'(t)\neq 0$ in}\;\;(s_i,t_{i+1})
    \end{array}
    \right\}
    \;\hbox{if}\;\;\k_i\neq\k_{i+1}
    \end{array}
    \right.
\end{equation*}
Moreover, $u(t)$ and $u'(t)$ have a unique zero in $(s_0,t_1)$ (see (1.a) in the proof of Proposition \ref{pr6.1}), the number of zeroes of $u(t)$ in $[t_0,s_1)$ depends on $\alpha_0$, $j_0$ and the position of $\hat{r}_0$, while in the last interval $(t_m,s_m]=(t_m,L]$ depends on ${\mathcal A}_{m-1,\k_{m-1}}$, the position of $\hat{r}_L$ and the choice of $\alpha_m$ and $j_m$.
\end{remark}

\begin{remark}
\label{re7.2}    \rm
By the oddity of the nonlinearity, under homogeneous boundary conditions, either Dirichlet, or Neumann, or mixed type, the solutions of \eqref{vii.1} arise by pairs: $(\l,u)$ and $(\l,-u)$. To describe the internal
structure of these solutions, we focus our attention into the Dirichlet problem, where we can take $\hat{r}_{0}$ and $\hat{r}_{L}$ as the positive and the negative $y$-axis, respectively. In this case, by Theorem \ref{th7.1},
there are, at least,  $2^{m}\prod_{i=0}^m(\b_i-\a_i)$ solutions of
\begin{equation}
\label{vii.7}
\begin{cases}
-u''= a(t)|u|^{p-1}u,\\
u(0)=u(L)=0.
\end{cases}
\end{equation}
Moreover, $u'(0)>0$ and $u'(L)>0$ if $\hat{r}_{0}=\hat{r}_{L}$ is the positive $y$-axis, $u'(0)>0$ and $u'(L)<0$ if $\hat{r}_{0}=-\hat{r}_{L}$ is the positive $y$-axis, $u'(0)<0$ and $u'(L)>0$ if $\hat{r}_{0}=-\hat{r}_{L}$ is the negative $y$-axis, and $u'(0)<0$ and $u'(L)<0$ if $\hat{r}_{0}=\hat{r}_{L}$ is the negative $y$-axis. An analogous classification has been given by Kajikiya \cite{Ka-2021} for the classical model of Moore--Nehari \cite{MN-1959}, where $a(t)$ has two positive humps.
\par
Note that, for some special boundary conditions, like, e.g., the Dirichlet problem, the twist assumption can be slightly sharpened at the levels $i=0$ and $i=m$.  For instance,
Theorem \ref{th6.1} can be applied for the Dirichlet problem imposing the first restrictions of \eqref{vii.5} for all $i\in\{0,\dots,m\}$,  because
the half-line $\{(0,y):y>0\}$ can be though as belonging either to the first, or the second, quadrant, as well as the half-line $\{(0,y):y<0\}$, which belongs both to the third and the fourth quadrants.
\end{remark}

Subsequently, we will apply Theorem \ref{th7.1} to the special case when $a(t)$ satisfies
\begin{equation}
\label{vii.8}
\mu_i=\mu >0, \;\; \tau_i=\tau> 0, \;\; \varsigma_i=\varsigma >0, \quad\hbox{for all}\;\; i\in\{0,\dots,m\},
\end{equation}
like in Moore and Nehari \cite{MN-1959}. We will fix  $\alpha_i = 0$ and $\beta_i=\beta\geq 1$, where $\beta$ is a constant.  In this case, the most natural choice is to take the
same annular region at each step and hence, we also take constants $d_i= d > 0$ and $e_i=e\equiv \theta
d> 0$. Under these assumptions, by \eqref{vii.3}, condition \eqref{vii.5} becomes
\begin{equation}
\label{vii.9}
d^{\frac{1-p}{2(p+1)}}C(p) > \tau \mu^{\frac{1}{p+1}} > 2(2\beta
+1) \theta^{\frac{1-p}{2(p+1)}}d^{\frac{1-p}{2(p+1)}}C(p).
\end{equation}
Suppose $p>1$. Then, $\frac{1-p}{2(p+1)}<0$ and hence, for every $\mu>0$ and $\tau>0$, there exist $d_0>0$ and $e_0>d_0$ such that \eqref{vii.9} holds for all $d\in (0,d_0]$ and $\t\geq \t_0:= e_0/d_0$. Similarly, when $p\in (0,1)$, since $\frac{1-p}{2(p+1)}>0$, there are $0<e_0<d_0$ such that \eqref{vii.9} holds for all $d\geq d_0$ and $\t \in (0,\t_0]$, where $\t_0:= e_0/d_0$. On the other hand, the conditions \eqref{vii.6} become
\begin{equation}
\label{vii.10}
\varsigma \mu^{\frac{1}{p+1}} \geq \left\{ \begin{array}{ll}
\sqrt{2}(p+1)^{\frac{1}{p+1}} \theta^{\frac{1}{p+1}} d^{\frac{1-p}{2(p+1)}} &
\qquad \text{if }\; \; p >1,\\[3pt] \sqrt{2}(p+1)^{\frac{1}{p+1}}
\theta^{-\frac{1}{2}} d^{\frac{1-p}{2(p+1)}} & \qquad \text{if }\;\; 0
< p <1.\end{array}\right.
\end{equation}
Based on \eqref{vii.9} and \eqref{vii.10}, one can get some estimates on the ratio
$\varsigma/\tau$ in order to get an arbitrarily large number of nodal solutions for the problem
\eqref{vii.7}. Indeed, suppose  $p>1$. Then, there exists $d_0>0$ such that
\begin{equation}
\label{vii.11}
  d^{\frac{1-p}{2(p+1)}}C(p) > d_0^{\frac{1-p}{2(p+1)}}C(p)= \tau \mu^{\frac{1}{p+1}}\qquad
  \hbox{for all} \;\; d<d_0.
\end{equation}
Thus, for $d=d_0$, the second inequality of \eqref{vii.9} holds if, and only if,
$2(2\beta +1) \theta^{\frac{1-p}{2(p+1)}} < 1$, i.e., if $\theta=e/d_0 > (2(2\beta
+1))^{\frac{2(p+1)}{p-1}}$. Moreover, due to \eqref{vii.11}, \eqref{vii.10} holds if, and only if,
$$
   \varsigma \mu^\frac{1}{p+1} =\frac{\varsigma}{\tau}\tau \mu^\frac{1}{p+1} =
   \frac{\varsigma}{\tau} d_0^{\frac{1-p}{2(p+1)}}C(p)\geq \sqrt{2}(p+1)^{\frac{1}{p+1}} \theta^{\frac{1}{p+1}} d_0^{\frac{1-p}{2(p+1)}}
$$
or, equivalently, taking into account \eqref{vii.3},
\begin{equation}
\label{vii.12}
\frac{\varsigma}{\tau} \geq \Lambda(p,e) := \frac{ \sqrt{2}(p+1)^{\frac{1}{p+1}}
\t^{\frac{1}{p+1}}}{C(p)}= \frac{2(p+1)\t^{\frac{1}{p+1}}}
{B(\tfrac{1}{p+1},\tfrac{1}{2})}.
\end{equation}
Now, suppose that $0 < p < 1$. Then, since $\frac{1-p}{2(p+1)}>0$, there exists $d_0>0$ such that
\begin{equation}
\label{vii.13}
  d^{\frac{1-p}{2(p+1)}}C(p) > d_0^{\frac{1-p}{2(p+1)}}C(p)= \tau \mu^{\frac{1}{p+1}}\qquad
  \hbox{for all} \;\; d>d_0.
\end{equation}
Thus, for $d=d_0$, the second inequality of \eqref{vii.9} holds if and only if $\theta=e/d_0 < (2(2\beta+1))^{\frac{2(p+1)}{p-1}}$.
Moreover, thanks to \eqref{vii.13}, \eqref{vii.10} holds if and only if
$$
   \varsigma \mu^\frac{1}{p+1} =\frac{\varsigma}{\tau}\tau \mu^\frac{1}{p+1} =
   \frac{\varsigma}{\tau} d_0^{\frac{1-p}{2(p+1)}}C(p)\geq \sqrt{2}(p+1)^{\frac{1}{p+1}} \theta^{-\frac{1}{2}} d_0^{\frac{1-p}{2(p+1)}},
$$
or, equivalently,
\begin{equation}
\label{vii.14}
\frac{\varsigma}{\tau} > \Lambda(p,e):=\frac{2(p+1)\t^{-\frac{1}{2}}}
{B(\tfrac{1}{p+1},\tfrac{1}{2})}.
\end{equation}
Note that \eqref{vii.12} holds for sufficiently small $\t$, whereas \eqref{vii.14} holds for
sufficiently large $\t$. Therefore, as a direct consequence from Theorem \ref{th6.1} the following result holds.

\begin{corollary}
\label{co7.1} Suppose $\l=0$,  $p>0$ with $p\not=1$,  \eqref{vii.8}, and
\eqref{vii.12} if $p>1$, or \eqref{vii.14} if $p\in (0,1)$. Then, for any positive integer $\beta \geq 1$
such that $(2(2\beta+1))^{\frac{2(p+1)}{p-1}}<\theta$ if $p>1$, or $(2(2\beta+1))^{\frac{2(p+1)}{p-1}}>\theta$ if $p\in (0,1)$, the problem \eqref{vii.7} has, at least, $2^m\beta^{m+1}$ solutions.  Moreover, these solutions can be classified according to its
nodal properties on each of the intervals $(t_i,s_i)$ and $(s_i,t_{i+1})$,  depending on its itinerary in the Markov-type of diagram of
Figure \ref{fig-04} (right panel).
\end{corollary}

In applying Corollary \ref{co7.1} to \eqref{vii.7} under assumptions \eqref{vii.8}, one can take as the initial half-line $\hat{r}_0$ the positive or the negative $y$-axis, and the same choice can be made
for the final line $\hat{r}_L$. Thus, at the end of the day, Corollary \ref{co7.1} provides us with $4\times 2^m\beta^{m+1} = 2(2\beta)^{m+1}$ solutions. As already discussed in Remark \ref{re7.2},
these solutions can be classified into four different classes. For each of these classes, the nodal properties of the solutions follow the rules described in Remark \ref{re7.1}.
\par
We stress the fact that Corollary \ref{co7.1} allows us to consider $\alpha_i=0$
for all $i\in\{0,\dots,m\}$. Thus, we can obtain solutions with prescribed nodal properties
on each of the intervals $(0,s_0)$, $(s_0,t_1)$ and $(t_m,L)$, while in the internal intervals the nodal behavior obeys the patterns described by Remark \ref{re7.1} according to the
itinerary ${\mathcal A}_{1,\k_1},\dots, {\mathcal A}_{m-1,\k_{m-1}}$ with $\kappa_i\in\{I,III\}$ for
all $i\in\{1,\dots,m-1\}$. The same general patterns are respected by the Neumann and the Robin problem,
as well as for general mixed boundary value problems.

\subsection{The case $\l>0$}
When $\l\neq 0$, we will restrict ourselves to consider the simpler situation case when the annular regions
${\mathcal A}_i$ are given by
\begin{equation}
\label{vii.15}
{\mathcal A}:= \{(x,y): c_1 \leq
\tfrac{1}{2} y^2 +\tfrac{\lambda}{2}x^2 + \tfrac{\mu}{p+1} x^{p+1} \leq c_2\},
\end{equation}
for a suitable choice of $0 < c_1 < c_2$.  Moreover, we take $\mu_i= \mu >0$ for all $i\in\{0,...,m\}$.
In the case $\l>0$ applies Theorem \ref{th6.4}. Thus, one needs to check the conditions \eqref{vi.10} and \eqref{vi.11}. According to \eqref{v.1} and \eqref{v.9}, we have that
$y_{+,i}= \sqrt{2c_1}$ for all $i$ and $H_{*}= c_1$. Thus, by choosing $c= c_1$, the condition
\eqref{v.11} holds provided that $x_+(c_2) < \sqrt{2 c_1/\lambda}$, i.e., if
\begin{equation}
\label{vii.16}
\sqrt{2c_1/\lambda} >
\left(\frac{(p+1)(c_2-c_1)}{\mu}\right)^{\frac{1}{p+1}}\equiv x_{\rm com}.
\end{equation}
For any given $\lambda > 0$, \eqref{vii.16} is always
satisfied for sufficiently small $c_2-c_1$, i.e., the annulus $\mc{A}$ is sufficiently narrow,
or sufficiently large $\mu$. Alternatively, for any given $\mc{A}$ and $\mu>0$, \eqref{vii.16}
holds for sufficiently small $\l>0$. Thus, \eqref{vii.16} somehow expresses the fact that the larger is $\l>0$ the poorer is the structure of the set of nodal solutions of the problem
\begin{equation}
\label{vii.17}
\begin{cases}
-u''= \l u + a(t)|u|^{p-1}u,\\
u(0)=u(L)=0,
\end{cases}
\end{equation}
as, whenever $\l$ crosses $\varsigma_\kappa$ for some $\kappa\geq 1$, \eqref{vii.17} looses its solutions with $\kappa-1$ interior nodes.
\par
With the previous choices, we have that
\begin{align*}
& |x_{\ell,c_{2,i}}| = |x_{\ell,c_{2,i+1}}| = x_{r,c_{2,i}} =
x_{r,c_{2,i+1}} = x_{\rm com},\\
& S^{c_{2,i\to i+1}}_{II\to I}(c)=
S^{c_{2,i\to i+1}}_{IV\to III}(c) = 2\int_{0}^{x_{\rm com}}\frac{dx}{\sqrt{2c_1 -\lambda x^2}},
\end{align*}
for all $i\in\{0,...,m-1\}$.
Moreover, we have also that
\begin{align*}
 & x_{\ell,c_{1,i}} = x_{\ell,c_{1,i+1}} = x_{r,c_{1,i}} =
x_{r,c_{1,i+1}} = 0,\\
& S^{c_{1,i\to i+1}}_{II\to I}(c)=
S^{c_{1,i\to i+1}}_{IV\to III}(c) = 0,\\
& S^{c_{1,i}\to\sqrt{}}_{II\to I}(c)=
S^{c_{1,i}\to\sqrt{}}_{IV\to III}(c) =
\int_{0}^{\sqrt{2c_1/\lambda}}\frac{dx}{\sqrt{2c_1 -\lambda x^2}}
=\frac{\pi}{2\sqrt{\lambda}}.
\end{align*}
Finally, by the symmetry of the nonlinearity of \eqref{vii.17}, we also find that
$$
\frac{\pi}{2\sqrt{\lambda}}<S_{III\to I}(c)=
    S_{I\to III}(c)
    <\frac{\pi}{\sqrt{\lambda}}\,.
$$
The next result is a byproduct of Theorem \ref{th6.4} and Remark \ref{re6.2}.

\begin{corollary}
\label{co7.2}
Suppose $\l>0$, $p>1$ and $\tau_i= s_i-t_i=\tau>0,$
$\varsigma_i=t_{i+1}-s_i=\varsigma>0$ and $\mu_i=\mu >0$, for all $i\in\{0,\dots,m\}$. Assume, in addition, that
there exist $0 < c_1 < c_2$ and two integers $0\leq \alpha < \beta$ such that
\begin{equation}
\label{vii.18}
    (\alpha + \tfrac{1}{2} ) {\mathcal T}(c_1) > {\tau}
    > (\beta + \tfrac{1}{2} ) {\mathcal T}(c_2), \quad \text{ for }\; i=1, \dots, m-1,
\end{equation}
\begin{equation}
\label{vii.19}
    ( \alpha + \tfrac{1}{4} ){\mathcal T}(c_1) > {\tau}
    > ( \beta + \tfrac{1}{2} ) {\mathcal T}(c_2), \quad \text{ for }\; i=0,m,
\end{equation}
where ${\mathcal T}(c) := {\mathcal T}_{i}(c)$ as in \eqref{ii.5}. Lastly, assume that
the compatibility condition \eqref{vii.16} holds together with
\begin{equation}
\label{vii.20}
    2\int_{0}^{x_{\rm com}}\frac{dx}{\sqrt{2c_1 -\lambda x^2}}
    +\frac{2\pi\xi}{\sqrt{\lambda}}\leq\varsigma
    \leq \frac{\pi}{2\sqrt{\lambda}} +\frac{2\pi\xi}{\sqrt{\lambda}},
\end{equation}
for some nonnegative integer $\xi.$  Then, the problem \eqref{vii.1} has at least
$2^{m}(\b-\a)^{m+1} $ solutions. Moreover, these solutions can be classified according to its
nodal properties on each of the intervals $(t_i,s_i)$ and $(s_i,t_{i+1})$,  depending on its itinerary in the Markov-type of diagram of
Figure \ref{fig-04} (right panel).
\end{corollary}
\begin{proof}
It is a direct application of Theorem \ref{th6.4}. More precisely, by \eqref{ii.5}, \eqref{vii.18}
implies the validity of the twist condition \eqref{vi.9} for $i\in\{1,\dots,m-1\}$, with $d=c_1$ and
$e=c_2$. Moreover, \eqref{vii.19} provides us with \eqref{iv.6} for $i=0, m$, with the same choices of
$d$ and $e$. Thus, by  Remark \ref{re6.2}, $\hat{r}_0$ and $\hat{r}_L$ do not need to be fixed. On the other hand, \eqref{vii.20} implies \eqref{vi.10}, while \eqref{vi.10} holds from the second inequality of \eqref{vii.20}, by the symmetry of the equation. Finally, the monotonicity of the angular function
\eqref{vi.1} can be checked by a direct argument on the equation $-{\vartheta}'=\lambda \cos^2({\vartheta}) + \sin^2({\vartheta})$,  which comes from $({\mathcal L}_{\lambda})$ in polar coordinates.
Therefore, all the hypotheses of Theorem \ref{th6.4} are    satisfied.
\end{proof}

\begin{remark}
\label{re7.3}
\rm With some obvious modifications in the proof, a similar result holds
for the    sublinear case when $0< p <1$. One should take into account that,
for the autonomous system in the intervals where $a(t)\equiv \mu$, the smaller
orbits are faster than the larger ones. Accordingly, the twist conditions
\eqref{vii.18} and \eqref{vii.19} should be modified to
\begin{align*}
    & (\alpha + \tfrac{1}{2} ) {\mathcal T}(c_2) > {\tau}
    > (\beta + \tfrac{1}{2} ) {\mathcal T}(c_1), \quad \text{ for }\; i=1, \dots, m-1,\\
    &   ( \alpha + \tfrac{1}{4} ){\mathcal T}(c_2) > {\tau}
    > ( \beta + \tfrac{1}{2} ) {\mathcal T}(c_1), \quad \text{ for }\; i=0,m,
\end{align*}
 respectively. The rest of the statement remains unchanged.
\end{remark}

\begin{remark}\label{re7.4}
\rm Concerning the nodal properties of the solutions, Remark \ref{re7.1} applies. Thus,
the solutions oscillate $\alpha + j_i$ times in each interval $(t_i,s_i)$, for $j_i=1,\dots,\b_i-\a_i$.
On the other hand, in the intervals $(s_i,t_{i+1})$ the system is linear with an isochronous center
and, therefore, the condition \eqref{vii.20} guarantees the existence of $2\xi +1$ zeroes for $u(t)$
in $(s_i,t_{i+1})$.
\end{remark}

As the assumption \eqref{vii.20} seems delicate, as it imposes two
constraints on the length $\varsigma$ of the components of $a^{-1}(0)$, to
apply Corollary \ref{co7.2} to some concrete example, we need to
find out the solution $x_+(c)$ in \eqref{ii.4} as well as the period ${\mathcal T}(x_+(c))$.
Once fixed the parameters $c$, $\lambda$ and $\mu_i=\mu$, some numerical estimates of
$x_+(c)$ and $T(x_+(c))$ can be given.

\begin{example}
\label{ex-4.1}\rm
Consider \eqref{vii.1} with $\lambda=1$ and $p=3$.  As in Corollary
\ref{co7.2} we assume that $\tau_i= s_i-t_i=\tau>0$,  $\varsigma_i=t_{i+1}-s_i=\varsigma>0$ and $\mu_i=\mu >0$, for all $i\in\{0,\dots,m\}$.  As a first choice, we take $\mu= 20$, $c_1=1$ and $c_2=5$.  Then,
$x_{\rm comp}\thickapprox 0.94 < \sqrt{2}= (2c_1/\lambda)^{1/2}$. Thus,
\eqref{vii.16} is satisfied. Moreover,
$$
  2\int_{0}^{x_{\rm comp}}\frac{dx}{\sqrt{2c_1 -\lambda x^2}}
    \thickapprox 1.46 < \frac{\pi}{2\sqrt{\lambda}}    =\frac{\pi}{2},
$$
$x_+^2(c_1)= \frac{2}{5}$, and $x_+^2(c_2) = \frac{\sqrt{401} -1}{20} \thickapprox 0.95$. Thus,
${\mathcal T}(x_+(c_1)) \thickapprox 2.41$ and ${\mathcal T}(x_+(c_2)) \thickapprox 1.63$. Hence,
${\mathcal T}(x_+(c_1))> {\mathcal T}(x_+(c_2))$, though  ${\mathcal T}(x_+(c_1))\not> 2{\mathcal T}(x_+(c_2))$. Therefore, for a suitable choice of $\alpha$, $\beta$, $\tau$ and $\varsigma$ we can apply Corollary \ref{co7.2} with \eqref{vii.18} satisfied for $i\in\{0,\dots,m\}$ and
$\hat{r}_{0}\in {\mathcal A}_I$, or $\hat{r}_{0}\in {\mathcal A}_{III}$,  and
$\hat{r}_{L}\in {\mathcal A}_{II}$, or $\hat{r}_{L}\in {\mathcal A}_{IV}$. However, one cannot
apply Remark \eqref{re6.2}    because \eqref{vii.19} is not satisfied.
\par
As a second choice, we take $\mu= 130$, $c_1=0.8$ and $c_2=20$.  Then,
$x_{\rm comp}\thickapprox 0.87 <  \sqrt{2c_1/\lambda} =\sqrt{2}$, and hence \eqref{vii.16} is satisfied. Moreover,
$$
   2\int_{0}^{x_{\rm comp}}\frac{dx}{\sqrt{2c_1 -\lambda x^2}}
    \thickapprox 1.53 < \frac{\pi}{2\sqrt{\lambda}}
    =\frac{\pi}{2},
$$
$x_+^2(c_1) \thickapprox 0.15$ and $x_+^2(c_2) \thickapprox 0.77$. Thus,
${\mathcal T}(x_+(c_1))) \thickapprox 1.62$ and ${\mathcal T}(x_+(c_2)) \thickapprox 0.73$.
Hence, in this case, ${\mathcal T}(x_+(c_1))> 2{\mathcal T}(x_+(c_2))$ and Corollary \ref{co7.2}
can be fully applied for a suitable choice of  $\alpha$, $\beta$, $\tau$ and $\varsigma$.
\end{example}

\subsection{The case $\l<0$}

Since we are considering $\mc{A}_i=\mc{A}$ as defined in \eqref{vii.15}, $\mu_i=\mu$ and $h$ and $g$ are odd, it follows that the abscissas and the ordinates calculated for the transition times in Section \ref{sec-lin} coincide for every $i\in\{0,...,m\}$. Thus,  $|x_{\ell,i}|=|x_{\ell}|=x_{r}=x_{r,i}$,  $|x_{-,i}|=|x_{-}|=x_{+,i}=x_{+}$, and $|y_{-,i}|=|y_{-}|=y_{+,i}=y_{+}$. Hence, by \eqref{v.2} and \eqref{v.4}, the transition times are
\begin{equation*}
{\mc S}_{II\to I} = 2\int_{0}^{x_{r}}
\frac{ds}{\sqrt{y_{+}^2 + |\lambda|s^2}}= 2\int_{0}^{x_{r}}
\frac{ds}{\sqrt{y_{-}^2 + |\lambda|s^2}}={\mc S}_{IV\to III}.
\end{equation*}
Similarly, \eqref{v.5} and \eqref{v.6}, we have that
\begin{equation*}
{\mc S}_{IV\to I} = 2\int_{x_{+}}^{x_{r}}
\frac{ds}{\sqrt{|\lambda|(s^2- x_{+}^2)}}=2\int_{|x_{-}|}^{|x_{\ell}|}
\frac{ds}{\sqrt{|\lambda|(s^2- x_{-}^2)}}={\mc S}_{II\to III}.
\end{equation*}
As in the case $\l>0$ we should verify conditions \eqref{vii.18} and \eqref{vii.19} to apply Theorem \ref{th6.1}  and Remark \ref{re6.2}. In this case, more favorable than the case $\l>0$, we can consider $\a_i=0$ for all $i\in\{0,\ldots,m\}$ to obtain any possible choice of nodal properties.

\begin{corollary}
\label{co7.3}
Suppose $\lambda<0$, $p>0$ with $p\not=1$, and $\tau_i= s_i-t_i=\tau>0$ and
$\varsigma_i=t_{i+1}-s_i=\varsigma>0$ for all $i\in\{0,\dots,m\}$. Then, for any positive integer $\beta \geq 1,$ there exists a constant  $\Lambda(p,\beta,\tau) >0$ such that \eqref{vii.1} has at least
$4^{m}\beta^{m+1}$ solutions for every  $\varsigma > \Lambda(\lambda,p,\beta,\tau)$.
 Moreover, these solutions can be classified according to its
nodal properties on each of the intervals $(t_i,s_i)$ and $(s_i,t_{i+1})$,  depending on its itinerary in the Markov-type of diagram of
Figure \ref{fig-04} (right panel).
\end{corollary}
\begin{proof}
In this proof, we set $\theta=\theta(x):=x/x^*$ with $x^*$ as in \eqref{ii.7}, and
\begin{equation}
\label{vii.21}
    {\mathscr T}_1(\theta):=\int_{0}^{1}
    \frac{ds}{\sqrt{-(1-s^2) + \theta^{p-1}(1- s^{p+1})}}.
\end{equation}
Our interest in ${\mathscr T}_1(\theta)$ coming from the identity
\begin{equation}
\label{vii.22}
    {\mathscr T}_1(\theta(x_+(c)))=\frac{\mc{T}(x_+(c))}{4}\sqrt{|\l|}
\end{equation}
where $\mc{T}(x_+(c))$ is given by \eqref{ii.5}.
The map $s\in[0,1] \mapsto q(s):=(1-s^{p+1})/(1-s^2)$ is increasing if $p>1$ and decreasing if $0<p<1$.
Thus, $q(1)=(p+1)/2$ provides us with its maximum if $p>1$, and its minimum if $p\in (0,1)$. Hence,
whenever $p>1$,  the denominator in the integrand of ${\mathscr T}_1(\theta)$ is positive for all
$s\in (0,1)$ if $\theta >1$. Moreover, ${\mathscr T}_1(\theta)$ is decreasing in $\t\in (1,+\infty)$ and
$$
   \lim_{\theta\da 1} {\mathscr T}_1(\theta)\to +\infty, \quad
   \lim_{\theta\ua +\infty} {\mathscr T}_1(\theta)=0.
$$
Whenever $0< p <1$ the positivity of the denominator occurs for
\begin{equation*}
    0< \theta < \theta^*:=\left(\frac{p+1}{2}\right)^{\frac{1}{1-p}}=\frac{\rm \o_+}{x^*}.
\end{equation*}
In this case,  ${\mathscr T}_1(\theta)$ is increasing in $(0,\theta^*)$ and
$$
   \lim_{\theta\da 0} {\mathscr T}_1(\theta)= 0,\qquad
     \lim_{\theta\ua  \theta^*} {\mathscr T}_1(\theta)= +\infty.
$$
Suppose $p>1$. Then, the parameters $c_1$ and $c_2$ defining ${\mathcal A}$ c an be chosen as
follows. First, we fix two constants $x_1 = \theta_1 x^*$ and $x_2 = \theta_2 x^*$,  with  $1<\theta_1 <\theta_2$, such that
\begin{equation}
\label{vii.23}
    {\mathscr T}_1(\theta_1) > \tau \sqrt{|\lambda|} > \frac{\tau \sqrt{|\lambda|}}{2(2\beta +1)}>{\mathscr T}_1(\theta_2).
\end{equation}
By the properties of ${\mathscr T}_1$, for any given $\tau$, the estimates \eqref{vii.23} are satisfied
for $\theta_1$ sufficiently close to $1$ and sufficiently large $\theta_2$. Thus, setting
$$
  c_1:= {\mathcal H}(x_1,0)>0,\qquad c_2:= {\mathcal H}_{\lambda}(x_2,0)>0,
$$
\eqref{vii.23} becomes
$$
    \mc{T}(c_1)>4\tau>2(2\b+1)\mc{T}(c_2).
$$
Thus, the stronger form of the twist condition \eqref{vii.19} holds. This implies \eqref{vii.18}. Hence, we are within the setting of Remark \ref{re6.2}, so that the Sturm--Liouville lines $\hat{r}_0$ and $\hat{r}_L$
can be chosen in any quadrant. Once chosen $x_1$, $x_2$, $c_1$ and $c_2$, we can estimate the transition times ${\mc S}$. As we have that  $x_+=x_1$, $y_+=\sqrt{2c_1}$ and $x_r \leq x_2$, it is apparent that
\begin{align*}
    S_{II\to I} & = S_{IV\to III} \leq 2\int_{0}^{x_2}
    \frac{ds}{\sqrt{y_+^2 + |\lambda|s^2}}
    = \frac{2}{\sqrt{|\lambda|}}{\rm arcsinh}(x_2\sqrt{|\lambda|}/y_+)=:\frac{2}{\sqrt{|\lambda|}}\Lambda_1,\\
    S_{IV\to I} & = S_{II\to III} \leq 2 \int_{x_1}^{x_2}
    \frac{ds}{\sqrt{|\lambda|(s^2- x_1^2)}} =
    \frac{2}{\sqrt{|\lambda|}}{\rm arccosh}(\theta_2/\theta_1)=:\frac{2}{\sqrt{|\lambda|}} \Lambda_2.
\end{align*}
$\Lambda_1$ can be as well expressed in terms of the parameters $\theta_1$ and $\theta_2$ as follows
\begin{equation*}
    \Lambda_1=
    {\rm arcsinh}\left(\frac{\theta_2}{\theta_1 ( \theta_1^{p-1} -1)^{1/2}}\right).
\end{equation*}
In particular, Theorem \ref{th6.2} can be applied provided
\begin{equation}
\label{vii.24}
    \Lambda^*:=\frac{2}{\sqrt{|\lambda|}} \max\{\Lambda_1,\Lambda_2\}\leq\varsigma.
\end{equation}
This shows the corollary with $\Lambda(\lambda,p,\beta,\tau)= \Lambda^*$. Observe that $\Lambda_1$
and $\Lambda_2$ depend  on the coefficients $\theta_1$ and $\theta_2$, and that these
coefficients are closely related to the values of $\tau\sqrt{|\lambda|}$ and $\beta$.
\par
The same argument can be adapted to deal with the case when $0 < p < 1$, except for the crucial difference
that small orbits around the origin have a small period while large orbits approximating
the heteroclinics have arbitrarily large periods. Suppose $p\in (0,1)$. Then, the values of $c_1$ and $c_2$ defining ${\mathcal A}$ are chosen as follows. First, we choose two constants $\theta_1=x_1/x^*$ and $\theta_2 =x_2/x^*$ such that $0 < \theta_1<\theta_2 < \theta^*$ and
\begin{equation}
\label{vii.25}
    {\mathscr T}_1(\theta_2) > \tau \sqrt{|\lambda|} > \frac{\tau \sqrt{|\lambda|}}{2(2\beta +1)}>{\mathscr T}_1(\theta_1).
\end{equation}
For any given $\tau$, \eqref{vii.25} holds by taking $\theta_2$ sufficiently close to $\theta^*$, i.e., $x_2$ sufficiently close to $\o_+$, and sufficiently small $\theta_1$. In this case, by \eqref{vii.22} and fixing $c_1:= {\mathcal H}(x_1,0)>0$ and $0<c_2:= {\mathcal H}(x_2,0)< C^*$ (see \eqref{ii.8}), it follows that
$$
    \mc{T}(c_1)>4\tau>2(2\b+1)\mc{T}(c_2),
$$
satisfying the stronger form of the twist condition \eqref{vii.19}, which implies \eqref{vii.18} and, thus, entering again in the setting of Remark \ref{re6.2}. Once fixed  $x_1$, $x_2$, $c_1$ and $c_2$,  the computation of the coefficients for the  integrals involving the transition times ${\mc S}$ follows
the same patterns as when $p>1$, and hence it is omitted here. The result follows for sufficiently large $\s$ satisfying \eqref{vii.24}.
\end{proof}

\begin{remark}
\rm Regarding the nodal properties of the solutions, the analysis of the transitions $\mc{A}_{i,II}\rightarrow \mc{A}_{i,I}$ and $\mc{A}_{i,IV}\rightarrow \mc{A}_{i,III}$ follows identical
patterns as in Remark \ref{re7.1}, while the analysis of $\mc{A}_{i,IV}\rightarrow \mc{A}_{i,I}$ and $\mc{A}_{i,II}\rightarrow \mc{A}_{i,III}$  follows the scheme of Case 2 of
Proposition \ref{pr6.1}.
\end{remark}

The next examples illustrate this situation, by considering
an application of Corollary \ref{co7.3} for each of the cases $p>1$ and $0<p<1$.

\begin{example}
\label{ex-4.2} \rm
Consider \eqref{vii.1} with $\lambda<0$ and $p=3$. As in Corollary
\ref{co7.3} we take $a(t)$ with $\tau_i= s_i-t_i=\tau>0$, $\varsigma_i=t_{i+1}-s_i=\varsigma>0$ and
$\mu_i=\mu >0$ for all $i\in\{0,\dots,m\}$. Then, we have that $x^* = \sqrt{2|\lambda|/\mu}$. To find out
some range of values of the constants for which Corollary \ref{co7.3} can be applied,
let us suppose that $\tau\sqrt{|\lambda|}=1$ and take $\theta_1=3/2$. Then, ${\mathscr T}_1(\theta_1) \thickapprox 1.07 > \tau\sqrt{|\lambda|}$. Next, we need to choose $\theta_2>\theta_1$ such that
${\mathscr T}_1(\theta_2) < \frac{\tau\sqrt{|\lambda|}}{2(2\beta +1)}$. For instance, take $\beta=2$
 and $\theta_2=14$. Then, ${\mathscr T}_1(\theta_2) \thickapprox 0.09 < \tau\sqrt{|\lambda|}/10$.
Moreover, $(\theta_1^{p-1} - 1)^{-1/2}=\sqrt{4/5} \thickapprox 0.89$, $a_2/a_1 = 28/3 \thickapprox 9.33$.
Therefore, $\Lambda_1 \thickapprox 2.82$, $\Lambda_2 \thickapprox 2.92$ and
$\Lambda^*\thickapprox 2.92$. Consequently, as soon as
$\varsigma\sqrt{|\lambda|} \geq 2\Lambda^* \thickapprox 5.85$, Corollary \ref{co7.3} applies with $\beta=2$.
\end{example}

%---Bibliography
%\nocite{*}
\bibliographystyle{plain}
\bibliography{LG-MH-Z_4-biblio}

\end{document}